\documentclass[11pt]{amsart}
\usepackage{fullpage, amsmath, amsthm,amsfonts,amssymb, mathrsfs}
\usepackage{hyperref}

\usepackage[pdftex]{color}
\usepackage[all]{xy}

\newtheorem{theorem}[subsection]{Theorem}
\newtheorem{lemma}[subsection]{Lemma}
\newtheorem{corollary}[subsection]{Corollary}
\newtheorem{conjecture}[subsection]{Conjecture}

\newtheorem{proposition}[subsection]{Proposition}
\theoremstyle{definition}

\newtheorem{definition}[subsection]{Definition}
\newtheorem{hypothesis}[subsection]{Hypothesis}
\newtheorem{example}[subsection]{Example}
\newtheorem{remark}[subsection]{Remark}

\newtheorem{notation}[subsection]{Notation}

\newcommand{\yichao}[1]{{\color{red} \sf $\heartsuit\heartsuit\heartsuit$ Yichao: [#1]}}

\numberwithin{equation}{subsection}

\def\calA{\mathcal{A}}
\def\calB{\mathcal{B}}

\def\calD{\mathcal{D}}
\def\calE{\mathcal{E}}
\def\calH{\mathcal{H}}
\def\calL{\mathcal{L}}
\def\calM{\mathcal{M}}
\def\calO{\mathcal{O}}

\def\calQ{\mathcal{Q}}
\def\calS{\mathcal{S}}
\def\calV{\mathcal{V}}

\def\gothe{\mathfrak{e}}

\def\gothp{\mathfrak{p}}

\def\AAA{\mathbb{A}}

\def\CC{\mathbb{C}}
\def\FF{\mathbb{F}}
\def\GG{\mathbb{G}}
\def\LL{\mathbb{L}}
\def\NN{\mathbb{N}}
\def\PP{\mathbb{P}}
\def\QQ{\mathbb{Q}}
\def\RR{\mathbb{R}}

\def\ZZ{\mathbb{Z}}

\def\bfe{\mathbf{e}}
\def\bff{\mathbf{f}}

\def\bfj{\mathbf{j}}

\def\rmM{\mathrm{M}}

\def\scrA{\mathscr{A}}
\def\scrC{\mathscr{C}}
\def\scrH{\mathscr{H}}

\DeclareMathOperator{\End}{End}
\DeclareMathOperator{\Gal}{Gal}
\DeclareMathOperator{\Hom}{Hom}
\DeclareMathOperator{\Aut}{Aut}

\DeclareMathOperator{\Res}{Res}
\DeclareMathOperator{\Spec}{Spec}

\newcommand{\cHom}{\calH om}
\newcommand{\cris}{\mathrm{cris}}
\newcommand{\cSh}{\mathcal{S}h}
\newcommand{\dR}{\mathrm{dR}}

\newcommand{\et}{\mathrm{et}}

\newcommand{\Frob}{\mathrm{Frob}}
\newcommand{\GL}{\mathrm{GL}}
\newcommand{\id}{\mathrm{id}}
\newcommand{\Image}{\mathrm{Im}}
\newcommand{\Ker}{\mathrm{Ker}}
\newcommand{\Lie}{\mathrm{Lie}}

\newcommand{\pr}{\mathrm{pr}}
\newcommand{\Qp}{\QQ_p}

\newcommand{\Sh}{\mathrm{Sh}}

\newcommand{\Zp}{\ZZ_p}

\newcommand{\coker}{\mathrm{Coker}}

\newcommand{\diag}{\mathrm{diag}}
\newcommand{\cl}{\mathrm{cl}}
\newcommand{\fin}{\mathrm{fin}}

\newcommand{\JL}{\mathcal{J}\!\mathcal{L}}

\newcommand{\cH}{\mathcal{H}}

\newcommand{\cO}{\mathcal{O}}
\newcommand{\F}{\mathbb{F}}
\newcommand{\Z}{\mathbb{Z}}
\newcommand{\Q}{\mathbb{Q}}

\newcommand{\Ql}{\overline{\mathbb{Q}}_{\ell}}
\newcommand{\Qlb}{\overline{\QQ}_{\ell}}
\newcommand{\Tr}{\mathrm{Tr}}
\newcommand{\univ}{\mathrm{univ}}

\DeclareMathOperator{\rank}{rank}

\newcommand{\Diag}{\mathrm{Diag}}
\newcommand{\Fpb}{\overline{\mathbb{F}}_p}
\newcommand{\cB}{\mathcal{B}}
\newcommand{\cD}{\mathcal{D}}
\newcommand{\tcD}{\tilde\calD}
\newcommand{\tcE}{\tilde\calE}
\newcommand{\Def}{\mathscr{D}\mathrm{ef}}
\newcommand{\uY}{\underline{Y}}
\newcommand{\Gr}{\mathbf{Gr}}
\newcommand{\im}{\mathrm{Im}}
\newcommand{\Isog}{\mathrm{Isog}}

\newcommand{\Gys}{\mathrm{Gys}}

\newcommand{\len}{\mathrm{length}}

\newcommand{\loc}{\mathrm{loc}}

\newcommand{\hra}{\hookrightarrow}
\newcommand{\ra}{\rightarrow}

\newcommand{\xra}{\xrightarrow}

\begin{document}

\title{Tate cycles on some unitary Shimura varieties mod $p$}
\author{David Helm, Yichao Tian, and Liang Xiao}

\begin{abstract}
Let $F$ be a real quadratic field in which a fixed prime $p$ is inert, and  $E_0$ be an imaginary quadratic field in which $p$ splits; put $E = E_0F$.   
Let $X$ be  the  fiber over $\FF_{p^2}$ of the Shimura variety for $G(U(1,n-1) \times U(n-1,1))$ with hyperspecial level structure at $p$ for some integer $n\geq 2$.
We show that, under some genericity conditions,
the middle-dimensional Tate classes of $X$ are generated by 
the irreducible components of its supersingular locus.
We also discuss a general conjecture regarding special cycles on the special fibers of unitary Shimura varieties, and on their relation to Newton stratification.
\end{abstract}


\maketitle

\setcounter{tocdepth}{1}
\tableofcontents

\section{Introduction}
The study of the geometry of Shimura varieties lies at the heart of the Langlands program.  Arithmetic information of Shimura varieties builds a bridge relating the world of automorphic representations and the world of Galois representations.

One of the interesting topics in this area is to understand the supersingular locus of the special fibers of Shimura varieties, or more generally, any interesting stratifications (e.g. Newton or Ekedahl-Oort stratification) of the special fibers of  Shimura varieties. 
 The case of unitary Shimura varieties has been extensively studied.  
In \cite{vollaard-wedhorn}, Vollaard and Wedhorn showed that the supersingular locus of the special fiber of the $GU(1,n-1)$-Shimura variety at an inert prime is a union of Deligne-Lusztig varieties. 
Further, Howard and Pappas \cite{howard-pappas} studied the case of $GU(2,2)$ at an inert prime, and Rapoport, Terstiege and Wilson proved similar results for    $GU(n-1,1)$ at a ramified prime. 
Finally, we remark that G\"ortz and He \cite{goertz-he} studied the basic loci in a slightly more general class of Shimura varieties. 

In all the work mentioned above, the authors use the uniformization theorem  of Rapoport--Zink to reduce the problem to the study of certain Rapoport--Zink spaces. 
In this paper, we take a different approach. Instead of using the uniformization theorem, we study the basic locus (or more generally other Newton strata) of certain unitary Shimura varieties by considering 
 correspondences between unitary Shimura varieties of different signatures. 
  This  method was introduced by the first author in  \cite{helm-PEL, helm}, and applied  successfully  to  quaternionic Shimura varieties by the second and the third authors \cite{tian-xiao1}.
  
  Another new aspect of this work is that, we study not only  the global geometry of the supersingular locus, but also their relationship  with the Tate Conjecture for Shimura varieties over finite fields. 
  We show that the basic locus  contributes to  all ``generic'' middle dimensional Tate cycles of the special fiber of the Shimura variety. Similar results have been obtained by the second and the third authors for even dimensional Hilbert modular varieties at an inert prime \cite{tian-xiao2}.
We believe that, this phenomenon is a general philosophy which holds for more general Shimura varieties.  Our slogan is: {\it irreducible components of the basic locus of a Shimura variety should generate all Tate classes under some genericity condition on the automorphic representations.}\\

We explain in more details the main results of this paper.
Let $F$ be a real quadratic field,  $E_0$ be an imaginary quadratic field, and $E=E_0F$.
 Let $p$ be a prime number inert in $F$, and split in $E_0$. 
 Let $ \gothp$ and $\bar \gothp$ denote the two places of $E$ above $p$ so that $E_\gothp$ and $ E_{\bar \gothp}$ are both isomorphic to $\QQ_{p^2}$, the unique unramified quadratic  extension of $\QQ_{p}$.
For an integer $n\geq 1$, let $G$ be the similitude unitary group associated to a division algebra over $E$ equipped with an involution of second kind. In the notation of Subsection~\ref{S:notation-real-quadratic}, our  $G$ is  denoted as $G_{1,n-1}$.
This is an algebraic group over $\QQ$ such  that $G(\QQ_p)\simeq \QQ_p^{\times }\times \GL_n(E_{\gothp})$ and $G(\RR)$ is the unitary similitude group with signature $(1,n-1)$ and $(n-1,1)$ at the two  archimedean places. (For a precise definition, see Section~\ref{S:Shimura data}.)

Let $\AAA$  denote the ring of finite adeles of $\QQ$, and $\AAA^{\infty}$ be its finite part.
 Fix a sufficiently small open compact subgroup  $K\subseteq G(\AAA^{\infty})$  with $K_p=\ZZ_p^{\times}\times \GL_n(\ZZ_{p^{2}}) \subseteq G(\QQ_p)$, where $\ZZ_{p^2}$ is the ring of integers of $\QQ_{p^2}$.
Let $\calS h(G)_{K}$ be the Shimura variety associated to $G$ of level $K$.\footnote{Strictly speaking, the moduli space $\calS h(G)_K$ is $\#\ker^1(\QQ,G)$-copies of the classical Shimura variety whose $\CC$-points are given by the double coset space $G(\QQ)\backslash G(\AAA)/K_{\infty} K$, where $K_{\infty}\subseteq G(\RR)$ is the maximal compact subgroup modulo center. See \cite[page 400]{kottwitz} for details.}
  
According to Kottwitz \cite{kottwitz},  when $K^p$ is neat,  $\calS h(G)_{K}$ admits a proper and smooth integral model over $\ZZ_{p^2}$ which parametrizes certain polarized abelian schemes with $K$-level structure (See Subsection~\ref{S:defn of Shimura var}).  
Let $\Sh_{1,n-1}$ denote the special fiber of $\calS h(G)_{K}$ over $\FF_{p^2}$. This is a proper smooth variety over $\FF_{p^2}$ of dimension $2(n-1)$.
Let $\Sh_{1,n-1}^{\mathrm{ss}}$ denote the supersingular locus of $\Sh_{1,n-1}$, i.e. the reduced closed subvariety of $\Sh_{1,n-1}$ that parametrizes supersingular abelian varieties. 
We will see in Proposition~\ref{P:normal-bundle} that $\Sh_{1,n-1}^{\mathrm{ss}}$ is equidimensional of dimension $n-1$.

Fix a prime $\ell\neq p$. The  $\ell$-adic \'etale  cohomology group 
$H^{2(n-1)}_{\et}\big(\Sh_{1,n-1,\Fpb},\Qlb(n-1)\big)$ is equipped with a natural action by $\Gal(\overline \FF_p/\FF_{p^2})\times \Qlb[K\backslash G(\AAA^{\infty})/K]$. 
We will take advantage of the Hecke action to  consider a variant of the Tate conjecture for $\Sh_{1,n-1}$.

Fix an irreducible admissible  representation $\pi$ of $G(\AAA^{\infty})$ (with coefficients in $\Qlb$). 
The $K$-invariant  subspace of $\pi$, denoted by $\pi^K$, is  a finite dimensional  irreducible representation of the Hecke algebra $\Qlb[K\backslash G(\AAA^{\infty})/K]$. 
We denote by 
  $H^{2(n-1)}_{\et}\big(\Sh_{1,n-1,\Fpb},\Qlb(n-1)\big)_{\pi}$ the $\pi^{K}$-isotypic component of $H^{2(n-1)}_\et\big(\Sh_{1,n-1,\Fpb},\Qlb(n-1)\big)$, and put 
\[
H^{2(n-1)}_\et\big(\Sh_{1,n-1,\Fpb},\Qlb(n-1)\big)_{\pi}^{\fin}:=\bigcup_{\FF_{q}/\FF_{p^2}} H^{2(n-1)}_\et\big(\Sh_{1,n-1,\Fpb},\Qlb(n-1)\big)_{\pi}^{\Gal(\Fpb/\FF_q)},
\] 
where $\FF_q$ runs through all finite extensions of $\FF_{p^2}$.
By projecting to the $\pi^K$-isotypic component, we have an $\ell$-adic cycle class map:
\begin{equation}\label{E:cycle-map}
\cl^{n-1}_{\pi}: A^{n-1}(\Sh_{1,n-1,\Fpb})\otimes_{\ZZ}\Qlb \longrightarrow H^{2(n-1)}_{\et}\big(\Sh_{1,n-1,\Fpb},\Qlb(n-1)\big)_{\pi}^{\fin},
\end{equation}
where $ A^{n-1}(\Sh_{1,n-1,\Fpb})$ is the abelian group of codimension $n-1$ algebraic cycles on $\Sh_{1,n-1,\Fpb}$. 
Then the Tate conjecture for $\Sh_{1,n-1}$ predicts that the above map is surjective. Our main result confirms exactly this statement under some ``genericity'' assumptions on $\pi$.

From now on, we  assume that $\pi$ satisfies Hypothesis~\ref{H:automorphic assumption} to ensure the non-triviality of the $\pi$-isotypic component of the cohomology groups.
 In particular, $\pi$ is the finite part of an automorphic cuspidal representation of $G(\AAA)$, and $H^{2(n-1)}_{\et}\big(\Sh_{1,n-1,\Fpb},\Qlb(n-1)\big)_{\pi}\neq 0$. 
Let  $\pi_p$ denote the $p$-component of $\pi$, which is an unramified principal series as $K_p$ is hyperspecial. Since $G(\QQ_p)\simeq \QQ_p^{\times}\times \GL_n(E_{\gothp})$, we  write $\pi_p=\pi_{p,0}\otimes \pi_{\gothp}$, where $\pi_{p,0}$ is a character of $\QQ_p^{\times}$ and $\pi_\gothp$ is an irreducible admissible representation of $\GL_n(E_\gothp)$.


Our main theorem is the following.
\begin{theorem}\label{T:main-result-introduction}
Suppose that $\pi$ is the finite part of an automorphic representation of $G(\AAA)$ that admits a cuspidal base change to $\GL_n(\AAA_E)\times\AAA_{E_0}^{\times}$, and 
  that the Satake parameters of $\pi_{\gothp}$ are distinct modulo roots of unity. 
 Then  $H^{2(n-1)}_{\et}\big(\Sh_{1,n-1,\Fpb},\Qlb(n-1)\big)_{\pi}^{\fin}$ is generated by the cohomological classes of the irreducible components of the supersingular locus  $\Sh_{1,n-1}^{\mathrm{ss}}$. In particular, the cycle class map \eqref{E:cycle-map} is surjective. \end{theorem}

This theorem will be restated in a more precise form in Theorem~\ref{T:main-theorem}. 
Here, the assumption that the Satake parameters of $\pi_{\gothp}$ are distinct modulo roots of unity is crucial for our method. It is closely tied to our geometric description of the irreducible components.  This condition will be reformulated  in Theorem~\ref{T:main-theorem} in terms of the Frobenius eigenvalues of certain Galois representation attached to $\pi_{\gothp}$ via the unramified local Langlands correspondence.
The other automorphic assumption on $\pi$ is of technical nature. 
It is imposed here to ensure certain equalities on the automorphic multiplicity on $\pi$ (See Remark~\ref{R:automorphic-mult}). 
The method of our paper may be extended to more general representations $\pi$, if we have  more knowledge of the multiplicity of automorphic forms on unitary groups.

What we will  prove is more precise than stated in Theorem~\ref{T:main-result-introduction}. 
We need another unitary group $G'=G_{0,n}$ over $\QQ$ for $E/F$ as in Lemma~\ref{L:change-signature}, which is the unique inner form of $G$ such that  $G'(\AAA^{\infty})\simeq G(\AAA^{\infty})$ and   the signatures of $G'$ at the two archimedean places are $(0,n)$ and  $(n,0)$. 
Let $\Sh_{0,n}$ denote the (zero-dimensional) Shimura variety over $\FF_{p^2}$ associated to $G'$. 
We will show  in Proposition~\ref{P:normal-bundle} that the supersingular locus $\Sh_{1,n-1}^{\mathrm{ss}}$ is a union of $n$ closed subvarieties $Y_j$ with $1\leq j\leq n$ such that each  of $Y_j$ admits a fibration   over $\Sh_{0,n}$ of the same level $K\subseteq G(\AAA^{\infty})\simeq G'(\AAA^{\infty})$ with fibers isomorphic to  a certain proper and smooth closed subvariety in a product of Grassmannians. 
 In other words, each $Y_j$ is an algebraic correspondence between $\Sh_{1,n-1}$ and $\Sh_{0,n}$:
 \[
 \Sh_{0,n} \longleftarrow Y_j \longrightarrow \Sh_{1,n-1}.
 \] This can be viewed as a geometric realization of the  Jacquet--Langlands correspondence  between $G$ and $G'$ in the sense of \cite{helm-PEL}. 
Alternatively, we may view these $Y_j$ as Hecke correspondences between special fibers of unitary Shimura varieties of different signatures.
To prove Theorem~\ref{T:main-result-introduction},
it suffices to show that, when the Satake parameters of $\pi_{\gothp}$ are distinct modulo roots of unity,   $H^{2(n-1)}_{\et}\big(\Sh_{1,n-1,\Fpb},\Qlb(n-1)\big)_{\pi}^{\fin}$ is generated by the cohomology classes of the irreducible components of $Y_j$. 
The key point is to show that the $\pi$-projection of  the intersection matrix of $Y_j$ is non-degenerate under the assumption above on $\pi_{\gothp}$.\\

We briefly describe the structure of this paper. In Section~\ref{Sec:intro}, we consider a more general setup of unitary Shimura varieties, and propose a general conjecture, which roughly predicts the existence of a certain algebraic correspondences between the special fibers of  Shimura varieties  with hyperspecial level at $p$ associated to unitary groups with different signatures at infinity (Conjecture~\ref{Conj:main}). 
Theorem~\ref{T:main-result-introduction} is a special case of Conjecture~\ref{Conj:main}. 
We believe that our conjecture will provide a  new perspective to understand of the special fibers of Shimura varieties.
In Section~\ref{S:preliminary}, we review some Dieudonn\'e theory and Grothendieck--Messing deformation theory that will be frequently used in later sections. Section~\ref{Section:U(1,n)} is devoted to the study of the supersingular locus $\Sh_{1,n-1}^{\mathrm{ss}}$, and construct the subvarieties $Y_j$ mentioned above. 
In Section~\ref{S:fundamental-intersection-number}, we compute certain intersection number on  products of Grassmannian varieties. 
These numbers will play a fundamental role in our later computation of the intersection matrix of the $Y_j$. 
In Section~\ref{Sect:intersection-matrix}, we will compute explicitly  the intersection matrix of the $Y_j$'s (Theorem~\ref{Th-Conj}), and show that its $\pi$-isotypic projection of the intersection matrix is non-degenerate as long as the Satake parameters of $\pi_{\gothp}$ are distinct (as opposed to being distinct modulo roots of unity).
 Then an easy cohomological  computation  allows us to conclude the proof of our main theorem. 
In Section~\ref{Sec:GU rs}, we will generalize the construction of the cycles $Y_j$ to the Shimura variety associated to unitary group for $E/F$ of signature $(r,s)\times (s,r)$ at infinity. 
In this case, we only obtain some partial results on these cycles predicted   by Conjecture~\ref{Conj:main}: the union of these cycles is exactly the supersingular locus of the unitary Shimura variety in question (Theorem~\ref{T:Zj supersingular locus}).
%



\subsection*{Acknowledgments}
We thank Xinwen Zhu for sharing his insights into Conjecture~\ref{Conj:main} and allowing us to include them in this current paper.
We also thank Ga\"etan Chenevier,
Matthew Emerton, Robert Kottwitz,   Sug-Woo Shin and Jacques Tilouine
for extremely helpful discussions.
We thank the anonymous referees for their careful reading of the paper and helpful comments.
We thank Fields Institute for hosting the thematic program, during which this collaboration started.
D.H. thanks University of Chicago and University of California at Irvine for hospitality when he visited.
Y.T. thanks the hospitality of Institut de Hautes \'Etudes Scientifiques where part of this work was written.
Y.T. thanks University of California at Irvine and University of Connecticut at Storrs for hospitality when he visited.
D.H. was partially supported by NSF grant DMS--1161582.
Y.T.  was supported  in part by the National Natural Science Foundation of China (No. 11321101).
L.X. was partially supported by Simons Collaboration Grant \#278433, CORCL grant from University of California, Irvine, and NSF Grant DMS--1502147.

\section{The conjecture on special cycles}
\label{Sec:intro}
We will only discuss certain unitary Shimura varieties so that the description becomes explicit. We will discuss after Conjecture~\ref{Conj:main} on how to possibly extend this conjecture to more general Shimura varieties.

\subsection{Notation}
\label{S:division algebra}
We fix a prime number $p$ throughout this paper.
We fix an isomorphism $\iota_{p}: \CC \xra{\sim} \overline \QQ_p$.
Let $\QQ_p^\mathrm{ur}$ be the maximal unramified extension of $\QQ_p$ inside $\overline \QQ_p$.

Let $F$ be a totally real field of degree $f$ in which $p$ is inert.
We label all real embeddings of $F$, or equivalently (via $\iota_p$), all $p$-adic embeddings of $F$ (into $\QQ_p^\mathrm{ur}$) by $\tau_1, \dots, \tau_f$ so that post-composition by the Frobenius map takes $\tau_i $ to $\tau_{i+1}$. Here the subindices are taken modulo $f$.
Let $E_0$ be an imaginary quadratic extension of $\Q$, in which $p$ splits.   Put $E=E_0F$. Denote by $v$ and $\bar v$ the two $p$-adic places of $E_0$.
Then  $p$ splits into two primes $\gothp$ and $\bar \gothp$ in $E$, where $\gothp$ (resp. $\bar \gothp$) is the $p$-adic place above $v$ (resp. $\bar v$).
Let $q_i$ denote the embedding $E \to E_\gothp \cong F_p \xrightarrow{\tau_i} \overline \QQ_p$ and $\bar q_i$ the analogous embedding which factors through $E_{\bar \gothp}$ instead. Composing with $\iota^{-1}_p$, we regard $q_i$ and $\bar q_i$ as complex embeddings of $E$, and  we put $\Sigma_{\infty,E}=\{q_1,\dots, q_f, \bar q_1,\dots,\bar q_f\}.$

\subsection{Shimura data}
\label{S:Shimura data}
Let $D$ be a division algebra of dimension $n^2$ over its center $E$, equipped with a positive involution $*$ which restricts to the complex conjugation $c$ on $E$. In particular, $D^{\mathrm{opp}}\cong D\otimes_{E, c}E$. 
We assume that $D$ splits at $\gothp$ and $\bar \gothp$, and we fix an isomorphism
$$
D \otimes_\QQ \QQ_p \simeq \rmM_n(E_{\gothp}) \times\rmM_n(E_{\bar\gothp})\cong \rmM_{n}(\QQ_{p^f})\times \rmM_n(\QQ_{p^{f}}),
$$
where  $*$ switches the two direct factors.  We use $\gothe$ to denote the element of $D \otimes_\QQ \QQ_p$ corresponding to the $(1,1)$-elementary matrix\footnote{By a $(1,1)$-elementary matrix, we mean an $n\times n$-matrix whose $(1,1)$-entry is $1$ and whose other entries are zero.} in the first factor.
Let  $a_{\bullet}=(a_i)_{1\leq i\leq f}$ be a tuple of $f$ numbers with  $a_i\in \{0,\dots, n\}$. Assume that there exists an element $\beta_{a_{\bullet}}\in (D^\times)^{*=-1}$ such that the following condition is  satisfied:\footnote{As explained in the proof of \cite[Lemma~I.7.1]{harris-taylor}, when $n$ is odd, such $\beta_{a_\bullet}$ always exists, and when $n$ is even, existence of $\beta_{a_\bullet}$ depends on the parity of $a_1+\cdots +a_f$. See also the proof of Lemma~\ref{L:change-signature}.}

\begin{itemize}
\item Let $G_{a_{\bullet}}$ be  the algebraic group over $\Q$ such that $G_{a_{\bullet}}(R)$ for a $\QQ$-algebra $R$ consists of  elements $g\in (D^{\mathrm{opp}} \otimes_\QQ R)^{\times}$ with  $g\beta_{a_{\bullet}} g^*=c(g) \beta_{a_{\bullet}}$ for some $c(g)\in R^{\times}$. If $G_{a_{\bullet}}^1$ denotes the kernel of the similitude character  $c: G_{a_{\bullet}}\ra \GG_{m,\Q}$, then there exists an isomorphism
\[
G^1_{a_{\bullet}}(\RR)\simeq \prod_{i=1}^f U(a_i,n-a_i),
\]
where the $i$-th factor corresponds to the real  embedding $\tau_i:F\hra \RR$.

\end{itemize}
Note that the assumption on $D$ at $p$ implies that
\[
G_{a_{\bullet}}(\Q_p)\simeq \Q_p^{\times} \times \GL_n(E_{\gothp})\cong \Q_p^{\times}\times \GL_n(\Q_{p^f}).
\]
We put  $V_{a_{\bullet}}=D$ and view it as a left $D$-module.
Let $\langle-,-\rangle_{a_{\bullet}}: V_{a_{\bullet}}\times V_{a_{\bullet}}\ra \Q$ be the perfect alternating pairing  given by
\[
\langle x,y\rangle_{a_{\bullet}}=\Tr_{D/\Q}(x\beta_{a_{\bullet}} y^*),\quad \text{for  }x,y\in V_{a_{\bullet}}.
\]
Then $G_{a_{\bullet}}$ is identified with the similitude group associated to $(V_{a_{\bullet}},\langle-,-\rangle_{a_{\bullet}})$, i.e. for all $\Q$-algebra $R$, we have
\[
G_{a_{\bullet}}(R)=\big\{g\in \End_{D\otimes_{\Q} R}(V_{a_{\bullet}}\otimes_{\Q} R)\;\big|\; \langle gx,gy\rangle_{a_{\bullet}}=c(g)\langle x,y\rangle_{a_{\bullet}} \text{ for some }c(g)\in R^{\times}\big\}.
\]

Consider the homomorphism of $\RR$-algebraic groups $h: \Res_{\CC/\RR}(\GG_m)\ra G_{a_{\bullet},\RR}$ given by
\begin{equation}\label{E:deligne-homomorphism}
h(z)=\prod_{i=1}^{f} \diag(\underbrace{z,\dots,z}_{a_{i}}, \underbrace{\bar z,\dots, \bar z}_{n-a_{i}}),\quad \text{for }z=x+\sqrt{-1}y.
\end{equation}
Let $\mu_{h}: \GG_{m,\CC}\ra G_{a_{\bullet},\CC}$ be the composite of $h_{\CC}$ with the map $\GG_{m,\CC}\ra \Res_{\CC/\RR}(\GG_{m})_{\CC}\cong \CC^{\times }\times \CC^{\times}$ given by $z\mapsto (z,1)$. Here, the first copy of $\CC^{\times}$ in $\Res_{\CC/\RR}(\GG_{m})_{\CC}$ is the one indexed by the identity element in $\mathrm{Aut}_{\RR}(\CC)$, and the other copy of $\CC^\times$ is indexed by the complex conjugation.

Let $E_{h}$ be the reflex field of $\mu_{h}$, i.e. the minimal subfield of $\CC$ where the conjugacy class of $\mu_{h}$ is defined.
 It has the following explicit description.  The group $\Aut_{\QQ}(\CC)$ acts naturally on $\Sigma_{\infty,E}$, and hence on the functions on $\Sigma_{\infty,E}$.
Then $E_{h}$ is the subfield of $\CC$ fixed by  the stabilizer of the  $\ZZ$-valued function $a$ on $\Sigma_{\infty,E}$ defined by $a(q_i)=a_i$ and $a(\bar q_i)=n-a_i$.
The isomorphism $\iota_p: \CC \xrightarrow{\sim} \overline \QQ_p$ defines a $p$-adic place $\wp$ of $E_h$. By our hypothesis on $E$, the local field $E_{h,\wp}$ is an unramified extension of $\Q_p$ contained in $\Q_{p^f}$,  the unique unramified extension over $\Q_p$ of degree $f$.

\subsection{Unitary Shimura varieties of PEL-type}
\label{S:defn of Shimura var}
Let $\calO_D$ be a $*$-stable order of $D$ and $\Lambda_{a_{\bullet}}$ an $\calO_D$-lattice of $V_{a_{\bullet}}$ such that $\langle \Lambda_{a_{\bullet}}, \Lambda_{a_{\bullet}} \rangle_{a_{\bullet}} \subseteq \ZZ$ and  $\Lambda_{a_{\bullet}} \otimes_{\Z} \Z_{p}$ is self-dual under the alternating pairing induced by $\langle-,-\rangle_{a_{\bullet}}$.
We put  $K_p=\Z_p^{\times}\times \GL_{n}(\cO_{E_{\gothp}})\subseteq G_{a_{\bullet}}(\Q_p)$,
and fix an open compact subgroup $K^p\subseteq G_{a_{\bullet}}(\AAA^{\infty,p})$ such that $K = K^pK_p$ is \emph{neat}, i.e. $G_{a_{\bullet}}(\Q)\cap g K g^{-1}$ is torsion free  for any $g\in G_{a_{\bullet}}(\AAA^{\infty})$.

Following
\cite{kottwitz}, we have a unitary Shimura variety $\cSh_{a_\bullet}$ defined over $\ZZ_{p^f}$;\footnote{Although one can descend $\cSh_{a_\bullet}$ to the  subring $\cO_{E_{h,\wp}}$ of $\ZZ_{p^f}$, we ignore this minor improvement here.} it represents the functor that takes a locally noetherian $\ZZ_{p^f}$-scheme $S$ to the set of isomorphism classes of tuples $(A, \lambda, \eta)$, where
\begin{enumerate}
\item
$A$ is an $fn^2$-dimensional abelian variety over $S$ equipped with an action of $\calO_D$ such that the induced action on $\Lie(A/S)$ satisfies the \emph{Kottwitz determinant condition}, that is, if we view the \emph{reduced} relative de Rham homology $H_1^\dR(A/S)^\circ : = \gothe H_1^\dR(A/S)$ and its quotient $\Lie^{\circ}_{A/S} : = \gothe\cdot \Lie_{A/S}$  as a module over $F_p \otimes_{\Zp} \calO_S \cong \bigoplus_{i=1}^f \calO_S$, they, respectively, decompose into the direct sums of locally free $\calO_S$-modules $H_1^\dR(A/S)^\circ_i$ of rank $n$ and, their quotients, locally free $\calO_S$-modules $\Lie^\circ_{A/S,i}$ of rank $n-a_i$;
\item
$\lambda: A \to A^\vee$ is a prime-to-$p$ $\calO_D$-equivariant polarization such that the Rosati involution induces the involution $*$ on $\calO_D$;
\item
$\eta$ is a collection of,
for each connected component $S_j$ of $S$ with a geometric point $\bar s_j$, a $\pi_1(S_j, \bar s_j)$-invariant $K^p$-orbit of isomorphisms $\eta_j: \Lambda_{a_{\bullet}} \otimes_\ZZ \widehat \ZZ^{(p)} \simeq T^{(p)}(A_{\bar s_j})$ such that the following diagram commutes for an isomorphism $\nu(\eta_j) \in \Hom( \widehat{\ZZ}^{(p)}, \widehat{\ZZ}^{(p)}(1))$:
\[
\xymatrix@C=60pt{
\Lambda_{a_{\bullet}} \otimes_{\ZZ}\widehat{\ZZ}^{(p)}\ \times\ \Lambda_{a_{\bullet}} \otimes_{\ZZ}\widehat{\ZZ}^{(p)} \ar[d]^{\eta_j \times \eta_j}
\ar[r]^-{\langle -, - \rangle}
& \widehat{\ZZ}^{(p)}
\ar[d]^{\nu(\eta_j)}
\\
T^{(p)}A_{\bar s_j} \ \times \ T^{(p)}A_{\bar s_j}
\ar[r]^-{\textrm{Weil pairing}}
&
\widehat{\ZZ}^{(p)}(1),
}
\]
where $\widehat \ZZ^{(p)} = \prod_{\ell \neq p}\ZZ_\ell$ and $T^{(p)}(A_{\bar s_j})$ denotes the product of the $\ell$-adic Tate modules of $A_{\bar s_j}$ for all $\ell\neq p$.
\end{enumerate}

The Shimura variety $\cSh_{a_\bullet}$ is smooth and projective over $\ZZ_{p^f}$ of relative dimension $d(a_\bullet) := \sum_{i=1}^f a_i(n-a_i)$.
 Note that if $a_i\in \{0,n\}$ for all $i$, then  $\cSh_{a_{\bullet}}$ is of relative dimension zero; we call it a \emph{discrete Shimura variety}.

We denote by $\cSh_{a_{\bullet}}(\CC)$  the complex points of $\cSh_{a_{\bullet}}$ via the embedding $\Z_{p^f}\hra\overline{ \Q}_p\xra{\iota^{-1}_p} \CC$.
Let $K_{\infty}\subseteq G_{a_{\bullet}}(\RR)$ be the stabilizer of $h$ \eqref{E:deligne-homomorphism} under the conjugation action, and let $X_{\infty}$ denote the $G_{a_{\bullet}}(\RR)$-conjugacy class of $h$.
 Then $K_{\infty}$ is a maximal compact-modulo-center subgroup of $G_{a_{\bullet}}(\RR)$.
 According to \cite[page 400]{kottwitz}, the complex manifold $\cSh_{a_{\bullet}}(\CC)$ is the disjoint union of
  $\#\mathrm{ker}^1(\Q,G_{a_{\bullet}})$ copies of
   \begin{equation}\label{E:shimura-complex}
   G_{a_{\bullet}}(\Q)\backslash \big (G_{a_{\bullet}}(\AAA^{\infty})\times X_{\infty}\big)/K\cong G_{a_{\bullet}}(\Q)\backslash G_{a_{\bullet}}(\AAA)/K\times K_{\infty}.
   \end{equation}
Here, if $n$ is even,  then  $\mathrm{ker}^1(\Q,G_{a_{\bullet}})=(0)$, while if $n$ is odd then
 \[
 \mathrm{ker}^1(\Q,G_{a_{\bullet}})=\mathrm{Ker}\Big(F^{\times}/\Q^{\times}N_{E/F}(E^{\times})\ra \AAA^{\times}_{F}/\AAA^{\times}N_{E/F}(\AAA^{\times}_E)\Big).
 \]
 In either case,  $\mathrm{ker}^1(\Q,G_{a_{\bullet}})$ depends only on the CM extension $E/F$ and the parity of  $n$ but not on the tuple $a_\bullet$. 

Let $\Sh_{a_\bullet} : = \cSh_{a_\bullet} \otimes_{\ZZ_{p^f}} \FF_{p^f}$ denote the special fiber of $\cSh_{a_{\bullet}}$, and let $\overline \Sh_{a_\bullet}: = \Sh_{a_\bullet} \otimes_{\FF_{p^f}} \overline \FF_p$ denote the geometric special fiber.

\subsection{$\ell$-adic cohomology}\label{S:l-adic-cohomology}
We fix a prime number $\ell\neq p$, and an isomorphism $\iota_{\ell}: \CC\simeq \Qlb$. Let $\xi$ be an  algebraic representation of $G_{a_{\bullet}}$ over $\Qlb$, and $\xi_{\CC}$ be the base change via $\iota_{\ell}^{-1}$.
The  theory of automorphic sheaves (\cite[Section~III]{milne}) or just reading off from   the rational $\ell$-adic Tate modules of the universal abelian variety allows us to attach to $\xi$ a lisse $\Qlb$-sheaf $\calL_\xi$ over $\cSh_{a_\bullet}$.
For example, if $\xi$ is the representation of $ G_{a_{\bullet}}$ on the vector space $V_{a_{\bullet}}$ (Section~\ref{S:Shimura data}),   the corresponding $\ell$-adic local system is given by the rational  $\ell$-adic Tate module (tensored with $\Qlb$) of the universal abelian scheme over $\cSh_{a_{\bullet}}$.

We assume that $\xi$ is irreducible.
Let $\scrH_{K}=\scrH(K,\Qlb)$ be the Hecke algebra of compactly supported $K$-bi-invariant $\Qlb$-valued functions on $G_{a_{\bullet}}(\AAA^{\infty})$.
The \'etale cohomology group $H^{d(a_{\bullet})}_{\et}(\overline \Sh_{a_{\bullet}}, \calL_{\xi})$ is equipped with a  natural action of $\scrH_{K}\times \Gal(\Fpb/\FF_{p^f})$. Since $\Sh_{a_{\bullet}}$ is proper and smooth, there is no continuous spectrum and we have  a canonical decomposition of $\scrH_{K}\times \Gal(\Fpb/\FF_{p^f})$-modules (see e.g. \cite[Proposition~III.2.1]{harris-taylor})
\begin{equation}\label{E:decomposition-cohomology}
H^{d(a_{\bullet})}_{\et}(\overline\Sh_{a_{\bullet}},\calL_{\xi})=\bigoplus_{\pi\in \mathrm{Irr}(G_{a_{\bullet}}(\AAA^{\infty}))} \iota_{\ell}(\pi^K) \otimes R_{a_{\bullet},\ell}(\pi),
 \end{equation}
  where $\mathrm{Irr}(G_{a_{\bullet}}(\AAA^{\infty}))$ denotes the set of irreducible admissible representations of $G_{a_{\bullet}}(\AAA^{\infty})$ with coefficients in $\CC$,  $\pi^K$ is the $K$-invariant subspace of $\pi\in \mathrm{Irr}(G_{a_{\bullet}}(\AAA^{\infty}))$, and $R_{a_{\bullet},\ell}(\pi)$ is a certain $\ell$-adic representation of $\Gal(\Fpb/\FF_{p^f})$ which we specify below.
  
We write $H_\et^{d(a_\bullet)}(\overline\Sh_{a_\bullet}, \calL_\xi)_\pi$ for the  \emph{$\pi$-isotypic component} of the cohomology group, that is the direct summand of \eqref{E:decomposition-cohomology} labeled by $\pi$.
We  make the following assumption on $\pi$.
  \begin{hypothesis}
\label{H:automorphic assumption}
\begin{enumerate}
\item
We have $\pi^K \neq 0$.
\item
There exists an admissible irreducible representation $\pi_{\infty}$  of $G_{a_{\bullet}}(\RR)$ such that  $\pi\otimes\pi_{\infty}$ is a cuspidal  automorphic representation of $G_{a_{\bullet}}(\AAA)$, 
\begin{itemize}
\item[(2a)]
$\pi_\infty$ is \emph{cohomological in degree $d(a_{\bullet})$ for $\xi$} in the sense that
\begin{equation}\label{E:g-K cohomology} 
H^{d(a_{\bullet})}\big(\Lie(G_{a_{\bullet}}(\RR)), K_{\infty}, \pi_{\infty}\otimes\xi_{\CC}\big)\neq 0,\footnote{This automatically implies that $\pi_\infty$ has the same central  and infinitesimal characters as the \emph{contragradient} of $\xi_\CC$.}
\end{equation}
where $K_{\infty}$ is a maximal compact subgroup of $G_{a_{\bullet}}(\RR)$,

\item[(2b)] and 
$\pi\otimes\pi_{\infty}$ admits a  base change to a \emph{cuspidal} automorphic representation of $\GL_{n}(\AAA_{E})\times \AAA_{E_0}^{\times}$.

\end{itemize}

\end{enumerate}
\end{hypothesis}



Note that Hypothesis~\ref{H:automorphic assumption}(1) implies that the $p$-component $\pi_p$ is unramified.
Hypothesis~\ref{H:automorphic assumption} (2)(a)  ensures that $R_{a_{\bullet},\ell}(\pi)$ is non-trivial.
Moreover,  by \cite[Theorem~1.2]{caraiani}, this hypothesis implies that the base change of $\pi \otimes \pi_\infty$ to $\GL_{n,E}$ is tempered at all finite places, and hence $\pi_p$ is tempered. 

We recall now an explicit description, due to Kottwitz \cite{kottwitz2}, of the Galois module  $R_{a_{\bullet},\ell}(\pi)$.
As  $G_{a_{\bullet}}(\Q_p)=\Q_p^{\times}\times \GL_n(E_\gothp)$, we may write $\pi_{p}=\pi_{p,0}\otimes \pi_{\gothp}$, where $\pi_{p,0}$ is a character of $\Q_p^{\times}$ trivial on $\Z_p^{\times}$, and $\pi_{\gothp}$ is an irreducible admissible representation of $\GL_n(E_{\gothp})$ such that $\pi_{\gothp}^{\GL_n(\cO_{E_{\gothp}})}\neq 0$.
Choose a square root $\sqrt{p}$ of $p$ in $\overline{\Q}$. Depending on this choice of $\sqrt{p}$, one has an  (unramified) local Langlands parameter attached to $\pi_p$:
 $$
 \varphi_{\pi_{p}}=(\varphi_{\pi_{p,0}},\varphi_{\pi_{\gothp}})\colon  W_{\Qp} \longrightarrow {}^L(G_{a_{\bullet},\Q_p})\simeq \CC^\times \times \big( \GL_n(\CC)^{\Z/f\Z}\rtimes \Gal(\overline\QQ_p/\Qp) \big).
 $$
   Here, $W_{\Q_p}$ denotes the Weil group of $\Q_p$, and  $\Gal(\overline{\QQ}_p/\Qp)$  permutes cyclically the $f$ copies of $\GL_n(\CC)$ though the quotient $\Gal(\Q_{p^f}/\Q_p)\cong \Z/f\Z$.
   The image of   $\varphi_{\pi_p}|_{W_{\Q_{p^f}}}$  lies in $({}^{L}G_{a_{\bullet}})^{\circ}\simeq\CC^{\times}\times \GL_n(\CC)^{\Z/f\Z}$.
The cocharacter $\mu_{h}: \GG_{m,E_h}\ra G_{a_{\bullet},E_{h}}$ induces  a character $\check{\mu}_h$ of $({}^{L}G_{a_{\bullet}})^{\circ}$ over $E_h$. Let $r_{\mu_h}$ denote the algebraic representation  of $({}^{L}G_{a_{\bullet}})^{\circ}$ with extreme weight $\check\mu_{h}$.
Denote by  $\Frob_{p^f}$  a \emph{geometric} Frobenius element in $W_{\Q_{p^f}}$.
Let $\overline \QQ_\ell(\frac 12)$ denote the unramified representation of $W_{\Q_{p^f}}$ which sends $\Frob_{p^f}$ to $(\sqrt{p})^{-f}$. 
Then $R_{a_{\bullet},\ell}(\pi)$ can be described in terms of $\varphi_{\pi_p}$ as follows.

\begin{theorem}[{\cite[Theorem 1]{kottwitz2}}]\label{T:galois-action}
Under the hypothesis and notation above, we have an equality in the Grothendieck group of $W_{\Q_{p^f}}$-modules:
\[
[R_{a_{\bullet},\ell}(\pi)]=\#\mathrm{ker}^1(\Q,G_{a_{\bullet}}) \;m_{a_{\bullet}}(\pi)\;\big[ \big(\iota_{\ell}(r_{\mu_{h}}\circ \varphi_{\pi_p})\otimes \overline \QQ_\ell(-\tfrac{d(a_{\bullet})}{2})\big],
\]
where $m_{a_{\bullet}}(\pi)$ is a certain integer related to the automorphic multiplicities of automorphic representations of $G_{a_{\bullet}}$ with finite part $\pi$. 
\footnote{Rigorously speaking, Kottwitz' theorem describes the direct sum of the $\pi$-component of all cohomological degrees. Since our $\pi_p$ is tempered, so $\pi$ appears only in the middle degree for purity reasons because $\Sh_{a_\bullet}$ is compact.}
\end{theorem}

     In our case, one can make Kottwitz' Theorem more transparent.
    Define an $\ell$-adic representation
     \begin{equation}\label{E:defn-rho-pi}
     \rho_{\pi_{\gothp}}=\iota_{\ell}(\varphi_{\pi_{\gothp}}^{(1), \vee})\otimes \overline \QQ_\ell\big(\tfrac{1-n}{2}\big)\colon W_{\Q_{p^f}}\longrightarrow \GL_n(\Qlb),
     \end{equation}
where $\varphi^{(1),\vee}_{\pi_{\gothp}}: W_{\Q_{p^f}}\ra \GL_n(\CC)$ denotes the {\it contragredient} of  the projection to the first (or any) copy of $\GL_{n}(\CC)$. Both $\varphi_{\pi_{\gothp}}$ and $\overline \QQ_\ell(\frac 12)$ depend on the choice of $\sqrt{p}$,  but $\rho_{\pi_{\gothp}}$ does not.
Explicitly, $\rho_{\pi_{\gothp}}(\Frob_{p^f})$ is semisimple with the characteristic polynomial given by \cite[(6.7)]{gross}:
\begin{equation}\label{E:hecke-polynomial}
X^n + \sum_{i=1}^n(-1)^{i} (N\gothp)^{i(i-1)/2}a_\gothp^{(i)} X^{n-i},
\end{equation}
 where $a_\gothp^{(i)}$ is the eigenvalue on $\pi_\gothp^{\GL_n(\cO_{E_\gothp})}$ of the  Hecke operator
\[
T_\gothp^{(i)} = \GL_n(\cO_{E_{\gothp}}) \cdot
\mathrm{diag}(\underbrace{p, \dots, p}_i, \underbrace{1, \dots, 1}_{n-i}) \cdot \GL_n(\cO_{E_{\gothp}}).
\]
An easy computation shows that  $r_{\mu_{h}}=\mathrm{Std}_{\QQ_p^\times}^{-1}\otimes \bigotimes_{i=1}^f(\wedge^{a_i}\mathrm{Std}^\vee)$.
 Since the projection of $\varphi_{\pi_{p}}|_{W_{\Q_{p^f}}}$ to  each copy of $\GL_n(\CC)$ is conjugate to all others, Theorem~\ref{T:galois-action} is equivalent to
\begin{equation}\label{E:Galois-Sh}
[R_{a_{\bullet},\ell}(\pi)]=\#\mathrm{ker}^1(\Q,G_{a_{\bullet}}) \;m_{a_{\bullet}}(\pi)\;\big[\rho_{a_{\bullet}}(\pi_{\gothp})\otimes \chi_{\pi_{p,0}}^{-1}\otimes\Qlb\big(\textstyle\sum_i\tfrac{a_i(a_i-1)}{2}\big)\big],
\end{equation}
where $\rho_{a_{\bullet}}(\pi_{\gothp})=r_{a_{\bullet}}\circ\rho_{\pi_{\gothp}}$ with $r_{a_{\bullet}}=\bigotimes_{i=1}^f \wedge^{a_i}\mathrm{Std}$, and $\chi_{\pi_{p,0}}$ denotes the  character of $\Gal(\Fpb/\F_{p^f})$ sending $\Frob_{p^f}$ to $\iota_{\ell}(\pi_{p,0}(p^f))$.

\begin{remark}
The reason why we normalize the Galois representation as above is the following: By Hypothesis~\ref{H:automorphic assumption}, $\pi$ is the finite part of an automorphic representation of $G_{a_{\bullet}}(\AAA)$ which admits a base change to a cuspidal automorphic representation $\Pi\otimes \chi$ of $\GL_n(\AAA_E)\times \AAA_{E_0}^{\times}$. 
If $\rho_{\Pi}$ denotes the the  Galois representation  of $\Gal(\overline \QQ/E)$ attached to $\Pi$,  then    $\rho_{\pi_{\gothp}}$ is the semi-simplification of the  restriction of $\rho_{\Pi}$ to $W_{E_{\gothp}}$ (See \cite[Theorem~1.1]{caraiani}).
\end{remark}

\subsection{Tate conjecture}\label{S:Tate-conjecture}
We recall first  the Tate conjecture over finite fields.
Let $X$ be a projective smooth variety over a finite field $\FF_q$ of characteristic $p$. Put $\overline X = X_{\Fpb}$. For each prime $\ell\neq p$ and integer $r\leq \dim(X)$, we have a cycle class map
\[
\cl_{X}^r:
A^r( X)\otimes_{\ZZ} \overline\QQ_{\ell}\longrightarrow H^{2r}_\et\big(\overline X, \overline  \QQ_{\ell}(r)\big)^{\Gal(\Fpb/\FF_{q})},
\]
where $A^r(X)$ denotes the abelian group of codimension $r$ algebraic cycles in $X$ defined over $\FF_{q}$.
Then the Tate conjecture predicts that this map is surjective.
One has a geometric variant of the Tate conjecture, which claims that
the geometric cycle class map:
\[
\cl^r_{\overline X}\colon A^r(\overline X)\otimes_{\ZZ}\Qlb\longrightarrow H^{2r}_\et\big(\overline X, \Qlb(r)\big)^{\mathrm{fin}}:=\bigcup_{m\geq 1}H^{2r}_\et\big(\overline X, \Qlb(r)\big)^{\Gal(\Fpb/\FF_{q^m})}
\]
is surjective. Here, the superscript ``fin'' means  the subspace on which $\Gal(\Fpb/\FF_{q})$ acts through a finite quotient.
Note that the surjectivity of $\cl^{r}_{\overline X}$ implies that of $\cl_{X}^r$ by taking the $\Gal(\overline \FF_p/\FF_{q})$-invariant subspace.

Consider the case $X = \Sh_{a_\bullet}$ with  $d(a_{\bullet})$ even.
 Let $\pi$ be an irreducible admissible representation of $G_{a_{\bullet}}(\AAA^{\infty})$  as in Theorem~\ref{T:galois-action}.
 By Theorem~\ref{T:galois-action},
  the $\pi$-isotypic component of  $H^{d(a_{\bullet})}_{\et}\big(\overline\Sh_{a_{\bullet}},\Qlb(\frac{d(a_{\bullet})}{2})\big)^{\mathrm{fin}}$ is, up to Frobenius semi-simplification\footnote{Conjecturally, the   Frobenius action on the \'etale $\ell$-adic cohomology groups of a projective smooth variety over a finite field is always semisimple.}, isomorphic to
$\dim(\pi^K)\cdot \#\mathrm{ker}^1(\Q,G_{a_{\bullet}})\cdot m_{a_{\bullet}}(\pi)$ copies of
\begin{equation}\label{E:invariant-subspace}
\big(\rho_{a_{\bullet}}(\pi_{\gothp})\otimes \chi^{-1}_{\pi_{p,0}}\otimes \Qlb(\tfrac{(n-1)}2\textstyle\sum_{i=1}^fa_{i})
\big)^{\fin}.
\end{equation}
Note that $\chi_{\pi_{p,0}}(\Frob_{p^f})=\pi_{p,0}(p^f)$ is a root of unity.
 Hence, the dimension of
\eqref{E:invariant-subspace} is equal to the sum of the dimensions of the $\Frob_{p^f}$-eigenspaces of $\rho_{a_{\bullet}}(\pi_{\gothp})$ with eigenvalues $(p^f)^{\frac{(n-1)}{2}\sum_{i} a_i}\zeta$ for some root of unity $\zeta$.
In many examples, this space is known to be non-zero.

For instance, when $f=2$, $a_1 = r$ and $a_2 = n-r$ for some $1\leq r\leq n-1$, we have $d(a_\bullet) = 2r(n-r)$ and
$$
\rho_{a_{\bullet}}(\pi_{\gothp})=\wedge^{r}\rho_{\pi_{\gothp}}\otimes \wedge^{n-r}\rho_{\pi_{\gothp}}.
$$
Let $V_{\pi_\gothp, a_\bullet}$ denote the space of representation $\rho_{a_\bullet}(\pi_\gothp)$.
If $\rho_{\pi_{\gothp}}(\Frob_{p^f})$ has distinct eigenvalues $\alpha_1, \dots, \alpha_n$, then the eigenvalues of $\Frob_{p^f}$ on $V_{\pi_\gothp, a_\bullet}$ are given by $\alpha_{i_1} \cdots \alpha_{i_r} \cdot \alpha_{j_1} \cdots \alpha_{j_{n-r}}$, for distinct subscripts  $i_1, \dots, i_r$ and distinct subscripts $j_1, \dots, j_{n-r}$.
This product is exactly $(p^f)^{\frac{n(n-1)}{2}} a_{\gothp}^{(n)}$ (note that $a_{\gothp}^{(n)}$ is a root of unity) if the set $\{i_1, \dots, i_r\}$ and the set $\{j_1, \dots, j_{n-r}\}$ are the complement of each other as subsets of $\{1,\dots, n\}$.
On the other hand, if the subsets $\{i_1, \dots, i_r\}$ and $\{j_1, \dots, j_{n-r}\}$ are not the complement of each other and if the  $\alpha_i$ are ``sufficiently generic".\footnote{For example, if $r=1$ and $\alpha_1 = \alpha_2$, the eigenvalues $\alpha_1 \cdot \alpha_1\alpha_3\alpha_4\cdots \alpha_n$ is equal to $\alpha_1\cdots \alpha_n$ and hence is $p^{n(n-1)}$ times a root of unity. So to be in the generic case, we will need to require that $\alpha_i/\alpha_j$ for $i \neq j$ is not a root of unity if $r=1$. For another example, if $r=2$, ``generic" will mean that $\alpha_i/\alpha_j$ for $i \neq j$ and $\alpha_i\alpha_{i'}/\alpha_j\alpha_{j'}$ for $\{i,i'\} \neq \{j,j'\}$ are not roots of unity.}, the eigenvalue $\alpha_{i_1}\cdots \alpha_{i_r} \cdot \alpha_{j_1}\cdots \alpha_{j_{n-r}}$ is not a root of unity. 
In other words, the dimension of \eqref{E:invariant-subspace} is ``generically" equal to  $\binom nr$.
As predicted by the Tate conjecture, these cohomology classes should come from algebraic cycles.  Our main conjecture addresses exactly this,  and it  predicts that those desired  ``generic'' algebraic   cycles can be given by the irreducible components of the basic locus, and are birationally equivalent to certain fiber bundles over the special fiber of some other Shimura varieties associated to  inner forms of $G_{a_{\bullet}}$.
To make this precise, we need the following lemma. 

\begin{lemma}\label{L:change-signature}
Let $b_{\bullet}=(b_{i})_{1\leq i\leq f}$ be a tuple with $b_i\in \{0,\dots, n\}$ such that $\sum_{i=1}^fb_i\equiv \sum_{i=1}^fa_i\pmod 2$ if $n$ is even. Then there exists $\beta_{b_{\bullet}}\in (D^{\times})^{*=-1}$ such that
\begin{itemize}
\item
the alternating $D$-Hermitian space $(V_{b_\bullet}, \langle-, - \rangle_{b_\bullet})$ defined using $\beta_{b_\bullet}$ in place of $\beta_{a_\bullet}$ is isomorphic to $(V_{a_\bullet}, \langle-, - \rangle_{a_\bullet})$ when tensored with $\AAA^\infty$, and
\item
if $G_{b_{\bullet}}$ denotes the corresponding algebraic group over $\Q$ defined in the similar way with $\beta_{a_{\bullet}}$ replaced by $\beta_{b_{\bullet}}$, then
\[
G^1_{b_\bullet}(\RR)\simeq \prod_{i=1}^f U(b_i,n-b_i).
\]
\end{itemize}
\end{lemma}
\begin{proof}
We reduce the problem to the existence of a certain cohomology class. 
Note that $G_{a_{\bullet}}^1=\Aut(V_{a_{\bullet}}, \langle-,-\rangle_{a_{\bullet}})$ is the Weil restriction to $\QQ$ of a unitary group $U_{a_{\bullet}}$ over $F$.
The cohomology set $H^1(\QQ,G_{a_{\bullet}}^1)\cong H^1(F,U_{a_{\bullet}})$ is in bijection with the isomorphism classes of  one-dimensional skew-Hermitian $D$-modules $V$. 
 As $U_{a_{\bullet}}\times_{F}E\simeq \GL_{n,E}$, Hilbert Theorem 90  for $\GL_{n}$ implies that the inflation map induces an isomorphism
\[
 H^1(E/F,U_{a_{\bullet}})\xra{\sim} H^1(F, U_{a_{\bullet}}). 
\]
Denote by $g\mapsto g^{\sharp_{\beta_{a_{\bullet}}}}=\beta_{a_{\bullet}}g^*\beta_{a_{\bullet}}^{-1}$ the involution on $D$ induced by the alternating pairing $\langle -,-\rangle_{{a_{\bullet}}}$.
  Then a $1$-cocyle of $\Gal(E/F)$ with values in $U_{a_{\bullet}}$ is given by an element $\alpha\in D^{\times}$ such that $\alpha=\alpha^{\sharp_{\beta_{a_{\bullet}}}}$, and $\alpha_1,\alpha_2\in D^{\times}$ define the same cohomology class in $ H^1(F,U_{a_{\bullet}})$ if and only if there exists $g\in D^{\times}$ such that $g\alpha_1 g^{\sharp_{\beta_{a_{\bullet}}}}=\alpha_2$. 
  Explicitly, given such an $\alpha$, the corresponding skew-Hermitian $D$-module is given by $V=D$ equipped with the alternating pairing 
  \[
  \langle -,-\rangle_{\alpha}: V\times V\ra \QQ,\quad\quad   (x,y)\mapsto \Tr_{D/\QQ}(x \alpha\beta_{a_{\bullet}}y^*).
  \]
  For a place $v$ of $F$, we denote by 
  \[\loc_v: H^1(F,U_{a_{\bullet}})\ra H^1(F_v, U_{a_{\bullet}})\]
   the canonical localization map.
By  \cite[Proposition 8.1]{helm}, if $\sum_{i=1}^fb_i\equiv \sum_{i=1}^f a_i\mod 2$,  there exists a cohomology class  $[\alpha]\in H^1(F,U_{a_{\bullet}})$ such that 
  \begin{itemize}
  \item $\loc_v([\alpha])$ is trivial for every finite place $v$ of $F$, and 
  \item if $v=\tau_i$ with $1\leq i\leq n$ is an archimedean place, one has an isomorphism  of unitary groups over $\RR$: $\Aut(V\otimes_{F,\tau_i}\RR, \langle -,-\rangle_{\alpha})\simeq U(b_i, n-b_i)$.  
  \end{itemize} 
Then the element $\beta_{b_\bullet}=\alpha \beta_{a_{\bullet}}$ meets the requirements of  Lemma~\ref{L:change-signature}. 
\end{proof}

In the sequel,  we  always fix  a choice of $\beta_{b_{\bullet}}$, and as well as an isomorphism $\gamma_{a_{\bullet},b_{\bullet}}: V_{a_{\bullet}}\otimes_{\Q}\AAA^{\infty}\xra{\sim} V_{b_{\bullet}}\otimes_{\Q}\AAA^{\infty}$, which induces an isomorphism $G_{a_{\bullet}}(\AAA^{\infty})\simeq G_{b_{\bullet}}(\AAA^{\infty})$.
Recall that we have chosen a lattice $\Lambda_{a_{\bullet}}\subseteq V_{a_{\bullet}}$ to define the moduli problem for $\cSh_{a_{\bullet}}$.
   We put $\Lambda_{b_{\bullet}}:= V_{b_{\bullet}}\cap \gamma_{a_{\bullet},b_{\bullet}}(\Lambda_{a_{\bullet}}\otimes_{\Z} \widehat{\Z} )$. Then applying the construction of Subsection~\ref{S:defn of Shimura var} to the lattice $\Lambda_{b_{\bullet}}\subseteq V_{b_{\bullet}}$ and the open compact subgroup $K^p\subseteq G_{a_{\bullet}}(\AAA^{\infty,p})\simeq G_{b_{\bullet}}(\AAA^{\infty,p})$, we get a Shimura variety $\calS h_{b_{\bullet}}$ over $\Z_{p^f}$ of level $K^p$ as well as its special fiber $\Sh_{b_{\bullet}}$.
   Moreover, an algebraic representation $\xi$ of $G_{a_{\bullet}}$ over $\Qlb$ corresponds, via the fixed isomorphism $G_{a_{\bullet}}(\AAA^{\infty})\simeq G_{b_{\bullet}}(\AAA^{\infty})$, to an algebraic representation of $G_{b_{\bullet}}$ over $\Qlb$.
    We  use the same notation $\calL_{\xi}$ to denote  the \'etale sheaf on $\cSh_{a_{\bullet}}$ and $\cSh_{b_{\bullet}}$ defined by $\xi$.

\subsection{Gysin/trace maps}\label{S:Gysin-maps}
Before stating the main conjecture of this paper, we recall the general definition of Gysin maps. Let $f:Y \to X$ be a proper morphism of  smooth varieties over an algebraically closed field $k$. Let $d_X$ and $d_Y$ be the dimensions of $X$ and $Y$ respectively. 
Recall that the derived direct image $Rf_*$ on the derived category of constructible $\ell$-adic \'etale sheaves  has a left adjoint $f^!$. 
Since both $X$ and $Y$ are smooth, the $\ell$-adic dualizing complex of $X$ (resp. $Y$) is $\overline \Q_{\ell}(d_X)[2d_X]$ (resp. $\Qlb(d_Y)[2d_Y]$). 
Therefore, one has 
\[
f^!\big(\Qlb(d_X)[2d_X]\big)=\Qlb(d_Y)[2d_Y].
\]
The adjunction map $Rf_* f^!\Qlb\ra \Qlb$ induces a canonical morphism 
\[\Tr_f\colon  Rf_*\Qlb\ra \Qlb(d_X-d_Y)[2(d_X-d_Y)].\]
More generally, if $\calL$ is a lisse $\Qlb$-sheaf on $X$, it induces  a Gysin/trace map 
\[
Rf_*(f^*\calL)\cong \calL\otimes  Rf_*(\Qlb)\xra{1\otimes\Tr_f} \calL(d_X-d_Y)[2(d_X-d_Y)],
\]
where the first isomorphism is the projection formula \cite[XVII 5.2.9]{SGA4}. When $f$ is  flat with equidimensional  fibers of  dimension $d_Y-d_X$, this is the trace map as defined in \cite[XVIII 2.9]{SGA4}. When $f$ is a closed immersion of codimension $r=d_X-d_Y$, it is the usual Gysin map. 
For any integer $q$, the Gysin/trace map induces  a morphism on cohomology groups:
\begin{equation}\label{E:Gysin-map}
f_!\colon H^q_{\et}(Y, f^*\calL)\longrightarrow H^{q+2(d_X-d_Y)}_{\et}\big(X, \calL(d_X-d_Y)\big).
\end{equation} 

\subsection{Representation theory of $\GL_n$}
\label{S:representation of GLn}
As suggested by the description in Theorem~\ref{T:galois-action} of Galois representations appearing in the middle cohomology group of Shimura varieties, as well as the Tate conjecture, we need to understand the representation theory of $\GL_n$ embedded diagonally into the Langlands dual group
\[
({}^LG_{a_\bullet})^\circ \simeq \CC^\times \times \GL_n(\CC)^{\ZZ/f\ZZ}.
\]
The Hodge cocharacter $\mu$ of $G_{a_\bullet}$ gives rise to the representation $r_{a_\bullet}=\bigotimes_{i=1}^f (\wedge^{a_i}\mathrm{Std})$ of the diagonal $\GL_n$.
If $\lambda$ is a dominant weight of $\GL_n$ (with respect to the usual diagonal torus and upper triangular Borel subgroup) appearing in $r_{a_\bullet}$, we can write this weight $\lambda$ as the sum of $f$ dominant minuscule weights $\omega_{b_1}+\cdots +\omega_{b_f}$, where $\omega_i$ for $0 \leq i\leq n$ is the weight of $\GL_n$ that takes $\mathrm{Diag}(t_1, \dots, t_n)$ to $t_1 \cdots t_i$.  The set $\{b_1, \dots, b_f\}$ (counted with multiplicity) is unique, which we denote by $B_\lambda$.
Explicitly, if $\lambda$ takes $\mathrm{Diag}(t_1, \dots, t_n)$ to $t_1^{\beta_1} \cdots t_n^{\beta_n}$ (necessarily $\beta_1\leq f$), then
\[
B_\lambda = \big\{\underbrace{n, \dots, n}_{\beta_n},\, \underbrace{n-1, \dots, n-1}_{\beta_{n-1}-\beta_n}, \dots, \underbrace{1, \dots, 1}_{\beta_1-\beta_0}, \, \underbrace{0, \dots, 0}_{f-\beta_1}\big\}.
\]

 Moreover, we always have $\sum a_i = \sum b_i$.
In particular, this implies by Lemma~\ref{L:change-signature} that the Shimura variety $\Sh_{b_\bullet}$ makes sense, and the \'etale sheaf $\calL_\xi$ is well defined on $\Sh_{b_\bullet}$.

We write $m_\lambda(a_\bullet)$ for the multiplicity of the weight $\lambda$ in $r_{a_\bullet}$.

\begin{conjecture}
\label{Conj:main}
Let $\Sh_{a_{\bullet}}$ and $\calL_{\xi}$ be as in Subsection~\ref{S:l-adic-cohomology}.
Let $\lambda$ be a dominant weight that appears in 
the representation $r_{a_{\bullet}}$ as in Section~\ref{S:representation of GLn}. Define $B_\lambda$ and $m_\lambda(a_\bullet)$ as in Section~\ref{S:representation of GLn}.

Then there exist varieties $Y_1, \dots, Y_{m_\lambda(a_\bullet)}$ of dimension $\frac{d(a_\bullet) + d(b_\bullet)}{2}$ over $\FF_{p^f}$, equipped with natural action of prime-to-$p$ Hecke correspondences, such that each $Y_j$ fits into a diagram
\[
\xymatrix{
& Y_j\ar[rd]^{\pr_{b_{\bullet}^{(j)}}}\ar[ld]_{\pr^{(j)}_{a_{\bullet}}}\\
\Sh_{a_{\bullet}} &&\Sh_{b_{\bullet}^{(j)}},
}
\]
satisfying the following properties.
\begin{enumerate}

\item 
For each $j$,  $b_\bullet^{(j)} = (b_1^{(j)}, \dots, b_f^{(j)})$ is a reordering of the elements of the set $B_\lambda$, and  both $\pr^{(j)}_{a_{\bullet}}$ and $\pr_{b_{\bullet}^{(j)}}$ are  equivariant for the prime-to-$p$ Hecke correspondences.

\item The morphism $\pr^{(j)}_{a_{\bullet}}$ is a proper morphism and is birational onto the image. The morphism $\pr_{b_{\bullet}^{(j)}}$ is proper and generically smooth of relative dimension $\frac{d(a_\bullet) - d(b_\bullet)}{2}$ (note that $d(b_{\bullet})\equiv d(a_{\bullet})\pmod 2$ since $\sum_ia_i= \sum_ib_i$).

\item There exists a $p$-isogeny of abelian schemes over $Y_j$
\[
\phi_{b_{\bullet}^{(j)},a_{\bullet}}\colon \pr_{b_{\bullet}^{(j)}}^*(\calA_{b_{\bullet}^{(j)}})\longrightarrow \pr_{a_{\bullet}}^{(j),*}(\calA_{a_{\bullet}}),
\]
where $\calA_{a_{\bullet}}$ and $\calA_{b_\bullet^{(j)}}$ denote respectively the universal abelian scheme on $\Sh_{a_{\bullet}}$ and $\Sh_{b_{\bullet}^{(j)}}$. Let  
\[
\phi_{b_{\bullet}^{(j)},a_{\bullet}, *}\colon \pr_{b_\bullet^{(j)}}^*\calL_\xi\xra{\ \sim\ } \pr_{a_\bullet}^{(j),*}\calL_\xi.
\]
be the  isomorphism of the $\ell$-adic sheaves induced by $\phi_{b_{\bullet}^{(j)},a_{\bullet}}$ via the construction in Section~\ref{S:l-adic-cohomology}.\footnote{This isomorphism depends on the choice of the isomorphism $\gamma_{a_\bullet, b_\bullet}$ made earlier.}

\item
Let $\pi$ be an    irreducible admissible representation of $G_{a_{\bullet}}(\AAA^{\infty})\simeq G_{b_{\bullet}^{(j)}}(\AAA^{\infty})$ satisfying Hypothesis~\ref{H:automorphic assumption} for both $a_\bullet$ and $b_\bullet$, and   assume that $m_{a_{\bullet}}(\pi)=m_{b_{\bullet}^{(j)}}(\pi)$ for all $j$ \footnote{This assumption is satisfied when $\pi$ is the finite part of an automorphic cuspidal representation of $G_{a_{\bullet}}(\AAA)$, which admits a  base change to a cuspidal automorphic representation of $\GL_{n}(\AAA_E)\times\AAA_E^{\times} $. Indeed, in this case, White proved that $m_{a_{\bullet}}(\pi)=m_{b^{(j)}_\bullet}(\pi)=1$ \cite[Theorem E]{white}.}.
Suppose that the $n$ eigenvalues $\alpha_1, \dots, \alpha_n$ of $\rho_{\pi_{\gothp}}(\Frob_{p^f})$ are ``sufficiently generic" in the sense that
the generalized eigenspace decomposition of $\rho_{a_\bullet}(\Frob_{p^N})$ for any large $N$ is the same as the weight space decomposition of the algebraic representation $r_{a_\bullet}$.
Then the  natural homomorphism of $\pi$-isotypic components\footnote{The $\pi$-isotypic component is the same as the $\pi^p$-isotypic component according to Lemma~\ref{L:automorphic-prime-to-p}.} of the cohomology groups
\begin{align*}
\bigoplus_{j=1}^{m_{\lambda}(a_\bullet)}
H^{d(b_\bullet)}_\et\Big(\overline \Sh_{b_\bullet^{(j)}}, \calL_\xi\big(\tfrac{d(b_{\bullet})}{2}\big)\Big)_\pi^{\Frob_{p^f} = \lambda} &\xrightarrow{\oplus\pr_{b_{\bullet}^{(j)}}^*}
\bigoplus_{j=1}^{m_{\lambda}(a_\bullet)}
H^{d(b_\bullet)}_\et\Big(\overline Y_{j},\pr_{b_{\bullet}^{(j)}}^* \calL_\xi\big(\tfrac{d(b_{\bullet})}{2}\big)\Big)_\pi^{\Frob_{p^f} = \lambda}\\
&\xra{\oplus\phi_{b_{\bullet}^{(j)},a_{\bullet}, *}} \bigoplus_{j=1}^{m_{\lambda}(a_\bullet)}
H^{d(b_{\bullet})}_{\et}\Big(\overline Y_{j},\pr_{a_{\bullet}}^*\calL_{\xi}\big(\tfrac{d(b_\bullet)}{2}\big)\Big)_{\pi}^{\Frob_{p^f} = \lambda}\\
&\xrightarrow{\sum \pr^{(j)}_{a_{\bullet}, !}}
H^{d(a_\bullet)}_\et\Big(\overline \Sh_{a_\bullet}, \calL_\xi\big(\tfrac{d(a_\bullet)}{2}\big)\Big)_\pi^{\Frob_{p^f} = \lambda}
\end{align*}
is an isomorphism,  where $\pr^{(j)}_{a_{\bullet},!}$ is the Gysin map \eqref{E:Gysin-map}
 and the superscript $\Frob_{p^f}=\lambda$  means taking the (direct sum of) generalized $\Frob_{p^f}$-eigenspace with  eigenvalues in the Weyl group orbit
 $$
 \lambda \circ \rho_{\pi_{\gothp}}(\Frob_{p^f}) \cdot \chi_{\pi_{p,0}}^{-1}(p^f)(\sqrt{p})^{-f(n-1)\sum_{i}b_i}.
 $$
 Here, since the semisimple conjugacy classes of $\GL_{n}(\overline \Q_{\ell})$ is in natural bijection with the orbits of  $T(\overline \Q_{\ell})$ under the Weyl group of $\GL_n$,  it makes sense to evaluate a dominant  weight of $T$ on $\rho_{\pi_\gothp}(\Frob_{p^f})$ to get an orbit under the action of the Weyl group of $\GL_n$; hence the notation $ \lambda \circ \rho_{\pi_{\gothp}}(\Frob_{p^f})$.
\end{enumerate}

In particular, when $\xi$ is the trivial representation and the weight $\lambda$ is a power of determinant (so automatically, $\sum_i a_i$ is divisible by $n$, and $d(a_{\bullet})$ is even),
 the cycles given by the images of $Y_1, \dots, Y_{m_\lambda(a_\bullet)}$ parametrized by the \emph{discrete} Shimura varieties $\Sh_{b_\bullet^{(j)}}$, generate the Tate classes of $H_\et^{d(a_\bullet)}\big(\overline \Sh_{a_\bullet}, \Qlb(\frac{d(a_\bullet)}{2})\big)_{\pi}$ when  $\rho_{\pi_\gothp}(\Frob_{p^f})$ is ``sufficiently generic".
\end{conjecture}

\begin{remark}
\label{R:remark after Conj}
\begin{enumerate}
\item
A key feature of this Conjecture is that the codimension of the cycle map $\pr_{a_\bullet}: Y_j \to \Sh_{a_\bullet}$ is the same as the fiber dimension of $\pr_{b_\bullet^{(j)}} : Y_j \to \Sh_{b_\bullet^{(j)}}$.
\item
It seems that the fiber of $\pr_{b_\bullet^{(j)}}: Y_j \to \Sh_{b_\bullet^{(j)}}$ over a generic point $\eta \in \Sh_{b_\bullet^{(j)}}$ is likely to be isomorphic to a certain ``iterated Deligne--Lusztig variety," that is, a tower of maps $Y_{j,\eta} = Z_\alpha \to \cdots \to Z_0 =\eta$ such that each $Z_i \to Z_{i-1}$ is a fiber bundle with certain Deligne--Lusztig varieties as fibers.

\item
Xinwen Zhu pointed out to us that 
since the universal abelian varieties $\calA_{a_\bullet}$ and $\calA_{b_\bullet}$ are isogenous over each $Y_j$, the union of the images of $Y_1, \dots, Y_{m_\lambda(a_\bullet)}$ on $\Sh_{a_\bullet}$ is contained in the closure of the Newton strata, where the slope is the same as the $\mu$-ordinary slope of the universal abelian varieties on $\Sh_{b_\bullet^{(j)}}$ (for different $j$, they have the same $\mu$-ordinary slopes).
In fact, one should expect the union of images to be the same as the closure of this Newton stratum.

When $\lambda$ is central (i.e. a power of the determinant), Conjecture~\ref{Conj:main} says: \emph{irreducible components of the basic locus of the special fiber of a Shimura variety, generically, contribute to all Tate cycles in the cohomology}.
Implicitly, this means that the dimension of the basic locus is half of the dimension of the Shimura variety if and only if the Galois representations of the Shimura variety has generically non-trivial Tate classes.
Here two appearances of ``generic" both mean that we only consider those $\pi$-isotypic components where the Satake parameter for $\pi_p$ is sufficiently generic as in Conjecture~\ref{Conj:main}(4).
For example, the supersingular locus of Hilbert modular surface at a split prime or the supersingular locus of a Siegel modular variety (over $\QQ$) is not half the dimension. This is related to the fact that the $\pi$-isotypic component of the cohomology of the Shimura varieties  are not expected to have Tate classes, at least when the Satake parameter of $\pi_p$ is sufficiently general.\footnote{The Siegel varieties are Shimura varieties associated to $\mathrm{GSp}_{2g}(\QQ)$. The Langlands dual group is isogenous to $\mathrm{Spin}(2g+1)$ and the associated representation $r_\mu$  is the spin representation, which is minuscule and hence does not contain trivial weight subspace.}

\item
These varieties $Y_j$ may be viewed as Hecke correspondences at $p$ \emph{between the special fibers of two different Shimura varieties} $\Sh_{a_\bullet}$ and $\Sh_{b_\bullet^{(j)}}$. These correspondences certainly cannot be lifted to characteristic zero. We hope that the Conjecture will bring interests into the study of such Hecke correspondences.
\end{enumerate}

\end{remark}
\begin{remark}
\label{R:generalizations}
\begin{enumerate}


\item
The assumption on the decomposition of the place $p$ in $E/\QQ$ and working with unitary Shimura varieties is to simplify our presentation and to get to a situation where most terms can be defined.
We certainly expect the validity of analogous conjectures for the special fibers of Shimura varieties of PEL type or more generally of abelian type
(using the integral model of M. Kisin \cite{kisin}).
This would be a more precise version of the Tate conjecture in the context of special fibers of Shimura varieties:
if $\Sh_G$ and $\Sh_{G'}$ are the special fibers of two unitary Shimura varieties associated to the groups $G$ and $G'$ such that $G(\AAA_f) \simeq G'(\AAA_f)$, then, generically, the cycles on the product $\Sh_G \times \Sh_{G'}$ predicted by the Tate conjectures are likely to be constructed by understanding the ``isogenies" between the corresponding universal abelian varieties, and are closely related to the Newton stratifications of $\Sh_G$ and $\Sh_{G'}$.
In the case of Shimura varieties of abelian type, we expect some technical difficulties to reinterpret the meaning of isogenies of abelian varieties in terms of certain ``$G$-crystals".

For example, consider a real quadratic field $F/\QQ$ in which a prime $p$ is inert. Let $\Sh_G$ denote the special fiber of the Hilbert--Siegel modular variety for $G: = \Res_{F/\QQ} \mathrm{GSp}_{2g}$, with hyperspecial level structure at $p$.
Then by Langlands' prediction of the cohomology of $\Sh_G$, we should look at the representation $r_{\mathrm{spin}}^{\otimes 2}$ of the ``essential part" $\mathrm{Spin}_{2g+1}$ of the Langlands dual group, where $r_{\mathrm{spin}}$ is the $2^g$-dimensional spin representation.\footnote{As pointed out above, we have to work with the Hilbert--Siegel setup as opposed to the usual Siegel setup because $r_{\mathrm{spin}}$ is a minuscule representation.} The central weight space of $r_{\mathrm{spin}}^{\otimes 2}$ has dimension $2^g$. So we expect that the supersingular locus of $\Sh_G$ is the union of $2^g$ collection of varieties parameterized by the discrete Shimura variety $\Sh_{G'}$ where $G'$ is the inner form of $G$ which is split at all finite places and is compact modulo center at both archimedean places.
Unfortunately, the moduli problem that describes $G'$ uses a different division algebra from that describing $G$. We do not know how to interpret the meaning of isogenies of universal abelian varieties in this case, and the method of our paper does not apply directly to this case.

\item

Zhu pointed out to us that even if $p$ is ramified, we should expect Conjecture~\ref{Conj:main} continue to hold for (the special fiber of) the ``splitting models" of Pappas and Rapoport \cite{pappas-rapoport}.
Some evidences of this have already appeared in the case of Hilbert modular varieties; see \cite{RTW,reduzzi-xiao}.
\item
In our setup, we took advantage of many coincidences that ensures that for example the Shimura variety is compact and there is no endoscopy.  It would be certainly an interesting future question to study the case involving Eisenstein series, as well as the case when the representations come from endoscopy transfers.
 
\item
As explained in Remark~\ref{R:remark after Conj}(3), the images of $Y_j$ are expected to form the closure of certain Newton polygon where the slopes are related to $\lambda$.
Conjecture~\ref{Conj:main}(1)--(3) may have a degenerate situation:
when $\sum_i a_i$ is not divisible by $n$, the representation $V_{a_\bullet}$ does not contain a weight corresponding to a power of the determinant (which corresponds to the basic locus).
So our conjecture does not describe the basic locus of $\Sh_{a_\bullet}$, and it is indeed not of half dimension of $\Sh_{a_\bullet}$.

Yet, this basic locus may still have a good description as the union of some fiber bundles over the special fibers of some other Shimura varieties for reductive groups which are \emph{not} quasi-split at $p$.  For example, the supersingular locus of modular curve is related to the Shimura variety associated to the definite quaternion algebra which is ramified at $p$, by a theorem of Serre and Deuring \cite{serre}.  More such examples are given in \cite{tian-xiao1} and \cite{vollaard-wedhorn}.
\end{enumerate}

\end{remark}

\subsection{Known cases of Conjecture~\ref{Conj:main}}

Conjecture~\ref{Conj:main} is largely inspired by the work of second and the third authors \cite{tian-xiao1} and \cite{tian-xiao2}, where we proved the analogous conjecture for the special fibers of the Hilbert modular varieties assuming that $p$ is inert in the totally real field.

Another strong evidence of  Conjecture~\ref{Conj:main} is the work of Vollaard and Wedhorn \cite{vollaard-wedhorn}, where they considered certain stratification of the supersingular locus of the Shimura variety for $GU(1,n-1)$ with $s\in \NN$ at an \emph{inert} prime $p$. 
What concerns us is the case when $n-1$ is even.  In this case, it is hidden in the writing of \cite[Section~6]{vollaard-wedhorn} that one gets a correspondence (in the notation of {\it loc. cit.})
\begin{equation}
\label{E:correspondence VW}
I(\QQ) \big\backslash I(\AAA_f)  \big/ C^pC_p^{(n)} \longleftarrow I(\QQ) \big\backslash N_n \times^{C_p^{(n)}} J(\QQ_p) \times \GG(\AAA_f^{(p)}) \big/ C^p \longrightarrow \calM_{C^p}^{ss} \subset \calM_{C^p}.
\end{equation}
Note that $I(\AAA_f) \simeq \GG(\AAA_f)$.  Here $N_n$ is a certain Deligne--Lusztig variety.
In \cite{vollaard-wedhorn}, the parameterizing space, namely the first term in \eqref{E:correspondence VW}, is interpreted very differently, in terms of Bruhat--Tits building. 
The method of this paper should be applicable to their situation to verify the analogous Conjecture~\ref{Conj:main}.
In fact, in their case, there will be only one collection of cycles as given by \eqref{E:correspondence VW}, but the computation of the intersection matrix (only essentially one entry in this case) of them needs some non-trivial Schubert calculus similar to Section~\ref{S:fundamental-intersection-number}. 

When $n-1$ is odd, the result of \cite{vollaard-wedhorn} is related to the degenerate version of the Conjecture~\ref{Conj:main} in the sense of Remark~\ref{R:generalizations}(4).
\\

The aim of the rest of paper is to provide evidences for Conjecture~\ref{Conj:main} for some large rank groups.
In particular, we will construct cycles in the case of the unitary group  $G(U(r,s)\times U(s,r))$ with $s,r \in \NN$ (Section \ref{Sec:GU rs}). While we expect these cycles to verify Conjecture~\ref{Conj:main}, we do not know how to compute the ``intersection matrix" in general.
Nonetheless, when $r=1$, we are able to make the computation and prove Conjecture~\ref{Conj:main} (with trivial coefficients for the sake of simple presentation) in this case; see Section~\ref{Section:U(1,n)}--\ref{Sect:intersection-matrix}.
We point out that our method should be applicable to many other examples, and even in general reduce Conjecture~\ref{Conj:main} to a question of combinatorial nature. This combinatorics problem is the heart of the question.
In the Hilbert case \cite{tian-xiao2}, we model combinatorics question by  the so-called periodic semi-meander (for $\GL_2$).  The generalization of the usual (as opposed to periodic) semi-meander to other groups has been introduced; see \cite{FKK} for the corresponding references. 
The straightforward generalization to the periodic case does seem to agree with some of our computations with small groups.  Nonetheless, the corresponding Gram determinant formula seems to be extremely difficult. Even in the non-periodic case, we only know it for a special case; see \cite{difrancesco}. 

We also mention that in a very recent work \cite{xiao-zhu} of Zhu and the last author, we relate  Conjecture~\ref{Conj:main} with the geometric Satake theory \cite{zhu} of Zhu in mixed characteristic, and we proved many new cases of Conjecture~\ref{Conj:main}.

\section{Preliminaries on Dieudonn\'e modules and deformation theory}\label{S:preliminary}

We first introduce the basic tools that we will use in this paper.

\subsection{Notation}\label{S:notation}
Recall that we have an isomorphism
\[
\calO_D \otimes_\ZZ \ZZ_{p^f}
\cong \bigoplus_{i=1}^f
\big(\calO_D \otimes_{\calO_E, q_i}\ZZ_{p^f} \oplus \calO_D \otimes_{\calO_E,\bar q_i}\ZZ_{p^f} \big) \simeq \bigoplus_{i=1}^f \big(\rmM_n(\ZZ_{p^f}) \oplus \rmM_n(\ZZ_{p^f})\big).
\]
Let $S$ be a locally noetherian $\ZZ_{p^f}$-scheme. An $\cO_{D}\otimes_{\ZZ}\cO_{S}$-module $M$ admits a canonical decomposition
\[
M=\bigoplus_{i=1}^f\big( M_{q_{i}}\oplus M_{\bar{q}_i}\big),
\]
where $M_{q_i}$ (resp. $M_{\bar{q}_{i}}$) is  the direct summand of $M$ on which $\cO_E$ acts via $q_i$ (resp. via $\bar{q}_i$).
 Then each $M_{q_i}$ has a natural action by $\rmM_{n}(\cO_S)$.
Let $\gothe$ denote the element of $\rmM_{n}(\calO_S)$  whose $(1,1)$-entry is $1$ and whose other entries are $0$. We put $M^{\circ}_{i}:=\gothe M_{q_i}$, and call it the \emph{reduced part} of $M_{q_i}$.

Let $A$ be an $fn^2$-dimensional abelian variety over an $\FF_{p^f}$-scheme $S$, equipped with an $\calO_D$-action.
The de Rham homology $H^\dR_1(A/S)$ 
has a Hodge filtration
\[
0\ra \omega_{A^\vee/S}\ra H^\dR_1(A/S)\ra \Lie_{A/S}\ra 0,
\]
compatible with the natural  action of $\calO_D \otimes_\ZZ \calO_S$ on $H^{\dR}_1(A/S)$.
We call  $H^\dR_1(A/S)^\circ_{i}$ (resp. $\omega^\circ_{A^\vee/S, i}$, $\Lie^\circ_{A/S, i}$) the \emph{reduced de Rham homology} of $A/S$ (resp. the \emph{reduced invariant $1$-forms} of $A^\vee/S$, the \emph{reduced Lie algebra} of $A/S$) at $q_i$.
 In particular, the former is a locally free $\calO_S$-module of rank $n$ and the latter is a subbundle\footnote{Here and after, by a subbundle of a locally free coherent sheaf, we mean a locally free coherent sheaf that is Zariski locally a direct factor.} of the former; when $A \to S$ satisfies the moduli problem in Subsection~\ref{S:defn of Shimura var}, $\omega_{A^\vee/S,i}^\circ$ is locally free of rank $a_i$.

The Frobenius morphism $ A \to A^{(p)}$ induces a natural homomorphism
\[
V: H^\dR_1(A/S)^\circ_i \longrightarrow H^\dR_1(A/S)^{\circ, (p)}_{i-1},
\]
where the index $i$ is considered as an element of $\ZZ/f\ZZ$, and the superscript ``$(p)$'' means the pullback via the absolute Frobenius of $S$.
The image of $V$
is exactly $\omega^{\circ,(p)}_{A^\vee/S, i-1}$.
Similarly, the
Verschiebung morphism $A^{(p)} \to A$ induces a natural homomorphism\footnote{The notation $F$ for Frobenius was also used to denote the real quadratic field. But we think the chance for confusion is minimal.}
\[
F: H^\dR_1(A/S)^{\circ, (p)}_{i-1} \longrightarrow H^\dR_1(A/S)^\circ_i.
\]
We have $\Ker(F) = \Image(V)$ and $\Ker(V) = \Image(F)$.

When $S =  \Spec(k)$ with $k$ a perfect field containing $\F_{p^f}$, let $W(k)$ denote the ring of Witt vectors in $k$. Let  $\tcD(A)$ denote  the (covariant) Dieudonn\'e module associated to the $p$-divisible group of  $A$.
 This is a free $W(k)$-module of rank $2fn^2$ equipped with a $\Frob$-linear action of  $F$ and a $\Frob^{-1}$-linear action of $V$ such that $FV=VF = p$.
  The $\cO_D$-action on $A$ induces a natural action of $\cO_D$ on $\tcD(A)$ that commutes with $F$ and $V$.
   Moreover, there is a canonical isomorphism  $\tilde \calD(A)/p\tilde \calD(A)\cong H^\dR_1(A/k)$ compatible with all structures on both sides.
For each $i\in \Z/f\Z$, we have the reduced part $\tcD(A)^{\circ}_i:=\gothe \tcD(A)_{q_i}$. The Verschiebung and the Frobenius induce natural maps
$$
V: \tilde \calD(A)^\circ_i \longrightarrow \tilde \calD(A)^\circ_{i-1},\quad F: \tilde \calD(A)^\circ_i \longrightarrow \tilde \calD(A)^\circ_{i+1}.
$$
 Note that
$
\tcD(A)_{q_i}=(\tcD(A)^{\circ}_{i})^{\oplus n},
$
 and $\oplus_{i\in \Z/f\Z}\tcD(A)_{q_i}$ is the covariant Dieudonn\'e module of the $p$-divisible group $A[\gothp^{\infty}]$.
\\

For any $fn^2$-dimensional abelian variety $A'$ over $k$ equipped with an $\calO_D$-action, an $\calO_D$-equivariant isogeny $A' \to A$ induces a morphism $\tilde \calD(A')^\circ_i \to \tilde \calD(A)^\circ_i$ compatible with the actions of $F$ and $V$.
Conversely, we have the following.

\begin{proposition}\label{P:abelian-Dieud}
Let $A$ be an abelian variety of dimension $fn^2$ over prefect field $k$ which contains $\FF_{p^f}$, equipped with an $\cO_{D}$-action and  an $\cO_D$-compatible prime-to-$p$ polarization $\lambda$.
Suppose given an integer $m\geq 1$ and a $W(k)$-submodule   $\tcE_i\subseteq \tcD(A)^{\circ}_i$ for each  $i\in \Z/f\Z$ such that
\begin{equation}
\label{E:condition for getting another abelian varieties}
p^m\tcD(A)^{\circ}_{i}\subseteq \tcE_i, \quad F(\tcE_i)\subseteq \tcE_{i+1}, \quad \textrm{and} \quad  V(\tcE_i)\subseteq \tcE_{i-1}.
\end{equation}
Then there exists a unique abelian variety $A'$ over $k$ (depending on $m$) equipped with  an $\cO_{D}$-action,  a prime-to-$p$ polarization $\lambda'$,  and an $\cO_{D}$-equivariant $p$-isogeny $\phi: A'\ra A$ such that the natural inclusion $\tcE_i\subseteq \tcD(A)_i^\circ$ is naturally identified with the map
$\phi_{*,i}\colon\tcD(A')^{\circ}_i\ra \tcD(A)_i^{\circ}$
induced by $\phi$ and such that $\phi^\vee\circ\lambda\circ\phi=p^m\lambda'$.
Moreover, we have
\begin{enumerate}
\item

If $\dim \omega^\circ_{A^\vee/k, i} = a_i$ and $\len_{W(k)} \big( \tilde \calD(A)^\circ_i / \tilde \calE_i\big) = \ell_i$ for $i \in \ZZ/f\ZZ$, then
\begin{equation}
\label{E:dimension of new differentials}
\dim \omega^\circ_{A'^\vee/k, i} =a_i + \ell_i -\ell_{i+1}.
\end{equation}

\item
If $A$ is equipped with a prime-to-$p$ level structure $\eta$ (in the sense of Subsection~\ref{S:defn of Shimura var}(3)), then there exists a unique prime-to-$p$ level structure $\eta'$ on $A'$ such that $\eta=\phi\circ\eta'$.
\end{enumerate}
\end{proposition}
\begin{proof}
 By Dieudonn\'e theory,   the Dieudonn\'e submodule
$$
\bigoplus_{i\in \Z/f\Z}(\tcE_i/p^m\tcD(A)^{\circ}_{i})^{\oplus n}
\subseteq\bigoplus_{i\in \Z/f\Z}
\big(\tcD(A)^{\circ}_i/p^m\tcD(A)^{\circ}_i\big)^{\oplus n}
$$
corresponds to a closed  subgroup scheme  $H_{\gothp}\subseteq A[\gothp^m]$.
The prime-to-$p$ polarization $\lambda$ induces a perfect pairing
\[
\langle\bullet,\bullet \rangle_{\lambda}\colon  A[\gothp^{m}]\times A[\bar{\gothp}^{m}]\longrightarrow \mu_{p^m}.
\]
Let $H_{\bar \gothp}=H^{\perp}_{\gothp}\subseteq A[\bar{\gothp}^m]$ denote the orthogonal complement of $H_\gothp$.
 Put $H_{p}=H_{\gothp}\oplus H_{\bar\gothp}$.
  Let $\psi: A\ra A'$ be  the canonical quotient with kernel $H_p$, and $\phi: A'\ra A$ be the quotient with kernel $\psi(A[\gothp^m])$ so that $\psi\circ\phi=p^m\id_{A'}$ and $\phi\circ\psi=p^m\id_A$.
  By construction, $H_p\subseteq A[p^m]$ is a maximal totally isotropic subgroup.
  By \cite[\S 23, Theorem 2]{mumford}, there is a prime-to-$p$ polarization $\lambda'$ on $A'$ such that $p^m\lambda=\psi^\vee\circ\lambda'\circ\psi$.
It follows also that $p^m\lambda'=\phi^\vee\circ\lambda\circ\phi$.
The fact that $\phi_{*,i}:\tcD(A')^{\circ}_i\ra \tcD(A)^{\circ}_i$ is identified with the natural inclusion $\tcE_i\subseteq \tcD(A)^{\circ}_i$ follows from the construction.
The existence and uniqueness of tame level structure is clear.
The dimension of the differential forms can be computed as follows:
\begin{align*}
&\dim_k \omega^\circ_{A'^\vee/k, i} = \dim_k \frac{V(\tcD(A')^\circ_{i+1})}{p\tcD(A')^\circ_i} = \dim_k \frac{V(\tilde \calE_{i+1})}{p\tilde \calE_i}\\
 = &\dim_k \frac{V(\tcD(A)^\circ_{i+1})}{p\tcD(A)^\circ_i} - \len_{W(k)} \frac{V(\tcD(A)_{i+1}^\circ)}{V(\tilde \calE_{i+1})} + \len_{W(k)} \frac{p\tcD(A)^\circ_i}{p\tilde \calE_i} = a_i-\ell_{i+1} + \ell_i .\qedhere
\end{align*}
\end{proof}

\subsection{Deformation theory}
\label{S:deformation}
We shall frequently use Grothendieck--Messing deformation theory to compare the tangent spaces of moduli spaces.  We make this explicit in our setup.

Let $ \hat R$ be a noetherian $\FF_{p^f}$-algebra and $ \hat I \subset \hat R$ an ideal such that $\hat I^2=0$.  Put $R=  \hat R/ \hat I$.
 Let $\scrC_{\hat R}$ denote the category of tuples $(\hat A, \hat \lambda,\hat  \eta)$, where $\hat  A$ is an $fn^2$-dimensional abelian variety over $ \hat R$ equipped with an $\calO_D$-action, $ \hat  \lambda$ is a polarization on $\hat A$ such that the Rosati involution induces the $*$-involution on $\calO_D$, and $ \hat \eta$ is a level structure as in Subsection~\ref{S:defn of Shimura var}(3). We define $\scrC_{R}$ in the same way.
For an object $(A,\lambda,\eta)$ in the category $\scrC_{R}$, let $H^{\cris}_1(A/\hat R)$ be the evaluation of the first relative crystalline homology (i.e. dual crystal of the first crystalline cohomology) 
of $A/R$ at the divided power thickening $\hat R\ra R$, and $H^{\cris}_1(A/\hat R)^{\circ}_{i}:=\gothe H^{\cris}_1(A/\hat R)_{q_i}$ be the $i$-th reduced part.
  We denote by  $\Def( R,\hat R)$ the category of tuples $(A, \lambda,\eta, (\hat \omega^{\circ}_i)_{i=1,\dots, f})$, where $(A,\lambda,\eta)$ is an object in $\scrC_{R}$, and $\hat \omega^{\circ}_{i}\subseteq H^{\cris}_1(A/\hat R)_{i}^{\circ}$ for each $i\in \Z/f\Z$ is a subbundle that lifts $\omega^{\circ}_{A^\vee/R,i}\subseteq H^{\dR}_1(A/R)^{\circ}_i$.
   The following is a combination of Serre--Tate  and Grothendieck--Messing deformation theory.

\begin{theorem}[Serre--Tate, Grothendieck--Messing]
\label{T:STGM}
The  functor $(\hat A, \hat \lambda, \hat \eta)\mapsto (\hat A \otimes_{\hat R}R , \lambda,\eta,\omega_{\hat A^\vee/\hat R,i}^{\circ})$,  where   $\lambda$ and $\eta$ are the natural induced polarization and level structure on $\hat A\otimes_{\hat R}R$, is an equivalence of categories between $\scrC_{\hat R}$ and $\Def(R,\hat R)$.
\end{theorem}
\begin{proof}
The main theorem of the crystalline deformation theory (cf. \cite[pp. 116--118]{grothendieck},   \cite[Chap. II \S 1]{mazur-messing}) says that the category $\scrC_{\hat R}$ is equivalent to the category of objects $(A, \lambda, \eta)$ in $\scrC_R$ together with a lift of $\omega_{A^\vee/R} \subseteq H_1^\cris(A/R)$ to a subbundle $\hat \omega$ of $H_1^\cris(A/\hat R)$, such that $\hat \omega$ is stable under the induced $\calO_D$-action and is isotropic for the pairing on $H_1^\cris(A/\hat R)$ induced by the polarization $\lambda$.
But the additional information $\hat \omega$ is clearly equivalent to the subbundles $\hat \omega_i^\circ \subseteq H_1^\cris(A/\hat R)_i^\circ$ lifting $\omega^\circ_{A^\vee/R,i}$.
\end{proof}

\begin{corollary}
\label{C:tangent space of Sh}
If $\calA_{a_\bullet}$ denotes the universal abelian variety over $\Sh_{a_\bullet}$, then the tangent space $T_{\Sh_{a_\bullet}}$ of $\Sh_{a_\bullet}$ is
\[
\bigoplus_{i=1}^f \Lie_{\calA^\vee_{a_\bullet}/\Sh_{a_\bullet},i}^\circ \otimes \Lie_{\calA_{a_\bullet}/\Sh_{a_\bullet},i}^\circ.
\]
\end{corollary}
\begin{proof}
Even though this is a well-known statement often referred to as the Kodaira--Spencer isomorphism (e.g. \cite[Proposition~2.3.4.2]{lan}), we include a short proof, as the proof serves as a toy model of many arguments later.
Let $\hat R$ be a noetherian $\FF_{p^f}$-algebra and $\hat I \subset \hat R$ an ideal such that $\hat I^2=0$; put $R =\hat R/\hat I$.
By Theorem~\ref{T:STGM}, to lift an $R$-point $(A, \lambda, \eta)$ of $\Sh_{a_\bullet}$ to an $\hat R$-point, it suffices to lift, for $i=1, \dots, f$, the differentials $\omega^\circ_{A^\vee,i} \subseteq H_1^\cris(A/R)_i^\circ$ to a subbundle $\hat \omega_i \subseteq H_1^\cris(A/\hat R)_i^\circ$.
Such lifts form a torsor for the group
\[
\Hom_{R}\big(\omega_{A^\vee/R,i}^\circ, \Lie_{A/R,i}^\circ \big) \otimes_{R} \hat I.
\]
It follows from this
\[
T_{\Sh_{a_\bullet}} \cong \bigoplus_{i=1}^f \cHom\big(\omega_{\calA^\vee_{a_\bullet}/\Sh_{a_\bullet},i}^\circ , \Lie_{\calA_{a_\bullet}/\Sh_{a_\bullet},i}^\circ \big) \cong \bigoplus_{i=1}^f \Lie_{\calA^\vee_{a_\bullet}/\Sh_{a_\bullet},i}^\circ \otimes \Lie_{\calA_{a_\bullet}/\Sh_{a_\bullet},i}^\circ.
\]
Note that this proof also shows that $\Sh_{a_\bullet}$ is smooth.
\end{proof}

\subsection{Notation in the real quadratic case}\label{S:notation-real-quadratic}
For the rest of the paper, we assume $f=2$ so that  $F$ is a real quadratic field in which $p$ is inert.
For non-negative  integers $r \leq s$ such that $n=r+s$, we denote by $G_{r,s}$ the algebraic group previously denoted by $G_{a_{\bullet}}$ with $a_1=r$ and $a_2=s$; in particular, $G_{r,s}(\RR)=G(U(r,s)\times U(s,r))$.
If $r', s'$ is another pair of non-negative integers such that $n=r'+s'$ and $r'\leq s'$, Lemma~\ref{L:change-signature} gives an isomorphism $G_{r,s}(\AAA^{\infty})\simeq G_{r',s'}(\AAA^{\infty})$.

Let $\cSh_{r,s}$ be the Shimura variety  over $\Z_{p^2}$ attached to $G_{r,s}$ defined in Subsection~\ref{S:defn of Shimura var} of some fixed sufficiently small prime-to-$p$ level  $K^p\subseteq G_{r,s}(\AAA^{\infty,p})$.  Let  $\Sh_{r,s}$ denote  its special fiber over $\FF_{p^2}$.  Let $\calA = \calA_{r,s}$ denote the universal abelian variety over $\Sh_{r,s}$. It is a $2n^2$-dimensional abelian variety, equipped with an action of $\calO_D$ and a prime-to-$p$ polarization $\lambda_{\calA}$.  Moreover, $\omega_{\calA^\vee/\Sh_{r,s},1}^\circ$ (resp. $\omega_{\calA^\vee/\Sh_{r,s},2}^\circ$) is a locally free module over $\Sh_{r,s}$ of rank $r$ (resp. rank $s$).

\begin{remark}
\label{R:Sh0n supersingular}
When $r=0$ and $s=n$, the universal abelian variety $\calA =\calA_{0,n}$ over $\Sh_{0,n}$ is supersingular. Indeed, for each $\overline \FF_p$-point $z$ of $\Sh_{0,n}$, the Kottwitz condition implies that the Frobenius induces \emph{isomorphisms}
\[
\tilde \calD(\calA_z)^\circ_1 \xrightarrow{\;F\;} \tilde \calD(\calA_z)^\circ_2 \xrightarrow{\; F \;} p\tilde \calD(\calA_z)^\circ_1.
\]
In particular, $\frac 1pF^2$ induces a $\sigma^2$-linear automorphism of $\tilde \calD(\calA_z)^\circ_1$.
By Hilbert's Theorem 90, there exists a $\ZZ_{p^2}$-lattice $\LL$ of $\tilde \calD(\calA_z)^\circ_1$ that is invariant under the action of $\frac 1p F^2$, in other words, $F^2$ acts by multiplication by $p$ for a basis chosen from this lattice. It follows that all  slopes of the Frobenius  on $\tilde \calD(\calA_z)$ are $\frac 12$, and hence $\calA_z$ is supersingular.
\end{remark}

\section{The case of $G(U(1,n-1)\times U(n-1,1))$}\label{Section:U(1,n)}

We will verify Conjecture~\ref{Conj:main} for $\Sh_{1,n-1}$, namely the existence of some cycles $Y_j$ having morphisms to both $\Sh_{0,n}$ and $\Sh_{1,n-1}$ and generating Tate classes of $\Sh_{1,n-1}$ under a certain genericity hypothesis on the Satake parameters.
 We always fix an isomorphism $G_{1,n-1}(\AAA^{\infty})\simeq G_{0,n}(\AAA^{\infty})$, and write  $G(\AAA^{\infty})$ for either group.
 
\begin{notation}
For a smooth variety $X$ over $\F_{p^2}$, we denote by $T_X$ the tangent bundle of $X$, and for a locally free $\cO_X$-module $M$, we put $M^*=\cHom_{\cO_X}(M,\cO_X)$.
\end{notation}

\subsection{Cycles on $\Sh_{1,n-1}$}\label{S:cycles}
For each integer  $j$ with $1\leq j\leq n $, we first define the variety $Y_j$ we briefly mentioned in the introduction.   Let $Y_{j}$ be the moduli space over $\FF_{p^2}$ that associates to  each  locally noetherian $\FF_{p^2}$-scheme $S$, the set of isomorphism classes of tuples  $(A,\lambda,\eta, B,\lambda',\eta',\phi)$, where 
\begin{itemize}
\item
$(A,\lambda,\eta)$ is an $S$-point of $\Sh_{1,n-1}$,  
\item
$(B,\lambda',\eta')$ is an $S$-point of $\Sh_{0,n}$,  and 
\item
$\phi: B\ra A$ is an $\calO_D$-equivariant isogeny whose kernel is contained in $B[p]$,
\end{itemize}
such that 
\begin{itemize}
\item
$p\lambda'=\phi^\vee \circ\lambda\circ \phi$, 
\item
$\phi\circ\eta'=\eta$, and  \item
the cokernels of the maps
 \[
 \phi_{*,1}: H^\dR_1(B/S)^\circ_1 \to H^\dR_1(A/S)^\circ_1 \quad\textrm{and}\quad
\phi_{*,2}: H^\dR_1(B/S)^\circ_2 \to H^\dR_1(A/S)^\circ_2
\]
are locally free $\cO_S$-modules of rank $j-1$ and $j$, respectively.
\end{itemize}

There is a unique isogeny $\psi: A\ra B$ such that $\psi\circ\phi=p \cdot \id_{B}$ and $\phi\circ\psi=p \cdot \id_{A}$.
 We have
 $$
 \Ker(\phi_{*,i})=\im(\psi_{*,i}) \quad \textrm{and} \quad \Ker(\phi_{*,i})=\im(\psi_{*,i}),
 $$
 where $\psi_{*,i}$ for $i=1,2$ is the induced homomorphism on the reduced de Rham homology in the evident sense.
This moduli space $Y_j$ is represented by a scheme of finite type over $\FF_{p^2}$.
We have a natural diagram of morphisms:
\begin{equation}\label{D:two-projection}
\xymatrix{& Y_j\ar[rd]^{\pr'_j}\ar[ld]_{\pr_j}\\
\Sh_{1,n-1} &&\Sh_{0,n},
}
\end{equation}
where $\pr_j$ and $\pr'_j$  send  a tuple $(A,\lambda,\eta, B,\lambda',\eta',\phi)$ to $(A, \lambda, \eta)$ and to $(B, \lambda', \eta')$, respectively.
Letting $K^p$ vary, we see easily that both $\pr_{j}$ and $\pr'_{j}$ are equivariant under prime-to-$p$ Hecke actions given by the double cosets $K^p\backslash G(\AAA^{\infty,p})/K^p$.

\subsection{Some auxiliary moduli spaces}
The moduli problem for $Y_j$ is slightly complicated.  We will introduce a more explicit moduli space $Y'_j$ below and then show they are isomorphic.

Consider the functor $\uY'_j$ which associates to each locally noetherian $\F_{p^2}$-scheme $S$, the set of isomorphism classes of tuples $(B, \lambda', \eta', H_1,H_2)$, where
\begin{itemize}
\item $(B,\lambda',\eta')$ is an $S$-valued point of $\Sh_{0,n}$;

\item $H_1\subset H^{\dR}_1(B/S)^{\circ}_{1}$ and  $H_2\subset H^{\dR}_1(B/S)_{2}^{\circ}$ are  $\calO_S$-subbundles of rank $j$ and $j-1$ respectively such that
\begin{equation}\label{Equ:condition-Yj}
V^{-1}(H_2^{(p)})\subseteq H_1,\quad  H_2\subseteq F(H_1^{(p)}).
\end{equation}
Here, $F: H^{\dR}_1(B/S)^{\circ,(p)}_1\xra{\sim} H^{\dR}_1(B/S)_{2}^{\circ}$ and $V:H^{\dR}_1(B/S)^{\circ}_{1}\xra{\sim} H^{\dR}_1(B/S)_{2}^{\circ, (p)} $ are respectively the Frobenius and Verschiebung homomorphisms, which are actually isomorphisms because of the signature condition on $\Sh_{0,n}$.
\end{itemize}
It follows from the moduli problem that the  quotients $H_1/V^{-1}(H_2^{(p)})$ and $F(H_1^{(p)})/H_2$ are both locally free $\cO_{S}$-modules of rank one.

There is a natural projection  $\pi'_j\colon \uY'_j\ra \Sh_{0,n}$ given by $(B,\lambda',\eta',H_1,H_2)\mapsto (B,\lambda',\eta')$.

\begin{proposition}\label{P:Yj-prime}
The functor $\uY'_j$ is representable by a scheme $Y'_j$ smooth and projective over $\Sh_{0,n}$ of  dimension $n-1$.
 Moreover, if $(\calB,\lambda',\eta',\cH_1,\cH_2)$ denotes the universal object over $Y'_j$, then  the tangent bundle of $Y'_j$ is
\[
T_{Y'_j}\cong \big((\cH_1/V^{-1}(\cH_2^{(p)})\big)^*\otimes \big (H^{\dR}_1(\calB/\Sh_{0,n})^{\circ}_1/\cH_1\big) \oplus \bigl(\cH^*_{2}\otimes F(\cH_1^{(p)})/\cH_2\bigr).
\]
\end{proposition}
\begin{proof}
For each integer $m$ with $0\leq m\leq n$ and $i=1,2$, let $\Gr(H^{\dR}_1(\calB/\Sh_{0,n})^{\circ}_{i}, m)$ be the Grassmannian scheme over $\Sh_{0,n}$ that parametrizes subbundles of  the universal de Rham homology $H^{\dR}_1(\calB/\Sh_{0,n})^{\circ}_{i}$ of rank $m$.
Then $\underline Y'_j$ is a closed subfunctor of the product of the Grassmannian schemes
\[
\Gr\big(H^{\dR}_1(\calB/\Sh_{0,n})^{\circ}_{1}, j\big)\times \Gr\big(H^{\dR}_1(\calB/\Sh_{0,n})^{\circ}_{2}, j-1\big).
\]
The representability of $\uY'_j$ follows.  Moreover, $Y'_j$ is projective.

We show now that the structural map $\pi'_j: Y'_{j}\ra \Sh_{0,n}$ is smooth of relative dimension $n-1$.
Let $S_0\hra S$ be an immersion of locally noetherian $\F_{p^2}$-schemes with ideal sheaf $I$ satisfying $I^2=0$. Suppose we are given a commutative diagram
\[
\xymatrix{
S_0\ar[r]^{g_0}\ar[d] & Y'_j \ar[d]^{\pi'_j}\\
S\ar[r]^-{h}\ar@{-->}^{g}[ru] &\Sh_{0,n}
}
\]
with solid arrows.
 We have to show that, locally for the Zariski topology on $S_0$, there is  a morphism $g:S\to Y'_j$ making the diagram commute.
  Let $B$ be the abelian scheme over $S$ given by $h$, and $B_0$ be the base change to $S_0$. The morphism $g_0$ gives rises to subbundles $\overline{H}_1\subset H^{\dR}_1(B_0/S_0)^{\circ}_1$ and $\overline{H}_2\subset H^{\dR}_1(B_0/S_0)^{\circ}_2$ with
\[
F(\overline{H}_1^{(p)})\supset \overline{H}_2, \quad V^{-1}(\overline{H}_2^{(p)})\subset \overline{H}_1.
\]
Finding $g$ is equivalent to finding a subbundle $H_i\subset  H^{\dR}_1(B/S)^{\circ}_i$   which lifts each $\overline{H}_i$ for $i=1,2$ and satisfies \eqref{Equ:condition-Yj}; this is certainly possible when passing to small enough affine open subsets of $S_0$.
Thus  $\pi'_j: Y_{j}'\ra \Sh_{0,n}$ is formally smooth, and hence smooth.
 We note that $F_{S}^*: \cO_S\ra \cO_{S}$ factors through $\cO_{S_0}$.
Hence $V^{-1}(H_2^{(p)})$ and $F(H_1^{(p)})$ actually depend only on $\overline{H}_1, \overline{H}_2$, but not on the lifts $H_1$ and $H_2$.
 Therefore, the possible lifts $H_2$ form a torsor under the group
 $$
 \calH om_{\cO_{S_0}}\big(\overline{H}_2, F(\overline{H}^{(p)}_1)/\overline{H}_2\big)
 \otimes_{\cO_{S_0}} I,
 $$
 and similarly the possible lifts $H_1$ form a torsor under the group
 $$
 \calH om_{\cO_{S_0}}\big(\overline{H}_1/V^{-1}(\overline{H}_2^{(p)}),H^{\dR}_1(B_0/S_0)^{\circ}_1/\overline{H}_1\big)\otimes_{\cO_{S_0}} I.
 $$
 To compute the tangent bundle $T_{Y'_j}$, we take $S=\Spec(\cO_{S_0}[\epsilon]/\epsilon^2)$ and $I=\epsilon \cO_{S}$. 
 The morphism $g_0: S_0\ra Y'_j$ corresponds to an $S_0$-valued point of $Y'_j$, say $y_0$. 
 Then  the possible liftings $g$ form the tangent space $T_{Y'_j}$ at $y_0$, denote by $T_{Y'_j,y_0}$. 
  The discussion above shows that  
\[
T_{Y'_{j},y_0}\cong \cHom_{\cO_{S_0}} \big(\overline{H}_2, F(\overline{H}^{(p)}_1)/\overline{H}_2\big )\oplus \cHom_{\cO_{S_0}}\big(\overline{H}_1/V^{-1}(\overline{H}_2^{(p)}),H^{\dR}_1(B_0/S_0)^{\circ}_1/\overline{H}_1\big),
\]
which is certainly a vector bundle over $S_0$ of rank $j-1+(n-j)=n-1$. Applying this to the universal case when $g_0: S_0\ra Y'_j$ is the identity morphism, the formula of the tangent bundle follows.
\end{proof}

\begin{remark}
Let $(B,\lambda',\eta',H_1,H_2)$ be an $S$-point of $Y'_j$.

(a) If $j=n$, $H_1$ has to be  $H^{\dR}_1(B/S)^{\circ}_{1}$, and $H_2$ is a hyperplane of $H^{\dR}_1(B/S)^{\circ}_2$. Condition~\eqref{Equ:condition-Yj} is trivial.  In this case, $Y'_{n}$ is the projective space over $\Sh_{0,n}$ associated to $H^{\dR}_1(\calB/\Sh_{0,n})^{\circ}_2$, where $\calB$ is the universal abelian scheme over $\Sh_{0,n}$.  So it is geometrically a union of copies of $\PP^{n-1}_{\overline \FF_p}$.

(b) If $j=1$, then $H_1$ is a line in $H^{\dR}_1(B/S)^{\circ}_1$ and $H_2=0$.  So $Y'_{1}$ is the projective space over $\Sh_{0,n}$ associated to $(H^{\dR}_1(\calB/\Sh_{0,n})^{\circ}_1)^*$.

(c) If $j=2$, $H_2\subseteq H^{\dR}_1(B/S)^{\circ}_2$ is a line, and $H_1\subseteq H^{\dR}_1(B/S)^{\circ}_1$ is a subbundle of rank 2 such that $F(H_1^{(p)})$ contains both $H_2$ and $F\big(V^{-1}(H_2^{(p)})^{(p)}\big)$. Therefore,
if $H_2\neq F(V^{-1}(H_2^{(p)}))^{(p)}$, $H_1$ is determined up to Frobenius pullback.
 If $H_2 = F\big(V^{-1}(H_2^{(p)})^{(p)}\big)$, then $H_1$ could be  any rank 2 subbundle  containing $V^{-1}(H_2^{(p)})$.

 We fix a geometric point $z=(B,\lambda',\eta')\in \Sh_{0,n}(\Fpb)$.
 It is possible to find good basis for $H^{\dR}_1(B/\Fpb)^{\circ}_{1}$ and $H^{\dR}_1(B/\Fpb)^{\circ}_{2}$ such that both $F, V: H^{\dR}_1(B/\Fpb)^{\circ}_{1}\ra H^{\dR}_1(B/\Fpb)^{\circ}_{2}$ are given by the identity matrix.
  With these choices, we may identify the fiber $Y'_{2, z}=\pi'^{-1}_2(z)$ with a closed subvariety of
\[
 \Gr(\Fpb^{n}, 2)\times \Gr(\Fpb^{n},1).
\]
Moreover, one may equip $\Gr(\Fpb^{n},1)\cong \PP^{n-1}_{\Fpb}$ with an $\F_{p^2}$-rational structure such that $H_2 = F\big(V^{-1}(H_2^{(p)})^{(p)}\big)$ if and only if $[H_2]\in \PP^{n-1}_{\Fpb}$ is an $\F_{p^2}$-rational point.
 So $Y'_{2, z}$ is isomorphic to a ``Frobenius-twisted'' blow-up of $\PP^{n-1}_{\Fpb}$ at all of its $\F_{p^2}$-rational points.
  Here, ``Frobenius-twisted'' means that each irreducible component of the exceptional divisor has multiplicity $p$. For instance, when $n=3$, each  $Y_{2, z}$ is isomorphic to the closed subscheme of $\PP^2_{\Fpb}\times \PP^2_{\Fpb}$ defined by
\[\begin{cases}
a_1b_1^p+a_2b_2^p+a_3b_3^p=0\\
a_1^pb_1+a_2^pb_2+a_3^pb_3=0,
\end{cases}\]
where $(a_1:a_2:a_3)$ and $(b_1:b_2:b_3)$ are  the homogeneous coordinates on the two copies of $\PP^2$.

\end{remark}

\begin{lemma}\label{L:omega-F}
Let $(A,\lambda, \eta, B,\lambda',\eta',\phi)$ be an $S$-point of $Y_{j}$.
 Then the image of $\phi_{*,1}$ contains both $\omega^{\circ}_{A^\vee/S,1}$ and $F\big(H^{\dR}_1(A/S)^{\circ,(p)}_{2}\big)$, and the image of $\phi_{*,2}$ is contained in both $\omega^{\circ}_{A^\vee/S,2}$ and $F\big(H^{\dR}_1(A/S)^{\circ,(p)}_{1}\big)$.
\end{lemma}
\begin{proof}
By the functoriality, $\phi_{2,*}$ sends $\omega^{\circ}_{B^\vee/S,2}$ to $\omega^{\circ}_{A^\vee/S,2}$.  Since $\omega^{\circ}_{B^\vee/S,2}=H^{\dR}_1(B/S)^{\circ}_{2}$ by the Kottwitz determinant condition, it follows that $\im(\phi_{*,2})$ in contained in  $\omega^{\circ}_{A^\vee/S,2}$.
Similar arguments by considering  $\psi_{*,1}$ shows that $\omega^{\circ}_{A^\vee/S,1}\subseteq \Ker(\psi_{*,1})=\im(\phi_{*,1})$.
 The fact that $\im(\phi_{*,2})$ is contained in $F\big(H^{\dR}_1(A/S)^{\circ,(p)}_{1}\big)$  follows from the commutative diagram
\begin{equation}\label{D:Frob-phi}
\xymatrix{H^{\dR}_1(B/S)^{\circ,(p)}_1\ar[d]_{F}^{\cong}\ar[r]^{\phi_{*,1}^{(p)}} & H_1^{\dR}(A/S)^{\circ,(p)}_1\ar[d]^{F}\\
H_1^{\dR}(B/S)^{\circ}_2 \ar[r]^{\phi_{*,2}} &H_1^{\dR}(A/S)^{\circ}_2,
}
\end{equation}
  and the fact that the left vertical arrow is an isomorphism.
   The inclusion $F\big(H^{\dR}_1(A/S)^{\circ,(p)}_{2}\big)\subseteq \im(\phi_{*,1})=\Ker(\psi_{*,1})$ can be proved similarly using the functoriality of Verschiebung homomorphisms.
\end{proof}

\subsection{A morphism from $Y_j$ to $Y'_j$}
There is a natural  morphism $\alpha: Y_j\ra Y'_j$ for $1\leq j\leq n$ defined as follows. For a locally noetherian $\FF_{p^2}$-scheme $S$ and an $S$-point $(A,\lambda, \eta,B,\lambda',\eta',\phi)$ of $Y_j$,  we define
\begin{equation}
\label{E:definition of H12}
H_1:=\phi_{*,1}^{-1}(\omega^{\circ}_{A^\vee/S,1})\subseteq H^{\dR}_1(B/S)^{\circ}_{1},
\quad \text{and}\quad
 H_2:=\psi_{*,2}(\omega^{\circ}_{A^\vee/S,2})\subseteq H^{\dR}_{1}(B/S)^{\circ}_{2}.
\end{equation}
In particular, $H_1$ and  $H_2$ are $\calO_S$-subbundles of rank $j$ and $j-1$, respectively.
 Also, there is a canonical isomorphism $\omega^{\circ}_{A^\vee/S,2}/\im(\phi_{*,2})\xra{\sim} H_2$.
 From the commutative diagram \eqref{D:Frob-phi}, it is easy to see that $F(H_1^{(p)}) \subseteq\Ker(\phi_{*,2})=\im(\psi_{*,2})$, but comparing the rank forces this to be an equality. It follows that $H_2\subseteq F(H_1^{(p)})$.
 Similarly, $V^{-1}(H_2^{(p)})$ is identified with $\im(\psi_{*,1})= \Ker(\phi_{*,1})$, hence $V^{-1}(H_2^{(p)})\subseteq H_1$.
From these, we deduce two canonical isomorphisms:
\begin{equation}\label{Equ:H-omega}
H_1/V^{-1}(H_2^{(p)})\xra{\sim} \omega^{\circ}_{A^\vee/S,1},\quad \text{and} \quad
F(H_1^{(p)})/H_2\xra{\sim} H^{\dR}_1(A/S)^{\circ}_{2}/\omega^{\circ}_{A^\vee/S,2}\cong \Lie_{A/S,2}^{\circ}.
\end{equation}
Therefore, we have a well-defined map $\alpha\colon Y_j\ra Y'_j$ given by
 \[
 \alpha\colon (A,\lambda,\eta,B,\lambda',\eta',\phi)\longmapsto (B,\lambda',\eta',H_1,H_2).
 \]
Moreover, it is clear from the definition that $\pi'_j \circ \alpha = \pr'_j$.

 \begin{proposition}\label{P:isom-Yj}
The morphism $\alpha$ is an isomorphism.
\end{proposition}
\begin{proof}
Let $k$ be a perfect field containing $\F_{p^2}$. We first prove that $\alpha$ induces a bijection of points   $\alpha: Y_j(k)\xra{\sim} Y'_j(k)$.
It suffices to show that there exists  a morphism of sets $\beta: Y'_j(k)\ra Y_j(k)$ inverse to $\alpha$.
Let $y=(B, \lambda',\eta', H_1, H_2)\in Y'_j(k)$.
We define $\beta(y)=(A,  \lambda, \eta, B, \lambda',\eta', \phi)$ as follows.
Let $\tcE_1\subseteq \tcD(B)^{\circ}_1$ and $\tcE_2\subseteq \tcD(B)^{\circ}_{2}$ be respectively the inverse images of $V^{-1}(H_2^{(p)})\subseteq H^{\dR}_1(B/k)^{\circ}_{1}$ and $F(H_1^{(p)})\subseteq H^{\dR}_1(B/k)_{2}^{\circ}$ under the natural reduction maps
$$
\tcD(B)^{\circ}_{i}\ra\tcD(B)^{\circ}_{i}/p\tcD(B)^{\circ}_i \cong H^{\dR}_1(B/k)_i^{\circ}\quad \textrm{for }i=1,2.
$$
The condition \eqref{Equ:condition-Yj} ensures that $F(\tilde \calE_i) \subseteq \tilde \calE_{3-i}$ and $V(\tilde \calE_i) \subseteq \tilde \calE_{3-i}$ for $i = 1,2$.
Applying Proposition~\ref{P:abelian-Dieud} with $m=1$, we get a triple $(A,\lambda,\eta)$ and an $\calO_D$-equivariant isogeny $\psi: A\ra B$, where $A$ is an abelian variety over $k$ with an action of $\cO_D$, $\lambda$ is a prime-to-$p$ polarization on $A$, and $\eta$ is a prime-to-$p$ level structure on $A$, such that $\psi^\vee\circ\lambda'\circ\psi=p\lambda$, $p \eta' = \psi \circ \eta$ and such that $\psi_{*,i}: \tcD(A)^{\circ}_i\ra \tcD(B)^{\circ}_i$ is naturally identified with the inclusion $\tcE_i\hra \tcD(B)^{\circ}_i$ for $i=1,2$.
Moreover, the dimension formula \eqref{E:dimension of new differentials} implies that $\omega^{\circ}_{A^\vee/k,1}$ has dimension $1$, and $\omega^{\circ}_{A^\vee/k,2}$ has dimension $n-1$.
Therefore, $(A,\lambda, \eta)$ is a point of $\Sh_{1,n-1}$.
Finally, we take $\phi: B\ra A$ to be the unique isogeny such that $\phi\circ\psi=p\cdot \id_A$ and $\psi\circ\phi=p \cdot \id_{B}$. Thus we have $\phi \circ \eta' = \eta$.
    This finishes the construction of $\beta(y)$.
     It is direct to check that $\beta$ is the set theoretic inverse to $\alpha: Y_j(k)\ra Y'_j(k)$.

We show now that $\alpha$ induces an isomorphism on the tangent spaces at each closed point; as we have already shown that $Y'_j$ is smooth, it will then follow that $\alpha$ is an isomorphism.
Let $x=(A, \lambda, \eta, B, \lambda',\eta', \phi)\in Y_j(k)$ be a closed point.  Consider the infinitesimal deformation over $k[\epsilon]=k[t]/t^2$.
Note that
 $(B,\lambda',\eta')$ has a unique deformation $(\hat B,\hat \lambda', \hat \eta')$ to $k[\epsilon]$, namely the trivial deformation.
By the Serre--Tate and Grothendieck--Messing deformation theory (cf. Theorem~\ref{T:STGM}), giving a deformation $(\hat{A},\hat \lambda, \hat \eta)$ of $(A,\lambda, \eta)$ to $k[\epsilon]$ is equivalent to giving free $k[\epsilon]$-submodules  $\hat\omega_{A^\vee, i}^{\circ}\subseteq H^{\cris}_1(A/k[\epsilon])^{\circ}_i$  for $i=1,2$ which  lift  $\omega^{\circ}_{A^\vee/k, i}$.
The  isogeny $\phi$ and the polarization $\lambda$ deform to an isogeny  $\hat\phi: \hat B\ra \hat A$ and a polarization $\hat \lambda: \hat A^\vee \to \hat A$ (satisfying $p \hat \lambda' = \hat \phi^\vee \circ \hat \lambda \circ \hat \phi$), necessarily unique if they exist,  if and only if
\[
\hat\omega^{\circ}_{A^\vee, 2}\supseteq \phi_{*,2}^\cris\big(H^{\cris}_1(B/k[\epsilon])^{\circ}_2 \big) \quad \textrm{and } \quad \big(\phi^\cris_{*,1}(H^{\cris}_1(B/k[\epsilon])^{\circ}_1)\big)^\vee \subseteq (\hat\omega_{A^\vee, 1}^{\circ})^\vee,
\]
where the second inclusion comes from the consideration at the embedding $\bar q_2$ by taking duality using the polarization $\lambda$, and is equivalent to $\hat\omega_{A^\vee, 1}^{\circ} \subseteq \phi^\cris_{*,1}\big(H^{\cris}_1(B/k[\epsilon])^{\circ}_1\big)$.

    As discussed before Proposition \ref{P:isom-Yj}, we have  $\Ker(\phi_{*,1}) = V^{-1}(H_2^{(p)})$ and $F(H_1^{(p)})= \Ker(\phi_{*,2})=\im(\psi_{*,2})$. Then
according to the relation between $\omega^\circ_{A^\vee/k, i}$ and $H_1$ in \eqref{E:definition of H12},    
giving such $\hat \omega^{\circ}_{A^\vee, i}$ for $i=1,2$ is equivalent to lifting each $H_i$ to a free $k[\epsilon]$-submodule $\hat H_{i}\subseteq H^{\dR}_1(B/k)^{\circ}_i\otimes_k k[\epsilon] \cong H^\cris_1(B/k[\epsilon])_i^\circ$ for $ i=1,2$ such that $\hat H_{1}\supseteq V^{-1}(H_2^{(p)})\otimes_{k}{k[\epsilon]}$ and $\hat H_{2}\subseteq F(H_1^{(p)})\otimes_k {k[\epsilon]}$. This is exactly the description of the tangent space of $Y'_j$ at $\alpha(x)$. This concludes the proof.
\end{proof}

In the sequel, we will always identify $Y_j$ with $Y'_j$ and $\pr'_j$ with $\pi'_j$. Before proceeding, we prove some results on the structure of $\Sh_{0,n}(\Fpb)$.\\


We turn to the study of the  Shimura variety $\calS h_{0,n}$. The following Proposition was suggested  by an anonymous referee of this article.

\begin{proposition}\label{P:lifting-zero-dim}

\emph{(1)}
The Shimura variety $\calS h_{0,n}$ is finite and \'etale over $\ZZ_{p^2}$. In particular, the reduction map induces a bijection of geometric points $$\calS h_{0,n}(\overline \QQ_p)\xra{\sim} \Sh_{0,n}(\overline \FF_p).$$ 

\emph{(2)}
Let $\tilde x_i=(\tilde B_i,\tilde \lambda_i,\tilde \eta_i)\in \calS h_{0,n}(\overline \QQ_p)$ for $i=1,2$ be two geometric points in characteristic $0$, and $x_i=(B_i,\lambda_i,\eta_i)\in \Sh_{0,n}(\overline\FF_p)$ be their reductions. Then the reduction map on 
\[
\Hom_{\cO_D}(\tilde B_1,\tilde B_2)\xra{\sim} \Hom_{\cO_D}(B_1,B_2)
\]
 is an isomorphism.

\end{proposition}
\begin{proof}
(1)
 Let $\tilde z\in (\tilde B, \tilde \lambda, \tilde \eta)\in \calS h_{0,n}(\CC)$. 
Put $H=H_1(\tilde B(\CC), \Q)$.
  It is a left $D$-module of rank $1$ equipped with an alternating $D$-Hermitian pairing $\langle-, -\rangle_{\tilde \lambda}$ induced by the polarization $\tilde \lambda$.
  Let $(V_{0,n}=D,\langle-,-\rangle_{0,n})$ be the left $D$-module together with its alternating  $D$-Hermitian pairing as in the definition of $\calS h_{0,n}$.
    By results of Kottwitz \cite[\S 8]{kottwitz}, for every place $v$ of $\Q$, the skew-Hermitian $D_{\Q_v}$-modules $H_{\Q_v}$ and $V_{0,n,\Q_v}$ are isomorphic.\footnote{Note that  the two skew-Hermitian forms $(H, \langle-,-\rangle)_{\tilde\lambda}$ and $(V_{0,n},\langle-,-\rangle_{0,n})$ are not necessarily isomorphic  over $\QQ$. However, they differ at most  only by a scalar in $F$, hence define the same similitude unitary group. See \cite[p. 400]{kottwitz} for details.}
    Then $\End_{\cO_D}(\tilde B_{\CC})_{\Q}$ consists of  the elements of $D^{\mathrm{opp}}=\End_{D}(H)$  that preserves the complex structure on $H_{1,\RR}\simeq V_{0,n,\RR}$ induced the Deligne homomorphism by $h:\CC^{\times}\ra G_{0,n}(\RR)$.
   Since  $h(i)$ is necessarily central (because $G_{0,n}^1$ is compact), it follows that $\End_{\cO_{D}}(\tilde B_{\CC})_{\Q}=D^{\mathrm{opp}}$, and 
   $$D\otimes_{E}D^{\mathrm{opp}}\simeq \mathrm{M}_{n^2}(E)\subseteq \End(\tilde B)_{\QQ}.$$
   For dimension reasons, the  inclusion above is an equality, and  
 $\tilde B$ is  isogenous to the product of $n^2$-copies  of abelian varieties with complex multiplication by $E$. Therefore, $\tilde B$ is defined over a number field and has potentially good reduction everywhere.  
This implies that $\calS h_{0,n}$ is proper over $\ZZ_{p^2}$.

To see that $\calS h_{0,n}$ is finite and \'etale over $\ZZ_{p^2}$, it remains to show its \'etaleness over $\ZZ_{p^2}$. But this is clear from the description of its relative differential sheaf in Corollary~\ref{C:tangent space of Sh}, which is trivial as $\Lie_{\calA^\vee / \calS h_{0,n}, 1} = \Lie_{\calA / \calS h_{0,n}, 2} = 0$ by Kottwitz' determinant condition.

 (2) In general, the reduction map  
 \[\Hom_{\cO_D}(\tilde B_1,\tilde B_2)\hra \Hom_{\cO_D}(B_1,B_2)\]
is injective. It remains to see that every element $f\in \Hom_{\cO_D}(B_1,B_2)$ lifts to a homomorphism $\tilde f\in \Hom_{\cO_D}(\tilde B_1,\tilde B_2)$. 
Note that  points $\tilde x_1,\tilde x_2$ can be viewed over $W(\overline \FF_p)$. 
As recalled in Section~\ref{S:preliminary}, to show that $f$ lifts to a map $\tilde f: \tilde B_1\ra \tilde B_2$, it suffices to see that the  induced map on  crystalline homology 
\[f^*: H^{\cris}_{1}(B_2/W(\overline \FF_p))\ra H^{\cris}_1(B_1/W(\overline \FF_p))
\]
preserves the Hodge filtrations
\[
\omega_{\tilde B_i^{\vee}}\subseteq H^{\dR}_1(\tilde B_i/W(\overline \FF_p))\cong H^{\cris}_1(B_i/W(\overline \FF_p)). 
\]
 It is clear that $f^*$ preserves the decomposition 
 \[
 H^{\dR}_1(\tilde B_i/W(\overline \FF_p))=H^{\dR}_1(\tilde B_i/W(\overline \FF_p))_1\oplus H^{\dR}_1(\tilde B_i/W(\overline \FF_p))_2
 \]
 according to the two embeddings of $\cO_{E}$ into $W(\overline\FF_p)$. 
By the Kottwitz'  determinant condition for $\calS h_{0,n}$, the Hodge filtrations on $H^{\dR}_1(\tilde B_i/W(\overline \FF_p))$ are trivial, namely, one has    $\omega^\circ_{\tilde B_i^{\vee} / W(\overline \FF_p),1}=0$, and $\omega_{\tilde B_i^{\vee}/ W(\overline \FF_p),2}^{\circ}=H^{\dR}_1(\tilde B_i/W(\overline \FF_p))_{2}^\circ$ for $i=1,2$. It is clear now $f^*$ preserves this trivial Hodge filtration, since so does it when tensoring with $\overline \FF_p$.
\end{proof}

Fix a geometric point  $z=(B,\lambda,\eta)\in \Sh_{0,n}(\Fpb)$.
Put $C=\End_{\cO_D}(B)_{\Q}$, and  denote by  $\dagger$ the Rosati involution on $C$ induced by $\lambda$.
 Let  $I$ be the algebraic group over $\Q$ such that
\begin{equation}
I(R)=\{x\in C\otimes_{\Q}R\;|\; xx^{\dagger}\in R^{\times}\}, \quad \text{for all $\Q$-algebras }R.
\end{equation}

 \begin{corollary}\label{P:endomorphism-ring}
 We have an isomorphism of algebraic groups over $\QQ$: $I\simeq G_{0,n}$.
 \end{corollary}
 \begin{proof}
  Let $\tilde z=(\tilde B,\tilde \lambda, \tilde \eta)\in \calS h_{0,n}(\overline \QQ_p)$ denote the unique lift of $z$ according to Proposition~\ref{P:lifting-zero-dim} (1). By \ref{P:lifting-zero-dim} (2), we have a canonical  isomorphism 
  \[\End_{\cO_D}(\tilde B)_{\QQ}\xra{\sim}\End_{\cO_D}(B)_{\QQ}=C.\]
 In the proof of \ref{P:lifting-zero-dim}, we have seen that $C=D^{\mathrm{opp}}$. Moreover, the Rosati involution on $C$ corresponds to the involution
   $$b\mapsto b^{\sharp_{\beta_{0,n}}}=\beta_{0,n}b^*\beta_{0,n}^{-1}$$
    on $D^{\mathrm{opp}},$ where $\beta_{0,n}$ is the element in the definition of $\langle-,-\rangle_{0,n}$. It follows immediately that $I\simeq G_{0,n}$.
 \end{proof}
 
 Let $\mathrm{Isog}(z)\subseteq \Sh_{0,n}(\Fpb)$ denote the subset of points $z'=(B',\lambda',\eta')$ such that there exists an $\calO_D$-equivariant quasi-isogeny $\phi: B'\ra B$  such that $\phi^\vee\circ \lambda\circ\phi=c_0\lambda'$ for some $c_0\in \Q_{>0}$. We denote such a quasi-isogeny by   $\phi: z'\ra z$ for simplicity.

\begin{corollary}\label{C:isogeny-classes}
There exists a natural bijection of sets 
\[
\Theta_z\colon \mathrm{Isog}(z)\xra{\sim} G_{0,n}(\QQ)\backslash G_{0,n}(\AAA^{\infty})/K
\]
\end{corollary}
\begin{proof}
First, we give the construction of $\Theta_z$. 
Put $V^{(p)}(B)=T^{(p)}(B)\otimes_{\hat{\Z}^{(p)}}\AAA^{\infty,p}$.
Then $\eta$ determines an isomorphism
 \[
 \tilde\eta: V_{0,n}^{(p)}\otimes_{\Q}\AAA^{\infty,p}\xra{\ \sim\ } V^{(p)}(B),
 \]
modulo right translation by $K^p$.
 For any $z'=(B',\lambda',\eta')\in \mathrm{Isog}(z)$ and  a choice of $\phi:B'\ra B$ as above.
  The quasi-isogeny $\phi$ induces an isomorphism $\phi_*: V^{(p)}(B')\xra{\sim} V^{(p)}(B)$.
Then  there exists a $g^p\in G_{0,n}(\AAA^{\infty,p})$, unique up to right multiplication  by elements of  $K^p$, such that the $K^p$-orbit of $\phi_{*}^{-1}\circ\tilde\eta\circ g$ gives $\eta'$.

We put
\begin{equation}\label{Equ:lattice}
\LL_{z}=\tcD(B)^{\circ, F^2=p}_{1}=\{v\in \tcD(B)^{\circ}_1: F^2(v)=pv\}.
\end{equation}
Since $B$ is supersingular (See Remark~\ref{R:Sh0n supersingular}), this is  a free $\Z_{p^2}$-module of rank $n$, and we have $\tcD(B)^{\circ}_1=\LL_z\otimes_{\ZZ_{p^2}}W(\overline \FF_p)$. Put $\LL_z[1/p]=\LL_z\otimes_{\Z_{p^2}}\Q_{p^2}$.
 Then  $\phi$ induces an isomorphism $\phi_{*}: \LL_{z'}[1/p]\xra{\sim}\LL_{z}[1/p]$.
  Fix a $\Z_{p^2}$-basis for $\LL_{z}$.
  Then  there exists a $g_{\LL}\in \GL_n(\Q_{p^2})$ such that $\phi_{*}(\LL_{z'})=g_{\LL}(\LL_{z})$, and the right coset $g_{\LL}\GL_{n}(\ZZ_{p^2})$ is independent of the choice of such a basis.
  We put $g_p=(c_0, g_{\LL})\in \Q_p^{\times}\times \GL_n(\Q_{p^2})\simeq G_{0,n}(\Q_p)$, which is well defined up to right multiplication by elements of  $K_p=\ZZ^{\times}_{p^2}\times \GL_n(\ZZ_{p^2})$.

  Finally, note that  the quasi-isogeny $\phi': B'\ra B$ is well determined by $z'$ up to left composition with an element $\gamma\in I(\QQ)=G_{0,n}(\QQ)$. 
  If we replace $\phi$ by $\gamma\circ\phi$, then $g:=(g^p,g_p)\in G_{0,n}(\AAA^{\infty})$ is replaced by $\gamma g=(\gamma g^p, \gamma g_p)$. 
Therefore, the map 
\[\Theta_z: \Isog(z)\ra G_{0,n}(\QQ)\backslash G_{0,n}(\AAA^{\infty})/K, \quad  z'\mapsto G_{0,n}(\QQ)g K\] 
is well defined. 
 The fact that $\Theta_z$ is  a bijection follows from the similar classical statement in characteristic $0$.
\end{proof}

 \if false

 In particular, there exists an isomorphism
 \begin{equation}\label{E:isom-IQp}
I(\Q_p)\xra{\ \sim\ } G_{0,n}(\Q_p)\simeq \Q_p^{\times}\times \GL_n(\Q_{p^2}),
\end{equation}
 which is necessarily unique up to conjugation, because all automorphisms of $\GL_n(\Q_{p^2})$ are inner.
This conjugacy class of  isomorphisms  can be explicitly described as follows.

\subsection{Isogeny classes}\label{S:isogeny-class}

Put $V^{(p)}(B)=T^{(p)}(B)\otimes_{\hat{\Z}^{(p)}}\AAA^{\infty,p}$.
Then the prime-to-$p$ level structure $\eta$ determines an isomorphism
 \[
 \tilde\eta: V_{0,n}^{(p)}\otimes_{\Q}\AAA^{\infty,p}\xra{\ \sim\ } V^{(p)}(B),
 \]
unique up to right translation by elements of $K^p$.
 For any $z'=(B',\lambda',\eta')\in \mathrm{Isog}(z)$ and  a choice of $\phi:z'\ra z$ as above.
  The quasi-isogeny $\phi$ induces an isomorphism $\phi_*: V^{(p)}(B')\xra{\sim} V^{(p)}(B)$.
Then  there exists a $g\in G_{0,n}(\AAA^{\infty,p})$, unique up to right multiplication  by elements of  $K^p$, such that the $K^p$-orbit of $\phi_{*}^{-1}\circ\tilde\eta\circ g$ gives $\eta'$.
 If one replaces $\phi$ by $\gamma\phi$ with $\gamma\in I(\Q)\simeq G_{0,n}(\Q)$, then $g$ is replaced by $\gamma g$.

 Similarly, $\phi$ induces an isomorphism $\phi_{*}: \LL_{z'}[1/p]\xra{\sim}\LL_{z}[1/p]$. Fix a $\Z_{p^2}$-basis for $\LL_{z}$.
  Then  there exists a $g_{\LL}\in \GL_n(\Q_{p^2})$ such that $\phi_{*}(\LL_{z'})=g_{\LL}(\LL_{z})$, and the right coset $g_{\LL}\GL_{n}(\ZZ_{p^2})$ is independent of the choice of such a basis.
  We put $g_p=(c_0, g_{\LL})\in \Q_p^{\times}\times \GL_n(\Q_{p^2})\simeq G_{0,n}(\Q_p)$.
   If one replaces $\phi$ by $\gamma\phi$ with $\gamma\in I(\Q)\simeq G_{0,n}(\Q)$, then $(c_0,g_{\LL})$ is replaced by $\gamma g_p=(\gamma\gamma^{\dagger}c_0, \gamma g_{\LL})$.

  To sum up,  we obtain a well-defined map
  \begin{equation}\label{E:isogeny-class}
  \Theta_{z}\colon \Isog(z)\longrightarrow G_{0,n}(\Q)\backslash \big (G_{0,n}(\AAA^{\infty,p})\times G_{0,n}(\Q_p) \big )/K
  \end{equation}
  with $K=K^p\times K_p$,
  which attaches to $z'$ the class of pairs $(g,g_p)$.

 \begin{proposition}\label{P:isogenous-class}
 The map $\Theta_{z}$ is bijective.
 \end{proposition}
\begin{proof}
We show first the injectivity of $\Theta_z$.
Given $z_1=(B_1,\lambda_1,\eta_1), z_2=(B_2,\lambda_2,\eta_2)\in \Isog(z)$ that correspond to the same double coset, we have to prove that  they are isomorphic.
Choose quasi-isogenies $\phi_1: B_1\ra B$ and $\phi_2: B_2\ra B$, and let   $(g_1, g_{p,1}), (g_2, g_{p,2})\in G_{0,n}(\AAA^{\infty,p})\times G_{0,n}(\Q_p)$ denote the corresponding pairs. Then there exists $\gamma\in G_{0,n}(\Q)$ and $(k^p,k_p)\in K^p\times K_p$ such that $\gamma(g_{1},g_{p,1})=(g_{2}k^p,g_{p,2}k_p)$. Consider the quasi-isogeny
\[
\phi=\phi_2^{-1}\circ\gamma\circ \phi_1: B_1\ra B_2.
\]
Then  $\phi$ induces  isomorphisms $T^{(p)}(B_1)\xra{\sim}T^{(p)}(B_2)$, and $\LL_{z_1}\xra{\sim}\LL_{z_2}$.
Since $\tcD(B_i)^{\circ}_{1}\oplus \tcD(B_i)^{\circ}_2\simeq\big(\LL_{z_i}\oplus F(\LL_{z_i})\big)\otimes_{\Z_{p^2}}W(\Fpb)$, it follows that $\phi_{*}: \tcD(B_1)\xra{\sim} \tcD(B_2)$ is an isomorphism.
Hence, $\phi$ is an isomorphism of abelian varieties, and it is easily seen to be  compatible with all structures. This shows the injectivity of $\Theta_{z}$.

We now prove the surjectivity of $\Theta_z$. Suppose that we are given   $(g,g_p)\in G_{0,n}(\AAA^{\infty,p})\times G_{0,n}(\Q_p)$. We need to show that $(g,g_p)$ comes from some $z'\in \Isog(z)$.
Write $g_p=(p^mu,g_{\LL})\in G_{0,n}(\Q_p)\simeq \Q_p^{\times}\times \GL_n(\Q_{p^2})$ with $u\in \Z_{p}^{\times}$.
 Up to multiplying $(g,g_p)$ by an integer (viewed  in $\Q^{\times}\in G_{0,n}(\Q)$), we may assume that $g(\Lambda_{0,n}\otimes \widehat{\Z}^{(p)})\subseteq \Lambda_{0,n}\otimes \widehat{\Z}^{(p)}$, $m\geq 0$ and $p^{m}\LL_z\subseteq g_{\LL}(\LL_{z})\subseteq \LL_{z}$.  Here, $\Lambda_{0,n}\subseteq V_{0,n}$ denotes the $\cO_D$-stable lattice used in the definition of $\Sh_{0,n}$.
Then $T^{(p)}_{z'}:=\tilde\eta\circ g(\Lambda_{0,n}\otimes \widehat{\Z}^{(p)})$ is an $\cO_{D}$-stable sublattice of $T^{(p)}(B)=\tilde\eta(\Lambda_{0,n}\otimes \widehat{\Z}^{(p)})$.
Choose an integer $N\geq 1$ coprime to $p$ such that $N T^{(p)}(B)\subseteq  T^{(p)}_{z'}$ with quotient  $H^{(p)}$.
It is a subgroup of $B[N]$ stable under $\cO_D$.

Similarly, choose  an integer $m\geq 0$ such that $p^m\LL_z\subseteq g_{\LL}(\LL_z)\subseteq \LL_z$.
 Put $\tcE_1=g_{\LL}(\LL_z)\otimes_{\Z_{p^2}}W(\Fpb)$, $\tcE_2=F(\tcE_1)$.
  Then $\tcE_1\oplus \tcE_2$ is a Dieudonn\'e submodule  of $\tcD(B)^{\circ}$, and the quotient
 \[
 \bigoplus_{i=1}^2 \big(\tcE_i/p^m\tcD(B)_i^{\circ}\big)^{\oplus n}
 \]
 corresponds to a subgroup scheme $H_{\gothp}$ of $B[\gothp^m]$ stable under the action of $\cO_D$.
 Denote by $H_{\bar\gothp}\subseteq B[\gothp^m]$ the orthogonal complement of $H_{\gothp}$ under the perfect pairing $B[\gothp^m]\times B[\bar\gothp^m]\ra \mu_{p^m}$.
 Put $H_p=H_{\gothp}\oplus H_{\bar\gothp}$, and $H=H^{(p)}\oplus H_{p}\subseteq B[Np^m]$. Let $\psi: B\ra B'$ be the canonical  isogeny with kernel $H$, and $\phi: B'\ra B$ be the isogeny such that $\phi\circ\psi=Np^m \mathrm{Id}_{B}$.
  Then $B'$ is an abelian scheme equipped with an induced $\cO_D$-action.
   The induced map  $\phi_{*}: T^{(p)}(B')\ra T^{(p)}(B)$ is naturally identified with the inclusion  $T^{(p)}_{z'}\subseteq T^{(p)}(B)$, and $\phi_{*}: \tcD(B')^{\circ}_i\ra \tcD(B)^{\circ}_i$ is identified with the inclusion $\tcE_i\hra \tcD(B)^{\circ}_i$ for $i=1,2$.
   This implies in particular that $\Lie_{B',i}^{\circ}$ is of dimension $n$ for $i=1$ and dimension $0$ for $i=2$.
  By \cite[\S 23, Theorem 2]{mumford}, there is a prime-to-$p$ polarization $\lambda'$ on $B'$ such that $p^mN\lambda'=\phi^\vee\circ\lambda\circ\phi$.  Finally, the orbit of $\phi^{-1}_{*}\circ \tilde\eta\circ g$ under $K^p$  defines a $K^p$-level structure  on $B'$. Therefore $z'=(B',\lambda',\eta')$ defines a point of $\Isog(z)
\subseteq \Sh_{0,n}(\Fpb)$, and its image under $\Theta_z$ is the class of $(g,g_p)$. This finishes the proof of the surjectivity of $\Theta_z$.
\end{proof}
\fi

\begin{remark}\label{R:isogenous-class}
It follows from  Proposition~\ref{P:lifting-zero-dim} and  the description of $\calS h_{0,n}(\CC)$  in Subsection~\ref{S:defn of Shimura var} that
$\Sh_{0,n}(\Fpb)$  consists of $\#\mathrm{ker}^1(\Q,G_{0,n})$ isogeny classes of abelian varieties equipped with additional structures.
\end{remark}


\begin{lemma}\label{L:finiteness-auto}
Let $N$ be a fixed non-negative integer.
Up to replacing $K^p$ by an open compact subgroup of itself,  the following properties are satisfied:
if  $(B,\lambda,\eta)$ is an $\Fpb$-point of $\Sh_{0,n}$ and  $f: B\ra B$ is an $\cO_D$-quasi-isogeny such that
\begin{itemize}
\item $p^Nf\in \End_{\cO_D}(B)$,
\item $f^\vee\circ\lambda\circ f=\lambda$,
\item $f\circ \eta=\eta$,
\end{itemize}
 then $f=\id$.
\end{lemma}
\begin{proof}
It suffices to prove the Lemma for $(B,\lambda,\eta)$ in a fixed isogeny class $\Isog(z)$ of $\Sh_{0,n}(\Fpb)$.
 We write $G_{0,n}(\AAA^{\infty})=\coprod_{i\in I} G_{0,n}(\Q)g_i K$ with $K=K^pK_p$, where $g_i=g_i^pg_{i,p}$, with $g_i^p\in G_{0,n}(\AAA^{\infty,p})$ and $g_{i,p}\in G_{0,n}(\QQ_p)$, runs through a finite set of representatives of the double coset
 $$G_{0,n}(\Q)\backslash G_{0,n}(\AAA^{\infty})/K.$$
 Let $(B,\lambda, \eta)$ be a point of $\Sh_{0,n}$ corresponding to $G_{0,n}(\Q)g_i K$ for some $i\in I$, and  $f$ be an $\cO_D$-quasi-isogeny  of $B$ as in the statement.
  Then $f$ is given by an element of $G_{0,n}^1(\Q)$.
  The condition that $f\circ \eta=\eta$ is equivalent to  saying that the image of  $f$ in $G_{0,n}(\AAA^{\infty,p})$ lies in  $g_i^pK^p{g_i}^{p,-1}$.
   Moreover,  $p^Nf\in \End_{\cO_D}(B)$ implies that the image of $f$ in $G_{0,n}(\Q_p)$ belongs to  $\coprod_{\delta}g_{i,p}(K_p\delta K_p) g_{i,p}^{-1}$, where  $\delta$ runs through the set
  \[
  \big\{(1, \diag(p^{a_1}, p^{a_2},\dots, p^{a_{n}}))\in G_{0,n}(\Q_p)\simeq \Q_p^{\times}\times \GL_n(\QQ_{p^2})\;\big|\; 0\geq a_1\geq a_2\geq  \cdots \geq a_n\geq -N\big\}.
  \]
  Write $\coprod_{\delta}K_p\delta K_p=\coprod_{j\in J} h_jK_p $ for some finite set $J$. Hence, it suffices to show that there exists an open compact subgroup   $K'^p\subseteq K^p$ such  that, for all $g_i$,   $G_{0,n}^1(\Q)\cap g_i(K'^p\cdot h_jK_p)g_i^{-1}$ is equal to $\{1\}$  if $h_jK_p=K_p$, and empty otherwise.
  Since $K$ is neat, we have $G_{0,n}^1(\Q)\cap g_i (K'^pK_p) g_i^{-1}=\{1\}$ for any $g_i$ and any $K'^p\subseteq K^p$.
  Note that this implies that, for each $i\in I$, $G_{0,n}^1(\Q)\cap g_i (K^p\cdot h_jK_p) g_i^{-1}$ contains at most one element (because if it contains both  $x$ and $y$, then $x^{-1}y\in G_{0,n}^1(\Q)\cap g_i K g_i^{-1}=\{1\}$).
   Let $S\subset I\times J$ be the subset consisting of $(i,j)$ such that $h_jK_p\neq K_p$ and $G_{0,n}^1(\Q)\cap g_i (K^p\cdot h_jK_p) g_i^{-1}$ indeed contains one element, say $x_{i,j}$. Then $x_{i,j}\neq 1$ for all $(i,j)\in S$.
 Hence, one can choose a normal open compact subgroup  $K'^p\subseteq K^p$ so that $x_{i,j}\notin g_i^pK'^pg_i^{p,-1}$ for all $i$. We claim that this choice of $K'^p$ will satisfy the desired property. Indeed, if $K^p=\coprod_{l} b_{l}K'^p$, then the double coset $G_{0,n}(\Q)\backslash G_{0,n}(\AAA^{\infty})/K'^pK_p$ has a set of representatives of the form $g_ib_{l}$. Here, by abuse of notation, we consider $b_l$  as an element of $K$ with $p$-component equal to $1$.
 Then one has, for $h_jK_p\neq K_p$,
  \[
  G_{0,n}^1(\Q)\cap g_ib_l (K'^p h_jK_p)b_l^{-1}g_{i}^{-1}=G_{0,n}^1(\Q)\cap g_i(K'^p h_jK_p)g_i^{-1}=\emptyset.
  \]
Here, the first  equality uses the fact that $K'^p$ is normal in $K^p$. This finishes the proof.
\end{proof}

We come back to the discussion on the cycles $Y_j\subseteq \Sh_{1,n-1}$ for $1\leq j\leq n$.

\begin{proposition}\label{P:normal-bundle}
Let $(\calA,\lambda,\eta,\calB,\lambda',\eta',\phi^{\univ})$ denote the universal object on $Y_j$ for $1\leq j\leq n$, 
 and   $\cH_i\subset H^{\dR}_{1}(\calB/\Sh_{0,n})$ for $i=1,2 $ be the universal subbundles on $Y'_j\cong Y_j$.
\begin{enumerate}
\item The induced map $T_{Y_j}\ra \pr^*_jT_{\Sh_{1,n-1}}$ is universally injective, and we have canonical isomorphisms
\begin{align*}
N_{Y_j}(&\Sh_{1,n-1}): =\pr_j^*T_{\Sh_{1,n-1}}/T_{Y_j}\\
&\cong \big(\cH_1/V^{-1}(\cH_2^{(p)})\big)^*\otimes V^{-1}(\cH_2^{(p)})
\oplus \big(F(\cH_1^{(p)})/\cH_2 \big) \otimes \big (H^{\dR}_1(\calB/\Sh_{0,n})^{\circ}_2/F(\cH_1^{(p)})\big)^*\\
&\cong \Lie_{\calA^\vee,1}^{\circ}\otimes \coker(\phi^{\univ}_{*,1})\oplus \Lie_{\calA,2}^{\circ}\otimes\im(\phi^{\univ}_{*,2})^*.
\end{align*}

\item Assume that  $K^p$ is sufficiently small so that the consequences of Lemma~\ref{L:finiteness-auto} hold for $N=1$. For each fixed closed point $z\in \Sh_{0,n}$, the map $\pr_{j,z}:=\pr_j|_{Y_{j,z}}: Y_{j,z}\ra \Sh_{1,n-1}$ is a closed immersion, or equivalently, the morphism $(\pr_j,\pr'_j): Y_j\ra \Sh_{1,n-1}\times_{\Spec(\F_{p^2})} \Sh_{0,n}$ is a closed immersion.

\item The union of the images of $\pr_{j}$ for all $1\leq j\leq n$ is the supersingular locus of $\Sh_{1,n-1}$, i.e. the reduced closed  subscheme of $\Sh_{1,n-1}$ where all the slopes of the Newton polygon of the $p$-divisible group $\calA[\gothp^{\infty}]$ are $1/2$.

  \end{enumerate}
\end{proposition}
\begin{proof}

(1) Let $S$ be an affine noetherian $\FF_{p^2}$-scheme and let
$y=(A,\lambda,\eta,B,\lambda',\eta', \phi)$ be an $S$-point of $Y_j$. Put $\hat S=S\times_{\Spec(\F_{p^2})}\Spec(\F_{p^2}[t]/t^2)$. Then we have a natural bijection
\[
\Def(y,\hat S)\cong \Gamma(S, y^*T_{Y_j}),
\]
where $\Def(y,\hat S)$ is the set of deformations of $y$ to $\hat S$.
 Similarly, we have $\Def(\pr_j\circ y, \hat S)\cong  \Gamma(S, y^*\pr_{j}^*T_{\Sh_{1,n-1}})$.
 To prove the universal injectivity of $T_{Y_j}\ra \pr^*_jT_{\Sh_{1,n-1}}$, it suffices to show that the natural map $\Def(y,\hat S)\ra \Def(\pr_j\circ y,\hat S)$ is injective.
By crystalline deformation theory (Theorem~\ref{T:STGM}), giving a point of $\Def(y,\hat S)$ is equivalent to giving $\cO_{\hat S}$-subbundles  $\hat\omega^{\circ}_{A^\vee,i}\subseteq H^{\cris}_{1}(A/\hat S)_i^\circ$ over $\hat S$ for $i =1,2$ such that
\begin{itemize}
\item $\hat \omega^{\circ}_{A^\vee,i}$ lifts $\omega^{\circ}_{A^\vee/S,i}$;
\item $\hat \omega^{\circ}_{A^\vee, 1}\subseteq \im( \phi_{*,1})\otimes \F_{p^2}[t]/t^2$ and   $\im(\phi_{*,2})\otimes \F_{p^2}[t]/t^2\subseteq \hat \omega^\circ_{A^\vee, 2}$ are locally direct factors.
\end{itemize}
 Hence, one sees easily that
\begin{align*}
\Def(y,\hat S)&\cong \Hom_{\cO_S}\big(\omega^{\circ}_{A^\vee/S, 1}, \im(\phi_{*,1})/\omega^{\circ}_{A^\vee/S, 1}\big)\oplus \Hom_{\cO_S}\big(\omega^{\circ}_{A^\vee/S,2}/\im(\phi_{*,2}), H^{\dR}_1(A/S)^{\circ}_2/\omega^\circ_{A^\vee/S,2}\big)\\
&\cong \Lie^{\circ}_{A^\vee/S,1}\otimes \big(\im(\phi_{*,1})/\omega^{\circ}_{A^\vee/S,1}\big)\oplus \big(\omega_{A^\vee/S,2}^{\circ}/\im(\phi_{*,2})\big)^*\otimes \Lie_{A/S,2}^{\circ}.
\end{align*}
Similarly, $\Def(\pr_j\circ y,\hat S)$ is given by the lifts of $\omega^\circ_{A^\vee/S,i}$ to $\hat S$ for $i=1,2$.
These lifts are classified by $\Hom_{\cO_S}\big(\omega^\circ_{A^\vee/S,i}, H^{\dR}_1(A/S)^{\circ}_i/\omega^\circ_{A^\vee/S,i}\big)\cong \Lie^{\circ}_{A^\vee/S,i}\otimes_k\Lie_{A/S,i}^{\circ}$.
  Hence, $\Def(\pr_j\circ y,\hat S)$ is canonically isomorphic to
\[
\Lie^{\circ}_{A^\vee/S,1}\otimes_{\cO_S}\Lie_{A/S,1}^{\circ}
\oplus \Lie^{\circ}_{A^\vee/S,2}\otimes_{\cO_S}\Lie_{A/S,2}^{\circ}.
\]
The natural map
 $\Def(y,\hat S)\ra \Def(\pr_j\circ y, \hat S)$ is induced by the natural maps
 \begin{align*}
 \im(\phi_{*,1})/\omega^{\circ}_{A^\vee/S,1}&\hookrightarrow H^{\dR}_1(A/S)^{\circ}_{1}/\omega^{\circ}_{A^\vee/S, 1}\cong\Lie_{A/S,1}^{\circ},\\
 \big(\omega^{\circ}_
  {A^\vee/S, 2}/\im(\phi_{*,2}) \big)^*&\hra  \omega^{\circ,*}_{A^\vee/S,2}\cong\Lie_{A^\vee/S,2}^{\circ}.
 \end{align*}
It follows that $\Def(y,\hat S)\ra \Def(\pr_j\circ y, \hat S)$ is injective.
To prove the formula for $N_{Y_j}(\Sh_{1,n-1})$, we apply the arguments above to affine open subsets of $Y_j$. We see easily that
\begin{align*}
N_{Y_j}(\Sh_{1,n-1})&\cong \Lie_{\calA^\vee,1}^{\circ}\otimes_{\cO_{Y_j}}\coker(\phi^{\univ}_{*,1})\oplus \Lie^{\circ}_{\calA,2}\otimes_{\cO_{Y_j}}\im(\phi^{\univ}_{*,2})^*\\
&\cong \big(\cH_1/V^{-1}(\cH_2^{(p)})\big)^*\otimes V^{-1}(\cH_2^{(p)})
\oplus  \big(F(\cH_1^{(p)})/\cH_2\big)\otimes \big(H^{\dR}_1(\cB/Y_j)^{\circ}_2/F(\cH_1^{(p)})\big)^*.
\end{align*}
Here, the last step uses \eqref{Equ:H-omega} and the isomorphism
$$
\im(\phi^{\univ}_{*,2})\cong H^{\dR}_1(\cB/Y_j)^{\circ}_{2}/\Ker(\phi^{\univ}_{*,2})\cong H^{\dR}_1(\cB/Y_j)^{\circ}_{2}/F(\cH_1^{(p)}).
$$

(2) By statement (1), $\pr_{j,z}$ induces an injection of tangent spaces at each closed points of $Y_{j,z}$. To complete the proof, it suffices to  prove that $\pi_{j,z}$ induces injections on the closed points. Write $z=(B,\lambda',\eta')\in \Sh_{0,n}(\Fpb)$. Assume $y_1$ and $y_2$ are two closed points of  $Y_{j,z}$ with $\pi_{j}(y_1)=\pi_{j}(y_2)=(A,\lambda,\eta)$. Let $\phi_1,\phi_2:B\ra A$ be the isogenies given by $y_1$ and $y_2$. Then the quasi-isogeny $\phi_{1,2}=\phi_2^{-1}\phi_1\in \End_{\calO_D}(B)_{\Q}$ satisfies the conditions of Lemma~\ref{L:finiteness-auto} for $N=1$. Hence, we get $\phi_{1,2}=\id_{B}$, which is equivalent to  $y_1=y_2$. This proves that $\pi_{j,z}$ is injective on closed points.

(3) 
The proof resembles the work of Vollaard--Wedhorn \cite{vollaard-wedhorn}.
Since all the points of $\Sh_{0,n}(\Fpb)$ are supersingular by Remark~\ref{R:Sh0n supersingular}, it is clear that the image of each $\pr_{j}$ lies in the supersingular locus of $\Sh_{1,n-1}$. Suppose now given a  supersingular point $x=(A,\lambda, \eta)\in \Sh_{1,n-1}(\Fpb)$.
We have to show that there exists $(B, \lambda',\eta')\in \Sh_{0,n}$ and
an isogeny $\phi: B\ra A$ such that $(A,\lambda, \eta, \lambda',\eta';\phi)$ lies in $Y_j$ for some $1\leq j\leq n $.

Consider
$$
\LL_{\Q}=\big(\tcD(A)^{\circ}_{1}[1/p]\big)^{F^{2}=p}=\big\{a\in \tcD(A)^{\circ}_1[1/p]\;\big|\;F^2(a)=pa\big\}.
$$
Since $x$ is supersingular,  $\LL_{\Q}$ is a $\Q_{p^2}$-vector space of dimension $n$ by the Dieudonn\'e--Manin  classification, and $\tcD(A)^{\circ}_{1}[1/p]=\LL_{\Q}\otimes_{\QQ_{p^2}}W(\Fpb)[1/p]$. We put $\tcE^{\circ}_1=\big(\LL_{\Q}\cap \tcD(A)^{\circ}_1\big)\otimes_{\Z_{p^2}}W(\Fpb)$, and $\tcE^{\circ}_2=F(\tcE^{\circ}_1)\subseteq \tcD(A)^{\circ}_2$.
Thus $\tcE^{\circ}=\tcE_{1}^{\circ}\oplus \tcE_{2}^{\circ}$ is a Dieudonn\'e submodule of $\tcD(A)^\circ$.
We claim that $\tcE^{\circ}$ contains $p\tcD(A)^{\circ}$ as a submodule.
Then applying Proposition~\ref{P:abelian-Dieud} with $m=1$, we get an $\cO_{D}$-abelian variety $(B,\lambda',\eta')$ together with an $\cO_D$-isogeny $\phi: B\ra A$ with $\phi^\vee\circ\lambda\circ\phi=p\lambda$. It is easy to see in this case that $(A,\lambda,\eta,B,\lambda',\eta',\phi)$ defines a point in $Y_j$ with $j=\dim_{\Fpb}(\tcD(A)^{\circ}_2/\tcE^{\circ}_2)$.

It then suffices to prove the claim that $p \tcD(A)^\circ \subseteq \tcE^\circ$. Suppose not, then $\tcD(A)^\circ \nsubseteq \frac 1p \tcE^\circ$.
Consider $M_i:=\tcD(A)^{\circ}_i/\tcE^{\circ}_i$ for $i=1,2$.
For any integer $\alpha\geq 0$,  its   $p^{\alpha}$-torsion submodule is 
\[
M_i[p^{\alpha}]=\big(\tcD(A)^{\circ}_{i}\cap \frac{1}{p^\alpha}\tcE^{\circ}_i\big)/\tcE^{\circ}_i.
\]
It follows easily that 
\[M_i[p^{\alpha+1}]/M_i[p^{\alpha}]\cong \Big( \tfrac1{p^{\alpha+1}}\tcE^\circ_i \cap \big(\tcD(A)^\circ_i + \tfrac 1{p^\alpha} \tcE^\circ_i\big) \Big)\big/ \tfrac 1{p^\alpha} \tcE^\circ_i.
\]
On the other hand, the Kottwitz' signature condition implies that both $F$ and $V: \tcD(A)^\circ_1 \to \tcD(A)^\circ_2$ have cokernel isomorphic to $\overline \FF_p$, and both $F$ and $ V: \tilde \calE_1^\circ \to \tilde \calE_2^\circ$ are isomorphism.
Therefore, the two induced morphisms 
\[
F \textrm{ and }V: M_1\ra M_2
\]
are injective and both 
have cokernel isomorphic to $\overline \FF_p$.
It follows that the induced maps on the graded pieces  
\begin{equation}
\label{E:F and V are isomorphisms}
F\textrm{ and }V: \Big( \tfrac1{p^{\alpha+1}}\tcE^\circ_1 \cap \big(\tcD(A)^\circ_1 + \tfrac 1{p^\alpha} \tcE^\circ_1\big) \Big) \big/ \tfrac 1{p^\alpha} \tcE^\circ_1 
\longrightarrow 
\Big( \tfrac1{p^{\alpha+1}}\tcE^\circ_2 \cap \big(\tcD(A)^\circ_2 + \tfrac 1{p^\alpha} \tcE^\circ_2\big) \Big) \big/ \tfrac 1{p^\alpha} \tcE^\circ_2
\end{equation}
are injective maps, and are isomorphisms for all $\alpha\geq 0$ except for exactly one $\alpha$.\footnote{We point out that, for \eqref{E:F and V are isomorphisms}, $F$ is an isomorphism if and only if $V$ is an isomorphism, because this is equivalent to requiring the source and the target to have the same dimension.}
The assumption  $\tcD(A)^\circ \nsubseteq \frac 1p \tcE^\circ$ implies that there are at least two $\alpha\geq 0$ for which the right hand side of \eqref{E:F and V are isomorphisms} is non-zero.
 So there exists $\alpha \geq 0$ such that \eqref{E:F and V are isomorphisms} are isomorphisms of \emph{non-zero $\overline \FF_p$-vector spaces}.
Multiplication by $p^{\alpha}$ gives isomorphisms:
\begin{equation}
\label{E:F and V both isomorphism}
F\textrm{ and }V: \Big( \tfrac1{p}\tcE^\circ_1 \cap \big(p^{\alpha}\tcD(A)^\circ_1 +  \tcE^\circ_1\big) \Big) 
\longrightarrow 
\Big( \tfrac1p\tcE^\circ_2 \cap \big(p^\alpha\tcD(A)^\circ_2 +  \tcE^\circ_2\big) \Big).
\end{equation}
Now, Hilbert 90 theorem implies that
$\LL' = \Big( \tfrac1{p}\tcE^\circ_1 \cap \big(p^\alpha\tcD(A)^\circ_1 +  \tcE^\circ_1 \big)  \Big)^{F^2=p}$ in fact generates the source of \eqref{E:F and V both isomorphism}.
On the other hand, it is obvious that $\LL' \subset \LL_\QQ$ and $\LL' \subseteq p^\alpha \tcD(A)^\circ_1 + \tcE^\circ_1 \subseteq \tcD(A)^\circ_1$.
This means that $\LL'$, and hence $\LL_\QQ \cap \tcD(A)^\circ_1$,  generates the entire $\tfrac1{p}\tcE^\circ_1 \cap \big(p^\alpha\tcD(A)^\circ_1 +  \tcE^\circ_1\big) $, i.e. one has  $\tfrac1{p}\tcE^\circ_1 \cap \big(p^\alpha\tcD(A)^\circ_1 +  \tcE^\circ_1\big) =\tcE^{\circ}_1$.
But this contradicts with the non-triviality of the vector spaces in \eqref{E:F and V are isomorphisms} by our choice of $\alpha$.
Now the Proposition is proved.
\end{proof}

\begin{corollary}\label{C:vanishing}
 The morphism $\pr_1$ $($resp. $\pr_{n})$ is a closed immersion, and it identifies $Y_1$ $($resp. $Y_n)$ with the closed subscheme of $\Sh_{1,n-1}$ defined by the vanishing of $V:\omega^{\circ}_{\calA^\vee,2}\ra \omega^{\circ,(p)}_{\calA^\vee,1}$  $($resp. $V:\omega^{\circ}_{\calA^\vee,1}\ra \omega^{\circ,(p)}_{\calA^\vee,2})$.

\end{corollary}

\begin{proof}
We just prove the statement for $\pr_1$, and the case of $\pr_n$ is similar.
 Let $Z_1$ be the closed subscheme of $\Sh_{1,n-1}$ defined by the condition that $V: \omega^{\circ}_{\calA^\vee,2}\ra \omega^{\circ,(p)}_{\calA^\vee, 1}$ vanishes.
 We show first that $\pr_1:Y_1\ra \Sh_{1,n-1}$ factors through the natural inclusion $Z_1\hra \Sh_{1,n-1}$.
 Let $y=(A,\lambda,\eta, B,\lambda',\eta',\phi)$ be an $S$-valued point of $Y_1$.   By Lemma~\ref{L:omega-F}, $\im(\phi_{2,*})$ has rank $n-1$ and contains both $\omega_{A^\vee/S,2}^{\circ}$  and $F\big(H^{\dR}_1(A/S)^{\circ,(p)}_1\big)$, which are both $\calO_S$-subbundles  of rank $n-1$.
 This forces $\omega^{\circ}_{A^\vee/S, 2}=F\big(H^{\dR}_1(A/S)^{\circ,(p)}_1\big)$, and therefore  $V:\omega^{\circ}_{A^\vee/S,2}\ra \omega^{\circ,(p)}_{A^\vee/S, 1}$ vanishes. This shows that $\pr_1(y)\in Z_1$.

To prove that $\pr_1: Y_1\ra Z_1$ is an isomorphism, as $Y_1$ is smooth, it suffices to show that it induces a bijection between closed points and tangent spaces of $Y_1$ and $Z_1$.
 For any perfect field $k$ containing $\F_{p^2}$, one constructs a map $\theta: Z_1(k)\ra Y_1(k)$ inverse to $\pr_1: Y_1(k)\ra Z_1(k)$ as follows.
 Given $x=(A,\lambda, \eta)\in Z_1(k)$.
 Let $\tcE^{\circ}_1=\tcD(A)^{\circ}_1$ and $\tcE^{\circ}_{2}\subseteq \tcD(A)^{\circ}_2$ be the inverse image of $\omega^{\circ}_{A^\vee/k,2}\subseteq \tcD(A)^{\circ}_2/p\tcD(A)_2^{\circ}$.
  Then the condition that $y\in Z_1$ implies that $\tcE^\circ_1\oplus \tcE^{\circ}_2$ is stable under $F$ and $V$.
  Applying Proposition~\ref{P:abelian-Dieud} with $m=1$, we get a tuple $(B,\lambda',\eta',\phi)$ such that $y=(A,\lambda, \eta, B,\lambda',\eta',\phi)\in Y_1(k)$.
   It is immediate to check that $x\mapsto y$ and $\pr_1$ is  the   set theoretic  inverse of each other.
  It remains to show that $\pr_1$ induces a bijection between $T_{Y_1, y}$ and $T_{Z_1,x}$. Proposition~\ref{P:normal-bundle} already implies that we have an inclusion
$T_{Y_1, y} \hookrightarrow
T_{Z_1,x} \hookrightarrow T_{\Sh_{1,n-1}, x}$. It suffices to check that $\dim T_{Z_1, x} = n-1$. 
The tangent space $T_{Z_1,x}$ is the space of deformations $(\hat A, \hat\lambda, \hat\eta)$ over $\hat k = k[\epsilon]/(\epsilon^2)$ of $(A,\lambda, \eta)$ such that $V:  \omega^{\circ}_{\hat A^\vee/\hat k, 2}\ra \omega^{\circ,(p)}_{\hat A^\vee/\hat k, 1}=\omega^{\circ,(p)}_{ A^\vee/k, 1}\otimes_k \hat k$ vanishes.
 This uniquely determines the lift $\hat \omega^{\circ}_{ A^\vee, 2}=\omega^{\circ}_{\hat A^\vee/\hat k,2}$.  So by deformation theory (Theorem~\ref{T:STGM}), the tangent space $T_{Z_1, x}$ is determined by the liftings $\hat\omega^{\circ}_{A^\vee,1}=\omega^{\circ}_{\hat A^\vee/\hat k,1}$
 of $\omega^{\circ}_{A^\vee / k,1}$. So it is of dimension $n-1$.  This concludes the proof of the corollary.
\end{proof}


\subsection{Geometric Jacquet--Langlands morphism}\label{S:Tate-cycles}
 Let $\ell\neq p$ be a prime number.
 For $1\leq j\leq n$,  the diagram \eqref{D:two-projection} gives rise to  a natural morphism
\begin{equation}\label{E:jacquet-langlands}
\JL_j\colon H^0_\et\big(\overline \Sh_{0,n},\Ql\big)\xra{\pr_{j}'^*} H^0_\et\big(\overline Y_{j},\Ql\big)\xra{\pr_{j,!}} H^{2(n-1)}_\et\big(\overline\Sh_{1,n-1},\Ql(n-1)\big),
\end{equation}
where $\pr_{j,!}$ is \eqref{E:Gysin-map}, whose restriction to each $H^0_\et(Y_{j,z},\Ql)$ for $z\in \Sh_{0,n}(\Fpb)$ is the Gysin map associated to the closed immersion $Y_{j,z}\hra \overline\Sh_{1,n-1}$.
It is clear that the image of $\JL_j$ is the subspace generated by the cycle classes of $[Y_{j,z}]\in A^{n-1}(\overline\Sh_{1,n-1})$ with $z\in\Sh_{0,n}(\Fpb) $.
According to \cite{helm-PEL}, $\JL_j$ should be considered as a certain geometric realization of the Jacquet--Langlands transfer from  $G_{0,n}$ to  $G_{1,n-1}$.
Putting all $\JL_j$'s together, we get  a morphism
\begin{equation}\label{E:total-JL}
\JL=\sum_{j}\JL_j\colon \bigoplus_{j=1}^n H^{0}_\et\big(\overline\Sh_{0,n},\Ql\big) \longrightarrow H^{2(n-1)}_\et\big(\overline\Sh_{1,n-1},\Ql(n-1)\big).
\end{equation}
Recall that  we have fixed an isomorphism $G_{1,n-1}(\AAA^{\infty})\simeq G_{0,n}(\AAA^{\infty})$, which we write uniformly as  $G(\AAA^{\infty})$.
Denote by $\scrH(K^p,\Qlb)=\Qlb[K^p\backslash G(\AAA^{\infty,p})/K^p]$ the prime-to-$p$ Hecke algebra.
Then the homomorphism \eqref{E:total-JL} is a homomorphism of $\scrH(K^p,\Qlb)$-modules. 


For an irreducible admissible representation $\pi$ of $G(\AAA_{\infty})$, we write $\pi=\pi^p\otimes \pi_p$, where $\pi^p$ (resp. $\pi_p$) is the prime-to-$p$ part (resp. the $p$-component) of $\pi$.

\begin{lemma}\label{L:automorphic-prime-to-p}
Let $\pi_1$ and $\pi_2$ be two admissible irreducible representations of $G(\AAA^\infty)$, and  $(r_i, s_i)$ for $i=1,2$ be two pairs of integers with $0\leq r_i,s_i\leq n$ and $r_1+s_1\equiv r_2+s_2\mod 2$.
Assume that $\pi_1$ satisfies Hypothesis~\ref{H:automorphic assumption} with  $a_{\bullet}=(r_1, s_1)$,
 and there exists an admissible irreducible representation $\pi_{2,\infty}$ of $G_{(r_2,s_2)}(\RR)$ 
 such that $\pi_{2}\otimes\pi_{2,\infty}$ is a cuspidal automorphic representation of $G_{(r_2,s_2)}(\AAA)$. 
 If $\pi^p_1$ and $\pi^p_2$ are isomorphic as representations  of $G(\AAA^{p,\infty})$, then $\pi_{1,p}\simeq \pi_{2,p}$, and   $\pi_2\otimes\pi_{2,\infty}$ admits  a base change to a cuspidal automorphic representation of $\GL_{n}(\AAA_E)\times \AAA_{E_0}^{\times}$; in particular, $\pi_2$ satisfies Hypothesis~\ref{H:automorphic assumption} for $a_{\bullet}=(r_2,s_2)$.
\end{lemma}
\begin{proof}
By assumption on $\pi_1$, there exists an irreducible admissible representation $\pi_{1,\infty}$ of $G_{(r_1,s_1)}(\RR)$ such that $\pi_1\otimes\pi_{1,\infty}$ is a cuspidal automorphic representation of $G_{r_1,s_1}(\AAA)$, which base changes to a cuspidal automorphic representation $(\Pi_1, \chi_1)$ of $\GL_n(\AAA_E)\times\AAA_{E_0}^{\times}$. 
On the other hand, by \cite[Theorem~1.1]{shin}, there exists always a  base change of $\pi_{2}\otimes\pi_{2,\infty}$ to an automorphic representation $(\Pi_2,\chi_2)$ of $\GL_n(\AAA_{E})\times \AAA_{E_0}^{\times}$. The base changes $(\Pi_i,\chi_i)$ with $i=1,2$ satisfy the following properties:
\begin{itemize}
\item  $\Pi_i$ is conjugate self-dual,
\item for every unramified rational prime $x$, the $x$-component of $(\Pi_i,\psi_i)$ depends only on the $x$-component of $\pi_{i}$, and
\item if one writes $\pi_{i,p}=\pi_{i,0}\otimes \pi_{i,\gothp}$ as representation of $G(\Q_p)\simeq\Q_p^{\times}\times \GL_n(E_{\gothp})$, then  $\Pi_{i,p}=(\pi_{i,\gothp}\otimes \check\pi_{i,\gothp}^c)$ as representation of $\GL_{n}(E\otimes \Q_p)\cong \GL_n(E_{\gothp})\times \GL_n(E_{\bar\gothp})$, and $\psi_{i,p}= \pi_{i,0}\otimes \pi_{i,0}^{-1}$ as representation of $(E_{0}\otimes \Q_p)^{\times}=\Q_p^{\times}\times \Q_p^{\times}$. Here, $\check\pi_{i,\gothp}^c$  denotes the complex conjugate of the contragredient of $\pi_{i,\gothp}$.
\end{itemize}
As $\pi^p_1\simeq \pi_2^p$,  $(\Pi_{1},\psi_1)$ and $(\Pi_{2},\psi_2)$ are isomorphic at almost all finite places. 
By the strong multiplicity one theorem for $\GL_n$ \cite{jacquet-shalika}, we have $(\Pi_1,\psi_1)\simeq (\Pi_2,\psi_2)$; in particular, $(\Pi_2,\psi_2)$ is cuspidal. By the description of $(\Pi_{i,p}, \psi_{i,p})$, it follows immediately that $\pi_{1,p}\simeq \pi_{2,p}$.
\end{proof}

Let $\scrA_K$ be the set of isomorphism classes of irreducible admissible  representations $\pi$ of $G(\AAA^{\infty})$ satisfying Hypothesis~\ref{H:automorphic assumption} with $a_{\bullet}=(0,n)$. In particular, each $\pi\in \scrA_K$ is the finite part of an automorphic cuspidal representation of $G_{0,n}(\AAA)$.

We fix such a  $\pi\in \scrA_{K}$.
Let
 \[
 \JL_{\pi}\colon \bigoplus_{i=1}^n H^{0}_\et\big(\overline\Sh_{0,n},\Qlb\big)_{\pi^p}\longrightarrow H_\et^{2(n-1)}\big(\overline\Sh_{1,n-1},\Qlb(n-1)\big)_{\pi^p}
 \]
 denote
the homomorphism on the $(\pi^p)^{K^p}$-isotypic components induced by $\JL$, where for an $\scrH(K^p, \overline \QQ_\ell)$-module $M$ we put
\[
M_{\pi^p}: = \Hom_{\scrH(K^p, \overline \QQ_\ell)}((\pi^p)^{K^p}, M) \otimes (\pi^p)^{K^p}.
\]
Then Lemma~\ref{L:automorphic-prime-to-p} implies that $\pi$ is completely determined by its prime-to-$p$ part. 
Hence, taking the $\pi^p$-isotypic components is the same as taking the $\pi$-isotypic components. We can thus write  $M_{\pi}$ instead of $M_{\pi^p}$ for a $\scrH(K,\overline\QQ_{\ell})$-module $M$.

 Recall that the image of $\JL_{\pi}$ is included  in  $H_\et^{2(n-1)}\big(\overline\Sh_{1,n-1},\Qlb(n-1)\big)_{\pi}^{\mathrm{fin}}$, which is the maximal subspace of $H_\et^{2(n-1)}\big(\overline\Sh_{1,n-1},\Qlb(n-1)\big)_{\pi}$ where the action of $\Gal(\Fpb/\F_{p^2})$ factors through a finite quotient.
 Note that, at this moment, it is not clear if the target of $\JL_\pi$ is nonzero. But this will follow from the proof of our main Theorem~\ref{T:main-theorem} below.

Our main result claims that this inclusion is actually an equality under certain genericity  conditions on $\pi_{p}$. To make this precise, write $\pi_{p}=\pi_{p,0}\otimes \pi_{\gothp}$ as a representation of $G(\Q_p)\simeq\Q_p^{\times}\times \GL_n(E_{\gothp})$.
 Let
\[
 \rho_{\pi_{\gothp}}: W_{\Q_{p^2}}\ra \GL_n(\Qlb)
\]
  be the unramified representation of   the Weil group of $\Q_{p^2}$ defined in \eqref{E:defn-rho-pi}.
   It induces a continuous $\ell$-adic representation of $\Gal(\Fpb/\F_{p^2})$, which we denote by the same notation.
Then $\rho_{\pi_{\gothp}}(\Frob_{p^2})$ is semisimple with characteristic polynomial \eqref{E:hecke-polynomial}.
Using this, we get an explicit description of  $H_\et^{2(n-1)}\big(\overline\Sh_{1,n-1},\Qlb(n-1)\big)_{\pi}$ and $H^{0}_\et(\overline\Sh_{0,n},\Qlb)_{\pi}$ in terms of $\rho_{\pi_{\gothp}}$ by \eqref{E:decomposition-cohomology} and \eqref{E:Galois-Sh}.

 We can now state our main theorem.

  \begin{theorem}\label{T:main-theorem}
  Fix a $\pi$ in $\scrA_K$.
   Let $\alpha_{\pi_{\gothp},1}, \dots, \alpha_{\pi_{\gothp, n}}$ denote the eigenvalues of $\rho_{\pi_{\gothp}}(\Frob_{p^2})$. 
 \begin{enumerate}
 \item If $\alpha_{\pi_{\gothp},1}, \dots, \alpha_{\pi_{\gothp},n}$ are distinct, then the map $\JL_{\pi}$ is injective,
 
\item Let $m_{1,n-1}(\pi)$ (resp. $m_{0,n}(\pi)$) denote the multiplicity for $\pi$ appearing in Theorem~\ref{T:galois-action} for $a_{\bullet}=(1,n-1)$ (resp. for $a_{\bullet}=(0,n)$). Assume that $m_{1,n-1}(\pi)=m_{0,n}(\pi)$ and that $\alpha_{\pi_{\gothp},i}/\alpha_{\pi_{\gothp},j}$ is not a root of unity for all $1\leq i,j\leq n$. Then the map
\[
\JL_{\pi}:\bigoplus_{j=1}^n H^{0}_\et\big(\overline\Sh_{0,n},\Qlb\big)_{\pi}\longrightarrow H_\et^{2(n-1)}\big(\overline\Sh_{1,n-1},\Qlb(n-1)\big)_{\pi}^{\mathrm{fin}}
\]
is an isomorphism. In other words, $H^{2(n-1)}_\et\big(\overline\Sh_{1,n-1},\Qlb(n-1)\big)_{\pi}^{\mathrm{fin}}$ is generated by the  cycle classes of the irreducible components of $Y_{j}$ for $1\leq j\leq n$.
 \end{enumerate}
 \end{theorem}
 The proof of this theorem will be given at the end of Section~\ref{Sect:intersection-matrix}.
 \begin{remark}\label{R:automorphic-mult}
 The equality  $m_{1,n-1}(\pi)=m_{0,n}(\pi)$ is conjectured to be  true according to Arthur's formula on the automorphic multiplicities of unitary groups, and is known to hold when $\pi$ is the finite part of an automorphic representation of $G_{1,n-1}(\AAA)$ whose base change to  $\GL_n(\AAA_E)\times \AAA_{E_0}^{\times}$ is cuspidal, and $G_{1,n-1}$ is quasi-split at all finite places. See for instance \cite[Theorem E]{white}.
 
Conversely, Theorem~\ref{T:main-theorem}(1) gives partial results towards the equality $m_{1,n-1}(\pi)=m_{0,n}(\pi)$. 
Indeed, when combining with  Kottwitz' description \ref{T:galois-action} of the $\pi$-isotypic components of the cohomology groups, Theorem~\ref{T:main-theorem}(1) implies (under the assumption that the Satake parameters of $\pi_{\gothp}$ are distinct) that $m_{1,n-1}(\pi) \geq m_{0,n}(\pi)$ without using Arthur's trace formula .  
If we use only the fact that $\JL_{\pi}$ is non-zero (which is an easy consequence of our computation of the intersection matrix in Theorem~\ref{Th-Conj}),  we get the implication $m_{0,n}(\pi)\neq 0\Rightarrow m_{1,n-1}(\pi)\neq 0$. 
 \end{remark}




    \section{Fundamental intersection numbers}\label{S:fundamental-intersection-number}
In this section, we will compute some intersection numbers on certain Deligne--Lusztig varieties. These numbers  will play a key role in the computation  in the next section of the intersection matrix of the cycles $Y_{j}$ on $\Sh_{1,n-1}$.

\begin{notation}
Let $X$ be  an algebraic variety  of pure  dimension $N$ over $\Fpb$.  For an integer $r\geq 0$, let $A^r(X)$  (resp. $A_r(X)$)  denote the group of algebraic cycles on $X$ of codimension $r$ (resp. of dimension $r$) modulo rational equivalences.
If $Y\subseteq X$ is a subscheme equidimensional of codimension $r$, we denote by $[Y]\in A^r(X)$ the class of $Y$. We put $A^{\star}(X)=\bigoplus_{r=0}^{N}A^r(X)$.      For a zero-dimensional cycle $\eta\in A^N(X)$, we denote by
\[
\deg(\eta)=\int_{X}\eta
\]
the \emph{degree} of $\eta$. Let $\calV$ be a vector bundle over $X$. We denote by $c_r(\calV)\in A^r(X)$ the $r$-th Chern class of $\calV$ for $0\leq r\leq N$, and put $c(\calV)=\sum_{r=0}^Nc_r(\calV)t^r$ in the free variable $t$.

\end{notation}

   \subsection{A special Deligne--Lusztig variety}\label{S:grassmannian}
   We fix an integer $n\geq 1$. For an integer $0\leq k\leq n$, we denote by $\Gr(n,k)$ the Grassmannian variety over $\FF_p$ classifying  $k$-dimensional subspaces of $\FF_p^{\oplus n}$.
Given an integer $k$ with   $1\leq k\leq n$, let $Z_k^{\langle n \rangle}$ be the   subscheme   of $\Gr(n,k)\times \Gr(n,k-1)$ whose $S$-valued points are  isomorphism classes of pairs $(L_1,L_2)$, where
 $L_1$ and $L_2$ are respectively subbundles of $ \cO_{S}^{\oplus n}$  of rank $k$ and $k-1$ satisfying $L_2\subseteq L_1^{(p)}$ and $L_2^{(p)}\subseteq L_1$ (with locally free quotients).
 The same arguments as in Proposition~\ref{P:Yj-prime} show that $Z_k^{\langle n \rangle}$ is a  smooth variety over $\FF_p$ of dimension $n-1$.
 We denote the natural closed immersion by $$i_k: Z_k^{\langle n \rangle}\hra \Gr(n,k)\times\Gr(n,k-1).$$

  Let $\calL_1$ and $\calL_2$ denote  the universal subbundles on $\Gr(n,k) \times \Gr(n,k-1)$ coming from the two factors, and  $\calQ_1$ and $\calQ_2$  the universal quotients, respectively. When there is no confusion, we still use $\calL_i$ and $ \calQ_i$ for $i=1,2$ to denote their restrictions to $Z_k^{\langle n \rangle}$.
 We put
 \begin{equation}\label{E:excess-bundle}
  \calE_k= (\calL_1/\calL_2^{(p)})^*\otimes \calL_2^{(p)}\oplus (\calL_1^{(p)}/\calL_2)\otimes \calQ_1^{*,(p)},
  \end{equation}
 which is a vector bundle of rank $n-1$ on $Z_k^{\langle n \rangle}$.
(This vector bundle is modeled on the description of the normal bundle $N_{Y_j}(\Sh_{1,n-1})$ in Proposition~\ref{P:normal-bundle}(1), which is how our computation will be used in the next section; see Proposition~\ref{P:intersection-cycles}.)
  We have the top Chern class  $c_{n-1}(\calE_k)\in A^{n-1}(Z_k^{\langle n \rangle})$.
  We define the \emph{fundamental intersection number} on $Z_k^{\langle n \rangle}$ as 
  \begin{equation}\label{D:fundamental-number}
  N(n,k):=\int_{Z_k^{\langle n \rangle}} c_{n-1}(\calE_k).
 \end{equation}
 The main theorem we prove in this section is the following.

\begin{theorem}\label{T:intersection}
 For integers $n,r$ with $0\leq r\leq n$, let
\[
\binom{n}{r}_{q}=\frac{(q^n-1)(q^{n-1}-1)\cdots (q^{n-r+1}-1)}{(q-1)(q^2-1)\cdots (q^r-1)}
\]
be the Gaussian binomial coefficients, and let $d(n,k)=(2k-1)n-2k(k-1)-1$ denote the dimension of $\Gr(n,k)\times \Gr(n,k-1)$.
Then, for $1\leq k\leq n$, we have
\begin{equation}
\label{E:fundamental intersection formula}
N(n,k)=(-1)^{n-1}\sum_{\delta=0}^{\min\{k-1, n-k\}} (n-2\delta) p^{d(n-2\delta, k-\delta)}\binom{n}{\delta}_{p^2}.
\end{equation}
\end{theorem}

\begin{remark}
We point out that this theorem seems to be more than a technical result.
It is at the heart of the understanding of these cycles we constructed. 
\end{remark}

\begin{proof}
We first claim that $N(n,k) = N(n,n+1-k)$ for $1 \leq k\leq n$.
Let  $(L_1,L_2)$ be an $S$-valued point of $\Gr(n,k)\times \Gr(n,k-1)$, and $Q_i=\cO_S^{\oplus n}/L_i$ for $i=1,2$ be the corresponding quotient bundles.
Then $(L_1,L_2)\mapsto (Q_2^*,Q_1^*)$ defines a duality isomorphism
\[
\theta \colon \Gr(n,k)\times \Gr(n,k-1)\xra{\sim} \Gr(n,n+1-k)\times \Gr(n,n-k).
\]
Since $L_2^{(p)}\subseteq L_1$ (resp. $L_2\subseteq L_1^{(p)}$) is equivalent to $Q_1^*\subseteq Q_2^{*,(p)}$ (resp. to $Q_1^{*,(p)}\subseteq Q_2^{*}$), $\theta$ induces an isomorphism between $Z_k^{(n)}$ and $Z_{n+1-k}^{(n)}$.
It is also direct to check that $\calE_k=\theta^*(\calE_{n+1-k})$.  This verifies the claim.
Now since the right hand side of \eqref{E:fundamental intersection formula} is also invariant under replacing $k$ by $n+1-k$, it suffices to prove the Theorem when $k \leq \frac{n+1}{2}$.

We reduce the proof of the theorem to an analogous situation where the twists are given on one of the $L_i$'s.
Let $\tilde Z_k^{\langle n \rangle}$ be the subscheme of $\Gr(n,k)\times \Gr(n,k-1)$ whose $S$-valued points are the isomorphism classes of pairs $(\tilde L_1,\tilde L_2)$, where
 $\tilde L_1$ and $\tilde L_2$ are respectively subbundles of $ \cO_{S}^{\oplus n}$  of rank $k$ and $k-1$ satisfying $\tilde L_2\subseteq \tilde L_1$ and $\tilde L_2^{(p^2)}\subseteq \tilde L_1$.
The relative Frobenius morphisms on the two Grassmannian factors induce two morphisms
\[
\xymatrix@R=0pt{
Z_k^{\langle n \rangle} \ar[r]^{\varphi} & \tilde Z_k^{\langle n \rangle}\ar[r]^{\hat \varphi} & (Z_k^{\langle n \rangle})^{(p)}
\\
(L_1, L_2) \ar@{|->}[r] &  (L_1^{(p)}, L_2)
\\
& (\tilde L_1, \tilde L_2) \ar@{|->}[r] &  (\tilde L_1, \tilde L_2^{(p)}),
}
\]
such that the composition is the relative Frobenius on $\tilde Z_k^{\langle n\rangle}$. 
Using a simple deformation computation, we see that $\varphi$ has degree $p^{n-k}$ and $\hat \varphi$ has degree $p^{k-1}$.
Let $\tilde \calL_1$ and $\tilde  \calL_2$ denote the universal subbundles on $\Gr(n,k) \times \Gr(n,k-1)$ when restricted to $\tilde Z_k^{\langle n \rangle}$; let $\tilde \calQ_1$ and $\tilde \calQ_2$ denote the universal quotients, respectively.
 We put
 \begin{equation}
\tilde  \calE_k= (\tilde \calL_1/\tilde \calL_2^{(p^2)})^* \otimes \tilde \calL_2^{(p^2)}\oplus (\tilde \calL_1/\tilde \calL_2)\otimes \tilde \calQ_1^{*},
  \end{equation}
which is a vector bundle of rank $n-1$ on $\tilde Z_k^{\langle n \rangle}$.

Note that
\[
\varphi^*(\tilde \calE_k) = 
(\calL_1^{(p)}/\calL_2^{(p^2)})^* \otimes \calL_2^{(p^2)}\oplus (\calL_1^{(p)}/\calL_2)\otimes \calQ_1^{*,(p)}.
\]
Comparing with $\calE_k$, we see that $c_{n-1}(\varphi^*(\tilde \calE_k)) = p^{k-1} c_{n-1}(\calE_k)$, where the factor $p^{k-1}$ comes from the Frobenius twist on the first factor.
Thus, we have
\begin{align}
\label{E:one side untwist}
\int_{\tilde Z_k^{\langle n \rangle}}c_{n-1}(\tilde \calE_k)& = (\deg \varphi)^{-1} \int_{Z_k^{\langle n \rangle}} c_{n-1}(\varphi^*(\tilde \calE_k))\\
&
= p^{k-n} \int_{Z_k^{\langle n \rangle}} p^{k-1} c_{n-1}(\calE_k) = p^{2k-n-1} 
N(n,k).\nonumber
\end{align}
Since $d(n-2\delta,k-\delta) + 2k-n-1 = 2(k-\delta-1)(n-k-\delta+1)$, the Theorem is in fact equivalent to the following (for each fixed $k$).
\end{proof}

\begin{proposition}
\label{P:untwist intersection}
For $1 \leq k \leq \frac{n+1}{2}$, we have
\begin{equation}
\label{E:intersection formula untwisted}
\int_{\tilde Z_k^{\langle n \rangle}} c_{n-1}(\tilde \calE_k) = 
(-1)^{n-1} \sum_{\delta=0}^{k-1} (n-2\delta) p^{2(k-\delta-1)(n-k-\delta +1)}\binom{n}{\delta}_{p^2}.
\end{equation}
\end{proposition}

\begin{remark}
\label{R:equivalent untwist}
Before giving the proof of this proposition, we point out a variant of the construction of $\tilde Z_k^{\langle n\rangle}$.
Let $\tilde Z'^{\langle n\rangle}_k$ be the subscheme  of $\Gr(n,k)\times \Gr(n,k-1)$ whose $S$-valued points are the isomorphism classes of pairs $(\tilde L'_1,\tilde L'_2)$, where
 $\tilde L'_1$ and $\tilde L'_2$ are respectively subbundles of $ \cO_{S}^{\oplus n}$  of rank $k$ and $k-1$ satisfying $\tilde L'_2\subseteq \tilde L'_1$ and $\tilde L'_2\subseteq \tilde L'^{(p^2)}_1$ (Note that the twist is on $L'_1$ as opposed to be on $L'_2$).  This is again a certain partial-Frobenius twist of $Z_k^{\langle n \rangle}$; it is smooth of dimension $n-1$.
Define the universal subbundles and quotient bundles $\tilde \calL'_1$, $\tilde \calL'_2$, $\tilde \calQ'_1$, and $\tilde \calQ'_2$ similarly.
 We put 
 \[
 \tilde \calE'_k = (\tilde \calL'_1 / \tilde \calL'_2)^* \otimes \tilde \calL'_2 \oplus (\tilde \calL'^{(p^2)}_1 / \tilde \calL'_2) \otimes ( \tilde \calQ'^*_1)^{(p^2)}.
 \]
Using the same argument as above, we see that, for every fixed $k$, 
\[
\int_{\tilde Z'^{\langle n \rangle}_k} c_{n-1}(\tilde \calE'_k) = p^{n+1-2k} N(n,k).
  \]
  Note that 
the exponent is different from \eqref{E:one side untwist}.
So Proposition~\ref{P:untwist intersection} for each fixed $k$ is equivalent to
\[
\int_{\tilde Z'^{\langle n \rangle}_k} c_{n-1}(\tilde \calE'_k) = (-1)^{n-1} \sum_{\delta=0}^{k-1} (n-2\delta) p^{2(k-\delta)(n-k-\delta )}\binom{n}{\delta}_{p^2},
\]
as $2(k-\delta)(n-k-\delta)  = d(n-2\delta, k-\delta) + n-2k+1$.
\end{remark}

\begin{proof}[Proof of Proposition~\ref{P:untwist intersection}]
We first prove it in the case of $k=1,2$ and then we explain an inductive process to deal with the general case.

When $k=1$, $\tilde Z_1^{\langle n \rangle}$ classifies a line-subbundle $\tilde L_1$ of $\calO_S^{\oplus n}$ with no additional condition (as $\tilde L_2$ is zero); so $\tilde Z_1^{\langle n\rangle} \cong \PP^{n-1}$ and $\tilde \calL_1 = \calO_{\PP^{n-1}}(-1)$.  The vector bundle $\tilde \calE_1$ is equal to $\tilde \calL_1 \otimes \tilde \calQ_1^*$.
It is straightforward to check that 
\[
c(\tilde  \calE_1) = \big( 1+ c_1(\calO_{\PP^{n-1}}(-1))\big)^n \quad \textrm{and hence} \quad \int_{\tilde Z_1^{\langle n \rangle}}  c_{n-1}(\tilde \calE_1) = (-1)^{n-1}n;
\] the Proposition is proved in this case.

When $k=2$, we consider a forgetful morphism
\[
\xymatrix@R=0pt{
\psi: \quad \tilde Z_2^{\langle n \rangle} \ar[r] & \tilde Z_1^{\langle n \rangle}
\\
\quad (\tilde L_1, \tilde L_2) \ar@{|->}[r] &
\tilde L_2.
}
\]
This morphism is an isomorphism over the closed points $x \in \tilde Z_1^{\langle n\rangle}(\overline \FF_p)$ for which $\tilde L_{2,x} \neq \tilde L_{2,x}^{(p^2)}$, because in this case, $\tilde L_{1,x}$ is forced to be $\tilde L_{2,x} + \tilde L_{2,x}^{(p^2)}$.  On the other hand, for a closed point $x \in \tilde Z_1^{\langle n\rangle}(\overline \FF_p)$ where $\tilde L_{2,x} = \tilde L_{2,x}^{(p^2)}$, i.e. for $x \in \tilde Z_1^{\langle n\rangle}(\FF_{p^2}) \cong \PP^{n-1}(\FF_{p^2})$, $\psi^{-1}(x)$ is the space classifying a line $\tilde L_1$ in $\overline \FF_p^{\oplus n} / \tilde L_{2,x}$; so $\psi^{-1}(x) \simeq \PP^{n-2}$.
A simple tangent space computation shows that $\psi$ is the blow-up morphism of $\tilde Z_1^{\langle n\rangle} \cong \PP^{n-1}$ at all of its $\FF_{p^2}$-points.
We use $E$ to denote the exceptional divisors, which is a disjoint union of $\binom n1_{p^2}$ copies of $\PP^{n-2}$.

Note that the vanishing of the morphism  $\tilde \calL_2 \to \tilde \calL_1 / \tilde \calL_2^{(p^2)}$ defines the divisor  $E$ (as we  can  see using deformation); so 
\[
\calO_{\tilde Z_2^{\langle n \rangle}}(E)\cong \tilde \calL_1 / \tilde \calL_2^{(p^2)} \otimes\tilde  \calL_2^{-1}.
\]
Put $\eta = c_1(\tilde \calL_2) = \psi^* c_1\big(\calO_{\PP^{n-1}}(-1)\big)$ and $\xi = c_1(E)$.
Then
\begin{align}
\nonumber
c(\tilde \calE_2) & = c \big((\tilde \calL_1 / \tilde \calL_2^{(p^2)} )^* \otimes \tilde \calL_2^{(p^2)} \big) 
\cdot
c\big(( \tilde \calL_1 / \tilde \calL_2) \otimes \tilde \calQ_1^*\big)
\\
\label{E:Chern class k=2}
&= (1- \xi + (p^2-1)\eta) \cdot (1+ \xi + p^2 \eta)^n / (1+ \xi + (p^2-1)\eta),
\end{align}
where the computation of the second term comes from the following two exact sequences
\[
0 \to (\tilde \calL_1 / \tilde \calL_2) \otimes \tilde \calQ_1^* \to (\tilde \calL_1 / \tilde \calL_2)^{\oplus n}
\to (\tilde \calL_1 / \tilde \calL_2) \otimes \tilde \calL_1^* \to 0;
\]
\[
0 \to \calO_{\tilde Z_2^{\langle n \rangle}} \to (\tilde \calL_1 / \tilde \calL_2) \otimes\tilde  \calL_1^* \to (\tilde \calL_1 / \tilde \calL_2) \otimes \tilde \calL_2^* \to 0.
\]
Note that $\int_{\tilde Z_2^{\langle n \rangle}} \xi^i \eta^j = 0$ unless $(i,j) = (n-1, 0)$ or $(0, n-1)$, in which case we have  
$$
\int_{\tilde Z_2^{\langle n\rangle}}\eta^{n-1}=(-1)^{n-1}\quad \text{ and }\quad \int_{\tilde Z_2^{\langle n\rangle}}\xi^{n-1}=(-1)^n\binom{n}{1}_{p^2}.
$$
Here, to prove the last formula, we used the fact that the restriction of $\cO_{\tilde Z^{\langle n\rangle}_2}(E)$ to each irreducible component  $\PP^{n-2}$ of $E$ is isomorphic to $\cO_{\PP^{n-2}}(-1)$.
So it suffices to compute 
\begin{itemize}
\item
the $\xi^{n-1}$-coefficient of \eqref{E:Chern class k=2}, which is the same as the $\xi^{n-1}$-coefficient of $(1-\xi)(1+\xi)^{n-1}$ and is equal to $2-n$; and

\item 
the  $\eta^{n-1}$-coefficient of \eqref{E:Chern class k=2}, which is the same as the $\eta^{n-1}$ coefficient of $(1+(p^2-1)\eta) (1+p^2\eta)^n / (1+(p^2-1)\eta) = (1+p^2\eta)^n$ and is equal to $n p^{2(n-1)}$.
\end{itemize}
To sum up, we see that
\[
\int_{\tilde Z_2^{\langle n\rangle}} c_{n-1}(\tilde \calE_2) = (-1)^{n-1} np^{2(n-1)} + (-1)^n(2-n) \binom n1_{p^2},
\]
which is exactly \eqref{E:intersection formula untwisted} for $k=2$.

In general, we make an induction on $k$. 
Assume that the Proposition is proved for $k-1 \geq 1$ and we now prove the Proposition for $k$ (assuming that $k \leq \frac{n+1}{2}$).
By Remark~\ref{R:equivalent untwist}, we get the similar intersection formula for $\tilde \calE'_{k-1}$ on $\tilde Z'^{\langle n\rangle}_{k-1}$: 
\begin{equation}
\label{E:induction hypo intersection formula untwisted}
\int_{\tilde Z'^{\langle n \rangle}_{k-1}} c_{n-1}(\tilde \calE'_{k-1}) = 
(-1)^{n-1} \sum_{\delta=0}^{k-2} (n-2\delta) p^{2(k-\delta-1)(n-k-\delta+1)}\binom{n}{\delta}_{p^2}.
\end{equation}

We consider the moduli space $W$ over $\FF_{p^2}$ whose $S$-points are tuples $(\tilde L_1, \tilde L_2 = \tilde L'_2, \tilde L'_3)$, where $\tilde L_1$, $\tilde L_2$, and $\tilde L'_3$ are  respectively subbundles of $\calO_S^{\oplus n}$ of rank $k$, $k-1$, and $k-2$ satisfying $\tilde L'_3 \subset\tilde  L_2 \subset\tilde  L_1$ and $\tilde L'_3 \subset\tilde  L^{(p^2)}_2 \subset\tilde  L_1$.
It is easy to use deformation theory to check that $W$ is a smooth variety of dimension $n-1$.  There are two natural morphisms
\[
\xymatrix{
& W \ar[rd]^{\psi_{23}} \ar[ld]_{\psi_{12}} &&& (\tilde L_1, \tilde L_2 = \tilde L'_2, \tilde L'_3)
\ar@{|->}[dl]
\ar@{|->}[dr]
\\ \tilde Z_k^{\langle n \rangle} &&
\tilde Z'^{\langle n \rangle}_{k-1} & (\tilde L_1, \tilde L_2) && (\tilde L'_2, \tilde L'_3).
}
\]
Let $E$ denote the subspace of $W$ whose closed points $x \in W(\overline \FF_p)$ are those such that $\tilde L_{2,x} = \tilde L_{2,x}^{(p^2)}$, i.e. $\tilde L_{2,x}$ is an $\FF_{p^2}$-rational subspace of $\FF_{p^2}^{\oplus n}$ of dimension $k-1$.
It is clear that $E$ is a disjoint union of $\binom n{k-1}_{p^2}$ copies (corresponding to the choices of $\tilde L_2$) of $\PP^{n-k} \times \PP^{k-2}$ (corresponding to the choice of $\tilde L_1$ and $\tilde L'_3$ respectively).  It gives rise to a smooth divisor on $W$.

For a point $x \in (W\backslash E)(\overline \FF_p)$, we have $\tilde L_{2,x} \neq \tilde L^{(p^2)}_{2,x}$ and hence it uniquely determines both $\tilde L_{1,x}$ and $\tilde L'_{3,x}$; so $\psi_{12}$ and $\psi_{23}$ are isomorphisms restricted to $W \backslash E$.
On the other hand, when restricted to $E$, $\psi_{12}$ contracts each copy of $\PP^{n-k} \times \PP^{k-2}$ of $E$ into the first factor $\PP^{n-k}$; whereas $\psi_{23}$ contracts each copy of $\PP^{n-k} \times \PP^{k-2}$ of $E$ into the second factor $\PP^{k-2}$.
It is clear from this (with a little bit of help from a deformation argument) that $\psi_{12}$ is the blow-up of $\tilde Z_k^{\langle n \rangle}$ along $\psi_{12}(E)$ and $\psi_{23}$ is the blow-up of $\tilde Z'^{\langle n \rangle}_{k-1}$ along $\psi_{23}(E)$; the divisor $E$ is the exceptional divisor for both blow-ups.

A simple deformation theory shows that the normal bundle  of $E$ in $W$ when restricted to each component $\PP^{n-k} \times \PP^{k-2}$ is $\calO_{\PP^{n-k}}(-1) \otimes \calO_{\PP^{k-2}}(-1)$.
Moreover, we can characterize $E$ as the zero locus of either one of the following natural homomorphisms
\[
 \tilde \calL_2^{(p^2)} / \tilde \calL'_3 \longrightarrow \tilde \calL_1 / \tilde \calL_2 , \quad \tilde \calL_2 / \tilde \calL'_3  \longrightarrow\tilde \calL_1 / \tilde \calL_2^{(p^2)} .
\]
So as a line bundle over $W$, we have 
\[
\calO_W(E) \cong (\tilde \calL_2^{(p^2)} / \tilde \calL'_3)^{-1} \otimes (\tilde \calL_1 / \tilde \calL_2) \cong (\tilde \calL_2 / \tilde \calL'_3)^{-1} \otimes (\tilde \calL_1 / \tilde \calL_2^{(p^2)}).
\]

We want to compare
\begin{equation}
\label{E:reduction from Z to W}
\int_{\tilde Z_k^{\langle n\rangle}} c_{n-1}(\tilde \calE_k) = \int_W c_{n-1}\big(\psi_{12}^*(\tilde \calE_k)\big)
 \quad \textrm{and} \quad \int_{\tilde Z'^{\langle n\rangle}_{k-1}} c_{n-1} (\tilde \calE'_{k-1}) = \int_W c_{n-1} \big(\psi_{23}^*(\tilde \calE'_{k-1})\big).
\end{equation}
We will show that they differ by $(2k-n-2)(-1)^n \binom n{k-1}_{p^2}$ and this will conclude the proof of the Proposition by inductive hypothesis~\eqref{E:induction hypo intersection formula untwisted}.
Indeed, we have
\begin{align}
\label{E:expression of psi12}
c\big(\psi_{12}^*(\tilde \calE_k)\big) &=
c\big( (\tilde \calL_1/\tilde \calL_2^{(p^2)})^* \otimes \tilde \calL_2^{(p^2)}\big) \cdot c\big( (\tilde \calL_1/\tilde \calL_2)\otimes \tilde \calQ_1^{*}\big),
 \quad \textrm{and}\\\label{E:expression of psi23}
 c\big(\psi_{23}^*(\tilde \calE'_{k-1})\big) & = c\big(( \tilde \calL_2 / \tilde \calL'_3)^* \otimes \tilde \calL'_3\big) \cdot c \big( (\tilde \calL^{(p^2)}_2 / \tilde \calL'_3) \otimes \tilde \calQ_2^{*,(p^2)}\big),
\end{align}
where $\tilde \calQ_1$ and $\tilde \calQ_2$ are the universal quotient vector bundles.
Consider the following two exact sequences where the two last terms are identified:
\[
\xymatrix{
& & \cO_W(E) \otimes (\tilde \calL_1/ \tilde \calL_2^{(p^2)})^{-1} \otimes\tilde \calL_2^{(p^2)} \ar@{<->}[d]^\cong
\\
0 \ar[r] &
(\tilde \calL_2 / \tilde \calL'_3 )^{-1} \otimes \tilde \calL'_3
\ar[r] &
(\tilde \calL_2 / \tilde \calL'_3 )^{-1} 
\otimes \tilde \calL_2^{(p^2)}
\ar[r] &
(\tilde \calL_2 / \tilde \calL'_3 )^{-1} \otimes (\tilde \calL_2^{(p^2)} / \tilde \calL'_3)
\ar[r]  \ar@{<->}[d]^\cong & 0
\\
0 \ar[r] &
(\tilde \calL_1 / \tilde \calL_2) \otimes \tilde \calQ_1^*
\ar[r] &
(\tilde \calL_1 / \tilde \calL_2) \otimes \tilde \calQ_2^{*, (p^2)}\ar@{<->}[d]^\cong
\ar[r] &
(\tilde \calL_1 / \tilde \calL_2) \otimes (\tilde \calQ_2^{*, (p^2)} / \tilde \calQ_1^*)
\ar[r] & 0.
\\
& &\cO_W(E) \otimes ( \tilde \calL_2^{(p^2)} / \tilde \calL'_3) \otimes \tilde \calQ_2^{*,(p^2)}
}
\]
Here the right vertical isomorphism is given by
\begin{align*}
&
(\tilde \calL_1 / \tilde \calL_2) \otimes (\tilde \calQ_2^{*, (p^2)} / \tilde \calQ_1^*) \cong (\tilde \calL_1 / \tilde \calL_2) \otimes (\tilde \calL_1 / \tilde \calL_2^{(p^2)})^{-1} 
\\
\cong\; & \big((\tilde \calL_2^{(p^2)} / \tilde \calL'_3) \otimes \cO_W(E) \big) \otimes \big( (\tilde \calL_2 / \tilde \calL'_3) \otimes \cO_W(E)\big)^{-1} \cong (\tilde \calL_2 / \tilde \calL'_3 )^{-1} \otimes (\tilde \calL_2^{(p^2)} / \tilde \calL'_3).
\end{align*}
From these two exact sequences we see that
\[
c\big((\tilde \calL_2/ \tilde \calL'_3)^{-1} \otimes\tilde \calL'_3\big) \cdot c \big(\cO_W(E)\otimes (\tilde \calL_2^{(p^2)} / \tilde \calL'_3) \otimes \tilde \calQ_2^{*, (p^2)} \big) = c\big(( \tilde \calL_1 / \tilde \calL_2) \otimes \tilde \calQ_1^*\big) \cdot c \big(\cO_W(E) \otimes (\tilde \calL_1 / \tilde \calL_2^{(p^2)} )^{-1} 
\otimes \tilde \calL_2^{(p^2)} \big).
\]
Comparing this with \eqref{E:expression of psi23} and \eqref{E:expression of psi12}, we get
\begin{align*}
& c_{n-1}\big(\psi_{12}^*(\tilde \calE_k)\big) - c_{n-1}\big(\psi_{23}^*(\tilde \calE'_{k-1})\big)\\
=\ &  \Big( c_{k-1} \big((\tilde \calL_1 / \tilde \calL_2^{(p^2)} )^{-1} 
\otimes \tilde \calL_2^{(p^2)} \big) - c_{k-1} \big(\cO_W(E) \otimes (\tilde \calL_1 / \tilde \calL_2^{(p^2)} )^{-1} 
\otimes \tilde \calL_2^{(p^2)} \big)\Big) \cdot c_{n-k}\big( ( \tilde \calL_1 / \tilde \calL_2) \otimes \tilde \calQ_1^*\big)
 \\
&-c_{k-2}\big( (\tilde \calL_2/ \tilde \calL'_3)^{-1} \otimes\tilde \calL'_3 \big)\cdot \Big( c_{n-k+1} \big( (\tilde \calL_2^{(p^2)} / \tilde \calL'_3) \otimes \tilde \calQ_2^{*, (p^2)} \big) - c_{n-k+1} \big(\cO_W(E) \otimes (\tilde \calL_2^{(p^2)} / \tilde \calL'_3) \otimes \tilde \calQ_2^{*, (p^2)} \big) \Big).
\end{align*}
Recall that $E$ is the exceptional divisor for the blow-up $\psi_{12}$ centered at a disjoint union of $\PP^{n-k}$; so $c_1(E)$ kills $\psi_{12}^*(A^i(\tilde Z^{\langle n\rangle}_{k}))$ for $i \geq n-k+1$.
Similarly, $c_1(E)$ kills $\psi_{23}^*(A^i(\tilde Z'^{\langle n\rangle}_{k-1}))$ for $i \geq k-1$.
As a result, we can rewrite the above complicated formula as
\begin{align*}
& c_{n-1}\big(\psi_{12}^*(\tilde \calE_k)\big) - c_{n-1}\big(\psi_{23}^*(\tilde \calE'_{k-1})\big)\\
=\ & - c_1(E)^{k-2}|_E \cdot c_{n-k}\big(( \tilde \calL_1 / \tilde \calL_2) \otimes \tilde \calQ_1^*\big) + c_{k-2}\big( (\tilde \calL_2/ \tilde \calL'_3)^{-1} \otimes\tilde \calL'_3 \big) \cdot c_1(E)^{n-k}|_E
\\
=\ &(-1)^{k-1} c_{n-k}\big( ( \tilde \calL_1 / \tilde \calL_2) \otimes \tilde \calQ_1^*\big)|_{\psi_{12}(E)} + (-1)^{n-k} c_{k-2}\big( (\tilde \calL_2/ \tilde \calL'_3)^{-1} \otimes\tilde \calL'_3 \big)|_{\psi_{23}(E)}.
\end{align*}
For the first term, over each $\PP^{n-k}$ of $\psi_{12}(E)$, it is to take the top Chern class of the canonical subbundle of rank $n-k$ twisted by $\calO_{\PP^{n-k}}(-1)$; so the degree of the first term is $(-1)^{n-k}(n-k+1)$ on each $\PP^{n-k}$.
Similarly, for the second term, over each $\PP^{k-2}$, it is the top Chern class of the canonical subbundle of rank $k-2$ twisted by $\calO_{\PP^{k-2}}(-1)$; so the degree of the second term is $(-1)^{k-2}(k-1)$ on each $\PP^{k-2}$.
To sum up, we have
\begin{align}
\label{E:difference over W}
&\int_W c_{n-1}\big(\psi_{12}^*(\tilde \calE_k)\big) - \int_W c_{n-1}\big(\psi_{23}^*(\tilde \calE'_{k-1})\big)
\\
\nonumber
= \ & (-1)^{k-1}(-1)^{n-k}(n-k+1)\tbinom n{k-1}_{p^2} +(-1)^{n-k}(-1)^{k-2}(k-1) \tbinom n{k-1}_{p^2} \\
\nonumber
 =\ & (-1)^{n-1}(n-2k+2)\tbinom n{k-1}_{p^2}.
\end{align}
So by inductive hypothesis,
\begin{align*}
\int_{\tilde Z_k^{\langle n\rangle}} c_{n-1}(\tilde \calE_k) & \stackrel{\eqref{E:reduction from Z to W}}{=} \int_W c_{n-1}\big(\psi_{12}^*(\tilde \calE_k)\big)
\stackrel{\eqref{E:difference over W}}{=}\int_W c_{n-1}\big(\psi_{23}^*(\tilde \calE_k)\big) + (-1)^{n-1}(n-2k+2)\tbinom n{k-1}_{p^2}.
\\
&\stackrel{\eqref{E:reduction from Z to W}}{=}\int_{\tilde Z'^{\langle n\rangle}_{k-1}} c_{n-1} (\tilde \calE'_{k-1}) + (-1)^{n-1}(n-2k+2)\tbinom n{k-1}_{p^2}
\\
& = (-1)^{n-1} \sum_{\delta=0}^{k-2} (n-2\delta) p^{2(k-\delta-1)(n-k-\delta+1 )}\binom{n}{\delta}_{p^2} + (-1)^{n-1}(n-2k+2)\binom n{k-1}_{p^2}
\\
& = (-1)^{n-1} \sum_{\delta=0}^{k-1} (n-2\delta) p^{2(k-\delta-1)(n-k-\delta+1 )}\binom{n}{\delta}_{p^2}.
\end{align*}
This shows the statement of the Proposition for $k$ and hence concludes the inductive proof.
\end{proof}

\section{Intersection matrix of supersingular cycles on $\Sh_{1,n-1}$}\label{Sect:intersection-matrix}
Throughout this section, we fix an integer $n\geq 2$ and keep the notation as in Section~\ref{Section:U(1,n)}.
We will study the intersection theory of cycles $Y_{j}$ for $1\leq j\leq n$ on $\Sh_{1,n-1}$ considered in Section~\ref{Section:U(1,n)}.
For this, we may assume the following.

\begin{hypothesis}
We assume that the tame level structure $K^p$ is taken sufficiently small so that Lemma~\ref{L:finiteness-auto} holds with $N=2$.
\end{hypothesis}

\subsection{Hecke correspondences on $\Sh_{0,n}$}\label{S:hecke-operator}
  Recall that we have an isomorphism
   \[
   G(\Q_p) \simeq \Q_{p}^{\times}\times \GL_{n}(E_{\gothp})\cong \Q_p^{\times}\times\GL_n(\Q_{p^2}).
   \]
   Put $ K_{\gothp}=\GL_{n}(\calO_{E_{\gothp}})$ and $ K_p=\Z_{p}^{\times}\times  K_{\gothp}$.
    The Hecke algebra $\Z[ K_{\gothp}\backslash \GL_{n}(E_{\gothp})/ K_{\gothp}]$ can be viewed as a subalgebra of $\Z[ K_p\backslash G(\Q_p)/ K_p]$ (with trivial factor at the $\Q_p^{\times}$-component).

   For $\gamma\in \GL_{n}(E_{\gothp})$, the double coset $T_{\gothp}(\gamma):=K_{\gothp} \gamma  K_{\gothp}$ defines a Hecke correspondence on $\Sh_{0,n}$.
It induces a set theoretic Hecke correspondence
\[
T_{\gothp}(\gamma): \Sh_{0,n}(\Fpb)\longrightarrow \calS (\Sh_{0,n}(\Fpb)),
\]
where $\calS(\Sh_{0,n}(\Fpb))$ denotes the set of subsets of $\Sh_{0,n}(\Fpb)$.
By Remark~\ref{R:isogenous-class}, $\Sh_{0,n}(\Fpb)$ is a union of $\#\ker^{1}(\Q,G_{0,n})$-isogeny classes of abelian varieties.
Fix a base point $z_0\in \Sh_{0,n}(\Fpb)$.
Let 
\[
\Theta_{z_0}:\Isog(z_0)\xra{\ \sim\ } G_{0,n}(\Q)\big\backslash\big ( G(\AAA^{\infty,p})\times G(\Q_p)\big)\big/K^p\times K_p.
\]
be the bijection constructed as in Corollary~\ref{C:isogeny-classes}. 
Write $K_{\gothp}\gamma K_{\gothp}=\coprod_{i\in I}  \gamma_iK_{\gothp}$.
If $z\in \Isog(z_0)$ corresponds to the class of $(g^p,g_p)\in G(\AAA^{\infty,p})\times G(\Q_p)$ with $g_p=(g_{p,0}, g_{\gothp})$, then $T_{\gothp}(\gamma)(z)$ consists of points in $\Isog(z_0)$ corresponding to the class of $(g^p,(g_{p,0}, g_{\gothp} \gamma_{i}))$ for all $i\in I$.

Alternatively, $T_{\gothp}(\gamma)$ has the following description. Write $z=(A,\lambda,\eta)$, and let  $\LL_{z}$ denote the $\Z_{p^2}$-free module $\tcD(A)^{\circ,F^2=p}_1$.
Then a point  $z'=(B,\lambda', \eta')\in \Sh_{0,n}(\Fpb)$ belongs to $T_{\gothp}(\gamma)(z)$ if and only if  there exists an $\cO_D$-equivariant $p$-quasi-isogeny $\phi: B'\ra B$ (i.e. $p^m \phi$ is an isogeny of $p$-power order for some integer $m$) such that
\begin{enumerate}
\item $\phi^\vee\circ\lambda\circ\phi=\lambda'$,
\item $\phi\circ\eta'=\eta$,
\item $\phi_{*}(\LL_{z'})$ is a lattice of $\LL_z[1/p]=\LL_z\otimes_{\Z_{p^2}}\Q_{p^2}$ with the property:  there exists a $\Z_{p^2}$-basis $(e_1,\dots, e_n)$ for $\LL_z$ such that $( e_1,\dots,e_n)\gamma$ is a $\Z_{p^{2}}$-basis for $\phi_{*}(\LL_{z'})$.
\end{enumerate}
When $\gamma = \diag(p^{a_1}, \dots, p^{a_n})$ with $a_i \in \{-1, 0, 1\}$, 
For given $z$ and $z'$, such a $\phi$ is necessarily unique if it exists, by Lemma~\ref{L:finiteness-auto} (with $N=2$).
Therefore, $T_{\gothp}(\gamma)(z)$ is in natural bijection  with the set of  $\Z_{p^2}$-lattices $\LL'\subseteq \LL_z[1/p]$ satisfying property (3) above.

For each integer $i$ with $0\leq i\leq n$, we put
$$
T^{(i)}_{\gothp}=T_\gothp\big(\diag(\,\underbrace{p,\dots, p}_i,\, \underbrace{1,\dots,1}_{n-i}\,)\big).
$$
By the discussion above, one has a natural bijection
  \[
  T_\gothp^{(i)}(z)\xra{\sim} \big\{
  \LL_{z'}\subseteq \LL_{z}[1/p]\;\big|\; p\LL_{z}\subseteq \LL_{z'}\subseteq \LL_{z},\; \dim_{\F_{p^2}}(\LL_{z}/\LL_{z'})=i
 \big \}
  \]
  for $z\in \Sh_{0,n}(\Fpb)$.
   Note that  $T_\gothp^{(0)}=\id$, and  we put $S_\gothp:=T_\gothp^{(n)}$.
    Then the Satake isomorphism implies $\Z[K_\gothp\backslash \GL_{n}(E_{\gothp})/K_{\gothp}]\cong \Z[T_{\gothp}^{(1)},\dots, T_\gothp^{(n-1)}, S_\gothp, S_\gothp^{-1}]$.
  More generally, for $0\leq a\leq b\leq n$, we put
   $$
R^{(a,b)}_\gothp=T_\gothp\big(\diag(\,\underbrace{p^2,\dots,p^2}_a,\,\underbrace{p,\dots, p}_{b-a},\,\underbrace{1,\dots, 1}_{n-b}\,)\big).
   $$
Note that $R^{(0,i)}_{\gothp}=T_\gothp^{(i)}$, and $R_\gothp^{(a,b)}S_\gothp^{-1}$ is the Hecke operator  $T_\gothp\big(\diag(\,\underbrace{p,\dots, p}_a,\,\underbrace{ 1,\dots, 1}_{b-a},\,\underbrace{p^{-1},\dots, p^{-1}}_{n-b}\,)\big)$.
For the explicit relations between $R^{(a,b)}_{\gothp}$ and  $T_\gothp^{(i)}$, see Proposition~\ref{P:multiplication-Hecke}.

\subsection{Refined Gysin homomorphism}
For an algebraic variety $X$ over $\Fpb$ of pure dimension $N$ and any integer $r\geq 0$, we write $A_{r}(X)=A^{N-r}(X)$ to denote the group of dimension $r$ (codimension $N-r$) cycles in $X$ modulo rational equivalence.
Recall that the restriction of $\pr_j: Y_j\ra \Sh_{1,n-1}$ to  each $Y_{j,z}$ for $z\in \Sh_{0,n}(\Fpb)$ and  $1\leq j\leq n$   is a regular closed immersion (into $\overline \Sh_{1,n-1}$).
There is a well-defined Gysin homomorphism
  \begin{equation}\label{E:refined-Gysin}
  \pr_j^{!}:  A_{n-1}(\overline\Sh_{1,n-1})\longrightarrow A_0(\overline Y_{j})=\bigoplus_{z\in \Sh_{0,n}(\Fpb)}A_0(Y_{j,z}),
  \end{equation}
whose composition with the natural projection $A_0(\overline Y_{j})\ra A_0(Y_{j,z})$  is the refined Gysin map $(\pr_{j}|_{Y_{j,z}})^!$  defined in \cite[6.2]{fulton} for regular immersions.
Let $X\subseteq\overline  \Sh_{1,n-1}$ be a closed subvariety of dimension $n-1$. Consider the Cartesian diagram
\[
\xymatrix{
\overline Y_{j}\times_{\overline \Sh_{1,n-1}}X\ar[r]^-{g_{X}}\ar[d]_{g_j} & X\ar[d]\\
\overline Y_{j} \ar[r]^{\pr_{j}} &\overline \Sh_{1,n-1}.
}
\]
 Assume that the restriction of $g_X$ to each $Y_{j,z}\times_{\overline \Sh_{1,n-1}}X$ with $z\in \Sh_{0,n}(\Fpb)$ is a regular closed immersion as well.
 Then $\pr^!_{j}([X])\in A_0(\overline Y_{j})$ can be described as follows.
Put $N_{Y_{j,z}}(\overline \Sh_{1,n-1}):=\pr_j^*(T_{\overline \Sh_{1,n-1}})/T_{Y_{j,z}}$, and  we define $N_{Y_{j,z}\times_{\overline \Sh_{1,n-1}} X}(X)$ in a similar way.
We define the \emph{excess vector bundle} as
 \[
 \calE(Y_{j,z}, X):= g_j^*N_{Y_{j,z}}(\overline \Sh_{1,n-1})/ N_{Y_{j,z}\times_{\overline \Sh_{1,n-1}}X}(X).
 \]
 This is a vector bundle on $Y_{j,z}\times_{\overline \Sh_{1,n-1}}X$. Let $r$ be its rank function, which is equal to  the dimension of $\overline Y_{j}\times_{\overline \Sh_{1,n-1}}X$ on each of its connected component.
  Then the excess intersection formula \cite[6.3]{fulton} shows that
 \begin{equation}\label{E:excess-intersection}
 \pr_j^{!}([X])=\sum_{z\in \Sh_{0,n}(\Fpb)}
 \int_{Y_{j,z}\times_{\overline \Sh_{1,n-1}}X} c_{r}\big(\calE(Y_{j,z},X)\big),
 \end{equation}
where $c_{r}(\calE(Y_{j,z}, X))$ is the top Chern class of $\calE(Y_{j,z},X)$ over $Y_{j,z}\times_{\overline \Sh_{1,n-1}}X$. 
The integration should be understood as the sum over all connected components of $Y_{j,z}\times_{\overline \Sh_{1,n-1}}X$ of the degrees of $c_r(\calE(Y_{j,z}, X))$.

\begin{proposition}\label{P:intersection-cycles}
Let $i,j$ be integers with  $1\leq i\leq j\leq n$ and $z,z'\in \Sh_{0,n}(\Fpb)$.
\begin{enumerate}
\item
The subvarieties $Y_{i,z}$ and $Y_{j,z'}$ of $\overline \Sh_{1,n-1}$ have non-empty intersection  if and only if there exists an integer $\delta$ with $0\leq \delta\leq \min\{ n-j,i-1\}$ such that $z'\in R_\gothp^{(j-i+\delta, n-\delta)}S_\gothp^{-1}(z)$, or equivalently $z\in R_{\gothp}^{(\delta, n+i-j-\delta)}S_\gothp^{-1}(z')$, where $R_\gothp^{(a,b)}$ and $S_\gothp$ are the Hecke operators defined in Subsection~\ref{S:hecke-operator}.

\item If the condition in \emph{(1)} is satisfied for some $\delta$, then $Y_{i,z}\times_{\overline \Sh_{1,n-1}}Y_{j,z'}$ is isomorphic  to the variety $\overline Z^{\langle n+i-j-2\delta\rangle }_{i-\delta}$ defined in Subsection~\ref{S:grassmannian}.
Moreover,  the excess vector bundles $\calE(Y_{i,z},Y_{j,z'})$ and $\calE(Y_{j,z'},Y_{i,z})$ are both isomorphic to the vector bundle \eqref{E:excess-bundle} on $\overline Z_{i-\delta}^{\langle n+i-j-2\delta\rangle}$.

\end{enumerate}

\end{proposition}

\begin{proof}
Let $(\cB_{z},\lambda_{z},\eta_{z})$ and  $(\cB_{z'},\lambda_{z'},\eta_{z'})$ be  the universal polarized abelian varieties on $\overline \Sh_{0,n}$ at $z$ and $z'$, respectively.
Then $Y_{i,z}\times_{\overline \Sh_{1,n-1}}Y_{j,z'}$ is the moduli space of  tuples $(A,\lambda, \eta, \phi, \phi')$  where $\phi:\cB_{z}\ra A$ and $\phi': \cB_{z'}\ra A$ are isogenies
 such that $(A,\lambda, \eta, \cB_{z},\lambda_{z},\eta_{z},\phi)$ and $(A,\lambda,\eta,\cB_{z'},\eta_{z'},\phi')$ are  points of $Y_{i,z}$ and $Y_{j,z'}$ respectively.

  Assume first that $Y_{i,z}\times_{\overline \Sh_{1,n-1}}Y_{j,z'}$ is non-empty, and let $(A,\lambda, \eta,\phi,\phi')$ be an $\Fpb$-valued point of it.
  Denote by $\tilde\omega^{\circ}_{A^\vee,k}\subseteq \tcD(A)^{\circ}_{k}$ for $k=1,2$ the inverse image of $\omega^{\circ}_{A^\vee/\overline \FF_p,k}\subseteq H^{\dR}_1(A/\Fpb)^{\circ}_k\cong \tilde\calD(A)^{\circ}_{k}/ p \tilde \calD(A)^\circ_k$.
 We identify $\tcD(\cB_{z})^{\circ}_{k}$ and $\tcD(\cB_{z'})^{\circ}_{k}$ with their images in $\tcD(A)^{\circ}_{k}$  via $\phi_{z,*,k}$ and $\phi_{z',*,k}$. Then we have a diagram of inclusions of $W(\Fpb)$-modules:
 \begin{equation}\label{Diag:intersection-ab}
 \xymatrix{
 &&&\tcD(\cB_{z})^{\circ}_{1}\ar@{^(->}[rd]^{\delta}\\
 p\tcD(A)^{\circ}_{1}\ar@{^(->}[r]^1 &\tilde\omega^\circ_{A^\vee, 1}\ar@{^(->}[r]^-{n-j-\delta}
 &\tcD(\cB_{z})^{\circ}_1\cap \tcD(\cB_{z'})^{\circ}_{1}\ar@{^(->}[ru]^{j-i+\delta} \ar@{^(->}[rd]^{\delta} &&
 \tcD(\cB_{z})^{\circ}_1+\tcD(\cB_{z'})^{\circ}_1 \ar@{^(->}[r]^-{i-\delta-1} & \tcD(A)^{\circ}_{1}.\\
 &&&\tcD(\cB_{z'})^{\circ}_{1}\ar@{^(->}[ur]^{j-i+\delta}
}
 \end{equation}
Here the numbers on the arrows indicate the $\overline \FF_p$-dimensions of the cokernel of the corresponding inclusions, which we shall compute below.
 By the definition of $Y_{i}$ and $Y_j$, we have
 \[
 \dim_{\Fpb}\big(\tcD(A)^{\circ}_1/\tcD(\cB_{z})^{\circ}_{1}\big)=\dim_{\Fpb} \coker(\phi_{*,1})=i-1,
 \]
  and similarly, $\dim_{\Fpb}\big(\tcD(\cB_{z'})^{\circ}_1/\tilde\omega_{A^\vee,1}^{\circ}\big)=n-j$. Therefore, if we put
  \[
  \delta=\dim_{\Fpb}\big(\tcD(\cB_{z})^{\circ}_1+\tcD(\cB_{z'})^{\circ}_1\big)/\tcD(\cB_z)^{\circ}_1=\dim_{\Fpb}\tcD(\cB_{z'})^{\circ}_1/\big(\tcD(\cB_{z})^{\circ}_1\cap \tcD(\cB_{z'})^{\circ}_1\big),
  \]
   we have $0\leq \delta\leq \min\{i-1,n-j\}$.
    Moreover, the quasi-isogeny $\phi_{z,z'}=\phi^{-1}\circ\phi':\cB_{z'}\ra \cB_{z}$ makes $\cB_{z'}$ an element of $\Isog(z)$.
    We identify  $\LL_{z'}$ defined in \eqref{Equ:lattice} with a $\Z_{p^2}$-lattice of $\LL_{z}[1/p]$ via $\phi_{z',z,*,1}$. Then
    \[
    \dim_{\F_{p^2}}(\LL_{z}\cap \LL_{z'})/p\LL_{z}=\dim_{\Fpb}(\tcD(\cB_{z})^{\circ}_1\cap \tcD(\cB_{z'})^{\circ}_1)/p\tcD(\cB_{z})^{\circ}_1=n+i-j-\delta.
    \]
  Take a $\Z_{p^2}$-basis  $(e_1,\dots, e_n)$ of $\LL_{z}$ such that the image of $(e_{j-i+\delta+1},\dots, e_{n})$ in $\LL_{z}/p\LL_z$ form a basis of $(\LL_{z}\cap\LL_{z'})/p\LL_{z}$ and such that $p^{-1}e_{n-\delta+1}, \dots, p^{-1}e_n$ form a basis of $(\LL_z + \LL_{z'}) / \LL_z$.  Then
  \begin{equation}\label{E:basis}
  (pe_{1},\dots, pe_{j-i+\delta},e_{j-i+\delta+1},\dots, e_{n-\delta},p^{-1}e_{n-\delta+1},\dots, p^{-1}e_{n})
  \end{equation}
is a basis of $\LL_{z'}$, that is $z'\in R_{\gothp}^{(j-i+\delta, n-\delta)}S_\gothp^{-1}(z)$ according to the convention of Subsection~\ref{S:hecke-operator}.

Conversely, assume that there exists $\delta$ with $1\leq \delta \leq \min\{i-1,n-j\}$ such that the point $z'\in R_\gothp^{(j-i+\delta, n-\delta)}S_\gothp^{-1}(z)$.
We have to prove statement (2),  then the non-emptiness of $Y_{i,z}\times_{\overline \Sh_{1,n-1}}Y_{j,z'}$ will follow automatically.
Let $\phi_{z',z}: \cB_{z'}\ra \cB_{z}$ be the unique quasi-isogeny which identifies $\LL_{z'}$ with a $\Z_{p^2}$-lattice of $\LL_{z}[1/p]$.
  By the definition of $R_\gothp^{(j-i+\delta, n-\delta)}S_\gothp^{-1}$,   there exists a basis $e_1, \dots, e_n$ of  $\LL_{z}$ such that \eqref{E:basis} is a basis of $\LL_{z'}$.
One checks easily that $p(\LL_{z}+\LL_{z'})\subseteq \LL_{z}\cap\LL_{z'}$.
  We put
  $$
  M_k= \big( \tcD(\cB_z)^{\circ}_k\cap \tcD(\cB_{z'})^{\circ}_k\big) \big/p\big(\tcD(\cB_z)^{\circ}_k+\tcD(\cB_{z'})^{\circ}_k\big)
  $$
  for $k=1,2$. Then one has
\[
\dim_{\Fpb}(M_k)=\dim_{\F_{p^2}}(\LL_{z}\cap\LL_{z'})\big/p(\LL_{z}+\LL_{z'})=n+i-j-2\delta.
\]
The Frobenius and Verschiebung  on $\tcD(\cB_z)$ induce two bijective Frobenius semi-linear maps    $F: M_1\ra M_2$ and $V^{-1}: M_2\ra M_1$. We denote their linearizations by the same notation if no confusions arise.
Let $Z_{\delta}(M_{\bullet})$ be the moduli space which attaches to each locally noetherian $\Fpb$-scheme $S$ the set of isomorphism classes of pairs  $(L_1,L_2)$, where
 $L_1\subseteq M_1\otimes_{\Fpb}\cO_S$ and $L_2\subseteq M_2\otimes_{\Fpb}\cO_S$  are subbundles of rank $i-\delta$ and $i-1-\delta$ respectively such that
 \[
L_2\subseteq  F(L_1^{(p)}),\quad V^{-1}(L_2^{(p)})\subseteq L_1.
 \]
 Note that there exists a basis $(\varepsilon_{k,1}, \dots,\varepsilon_{k, n+i-j-2\delta})$ of $M_k$ for $k=1,2$ under which the matrices of
$F$ and $V^{-1}$ are both identity.
 Indeed, by solving a system of equations of Artin--Schreier type, one can take  a basis $(\varepsilon_{1,\ell})_{1\leq\ell\leq n+i-j-2\delta}$ for $M_1$ such that
\[
 V^{-1}(F(\varepsilon_{1,\ell}))=\varepsilon_{1,\ell}\quad \text{for all } 1\leq \ell\leq n+i-j-2\delta.
\]
We put $\varepsilon_{2,\ell}=F(\varepsilon_{1,\ell})$. Using these bases to identify both $M_1$ and $M_2$ with $\Fpb^{n+i-j-2\delta}$, it is clear that $Z_{\delta}(M_{\bullet})$ is isomorphic to the variety $\overline Z_{i-\delta}^{\langle n+i-j-2\delta\rangle}$ considered in Subsection~\ref{S:grassmannian}.

 We have to establish an isomorphism between $Z_{\delta}(M_{\bullet})$ and $Y_{i,z}\times_{\overline \Sh_{1,n-1}}Y_{j,z'}$.
 Let  $(L_1,L_2)$ be an $S$-point of $Z_{\delta}(M_{\bullet})$. Note that there is a natural surjection
 $$
 \big((\tcD(\cB_{z})^{\circ}_k\cap \tcD(\cB_{z'})^{\circ}_k)/p\tcD(\cB_{z})^{\circ}_k\big)\otimes_{\Fpb}\cO_{S}\ra M_k\otimes_{\Fpb}\cO_S.
 $$
We define $H_{z,k}$ for $k=1,2$ to be the inverse image of $L_k$ under this surjection.
Then $H_{z,k}$ can be naturally viewed as a subbundle of $\cD(\cB_{z})^{\circ}_k\otimes_{\Fpb}\cO_S$ of rank $i+1-k$, and we have  $H_{z,2}\subseteq F(H_{z,1}^{(p)})$ and $V^{-1}(H_{z,2}^{(p)})\subseteq H_{z,1}$ since the pair $(L_1,L_2)$ verifies similar properties.
 Therefore, $(L_1,L_2)\mapsto (\cB_{z,S},\lambda_{z,S},\eta_{z,S}, H_{z,1},H_{z,2})$ gives rise to a well-defined map $\varphi_{i,z}':Z_{\delta}(M_{\bullet})\ra Y'_{i,z}$, where $(\cB_{z,S},\lambda_{z,S},\eta_{z,S})$ is the base change of $(\cB_{z},\lambda_{z},\eta_{z})$ to $S$.
  Similarly, we have a morphism $\varphi'_{j,z'}: Z_{\delta}(M_{\bullet})\ra Y'_{j,z'}$ defined by $(L_1,L_2)\mapsto (\cB_{z',S},\lambda_{z',S},\eta_{z',S}, H_{z',1}, H_{z',2})$, where $H_{z',k}$ is the inverse image of $L_k$ under the natural surjection:
  \[
  \big(\big(\tcD(\cB_{z})^{\circ}_k\cap \tcD(\cB_{z'})^{\circ}_k\big)/p\tcD(\cB_{z'})^{\circ}_k\big)\otimes_{\Fpb}\cO_{S}\ra M_k\otimes_{\Fpb}\cO_S.
  \]
  By  Proposition~\ref{P:isom-Yj}, we get two morphisms
 \[
 \varphi_{i,z}\colon  Z_{\delta}(M_{\bullet})\ra Y_{i,z}, \quad\quad  \varphi_{j,z'}\colon Z_{\delta}(M_{\bullet})\ra Y_{j,z'}.
 \]
We claim that $\pr_{i}\circ\varphi_{i,z}=\pr_{j}\circ\varphi_{j,z}$, so that $(\varphi_{i,z},\varphi_{j,z'})$ defines a map
  \[
  \varphi\colon Z_{\delta}(M_{\bullet})\longrightarrow Y_{i,z}\times_{\overline \Sh_{1,n-1}} Y_{j,z'}.
  \]

Since  $Y_{i,z}\times_{\overline \Sh_{1,n-1}} Y_{j,z'}$ is separated, the 
locus where $\pr_{i}\circ\varphi_{i,z}$ coincides with $\pr_{j}\circ\varphi_{j,z}$ is  a closed subscheme of $Z_{\delta}(M_{\bullet})$. As $Z_{\delta}(M_\bullet)$ is reduced,  
it is enough to show $\pr_i(\varphi_{i,z}(x)) = \pr_j(\varphi_{j,z}(x))$ for each closed geometric point $x = (L_1, L_2) \in Z_\delta(M_\bullet)(\overline \FF_p)$. Let $(A,\lambda,\eta,\cB_{z},\lambda_{z},\eta_{z},\phi)$ and $(A',\lambda',\eta',\cB_{z'},\lambda_{z'},\eta'_{z'}, \phi')$ be respectively the image of $(L_1,L_2)$ under $\varphi_{i,z}$ and $\varphi_{j,z'}$.
To prove the claim, we have to show that there is an isomorphism $(A,\lambda,\eta)\cong (A',\lambda',\eta')$ as objects of $\overline \Sh_{1,n-1}$.
We  identify $\tcD(\cB_{z'})$, $\tcD(A)$, $\tcD(A')$ with $W(\Fpb)$-lattices of $\tcD(\cB_{z})[1/p]$ via the quasi-isogenies $\phi_{z',z}: \cB_{z'}\ra \cB_{z}$, $\phi^{-1}:A\ra \cB_{z}$ and $\phi_{z',z}^{-1}\circ\phi': A'\ra \cB_{z}$.
Then by the construction of $A$ (cf. the proof of Proposition~\ref{P:isom-Yj}), $\tcD(A)^{\circ}_1$ and $\tilde \omega^{\circ}_{A^\vee,1}$ fit into the diagram \eqref{Diag:intersection-ab} such that there is a canonical isomorphism
\begin{equation}\label{E:relation-L_1}
L_1\cong \tilde\omega^{\circ}_{A^\vee,1}\big/p\big(\tcD(\cB_{z})^{\circ}_1+\tcD(\cB_{z'})^{\circ}_{1}\big)\subseteq \big(\tcD(\cB_{z})_1^{\circ}\cap \tcD(\cB_{z'})_1^{\circ}\big)\big/p\big(\tcD(\cB_{z})^{\circ}_1+\tcD(\cB_{z'})^{\circ}_{1}\big)=M_1.
\end{equation}
Similarly, we have
\begin{equation}\label{E:relation-L_2}
L_2\cong p\tilde\omega_{A^\vee,2}^{\circ}/p\big(\tcD(\cB_{z})^{\circ}_2+\tcD(\cB_{z'})^{\circ}_{2}\big)\subseteq \big(\tcD(\cB_{z})_2^{\circ}\cap \tcD(\cB_{z'})_2^{\circ}\big)\big/p\big(\tcD(\cB_{z})^{\circ}_2+\tcD(\cB_{z'})^{\circ}_{2}\big)=M_2.
\end{equation}
 It is easy to see that such relations determine $\tcD(A)$ uniquely from $(L_1,L_2)$. But the same argument shows that the same relations are satisfied with $A$ replaced by $A'$. Hence, we see that the quasi-isogeny $f$ induces an isomorphism between the Dieudonn\'e modules of $A$ and $A'$. As $f$ is a $p$-quasi-isogeny, this implies immediately that $f$ is an isomorphism of abelian varieties, proving the claim.

 It remains  to prove that $\varphi\colon Z_{\delta}(M_{\bullet})\xra{\sim}Y_{i,z}\times_{\overline \Sh_{1,n-1}}Y_{j,z'}$ is an isomorphism.
 It suffices to show that $\varphi$ induces  bijections on closed points and tangents spaces. The argument is similar to the proof of Proposition~\ref{P:isom-Yj}. Indeed, given a closed point $x=(A,\lambda, \eta, \phi,\phi')$ of $Y_{i,z}\times_{\overline \Sh_{1,n-1}}Y_{j,z'}$,  one can construct a unique point $y=(L_1,L_2)$ of $Z_{\delta}(M_{\bullet})$ with $\varphi(y)=x$ by the relations \eqref{E:relation-L_1} and \eqref{E:relation-L_2}. It follows immediately that $\varphi$ induces a bijection on closed points.
 Let $x$ and $y$ be as above.
  By the same argument as in Proposition~\ref{P:Yj-prime}, the tangent space of $Z_{\delta}(M_{\bullet})$ at $y$ is given by
 \[
 T_{Z_{\delta}(M_{\bullet}),y}\cong \big(L_1/V^{-1}(L_2^{(p)})\big)^*\otimes \big(M_1/L_1\big)\oplus
 L_2^*\otimes F(L_1^{(p)})/L_2.
 \]
 On the other hand, using Grothendieck--Messing deformation theory, one sees easily that the tangent space of $Y_{i,z}\times_{\overline \Sh_{1,n-1}}Y_{j,z'}$ at $x$ is given by
 \begin{align*}
 T_{Y_{i,z}\times_{\overline \Sh_{1,n-1}}Y_{j,z'},x}&\cong\Hom_{\Fpb}\Big(\omega_{A^\vee,1}^{\circ}, \big(\tcD(\cB_z)^{\circ}_{1}\cap \tcD(\cB_{z'})^{\circ}_1\big)\big/\tilde\omega^{\circ}_{A^\vee,1}\Big)\\
 \quad &\oplus \Hom_{\Fpb}\Big(\tilde\omega_{A^\vee,2}^{\circ}\big/\big(\tcD(\cB_{z})^{\circ}_2+\tcD(\cB_{z'})^{\circ}_2\big), \tcD(A)^{\circ}_2/\tilde\omega^{\circ}_{A^\vee,2}\Big).
 \end{align*}
From  \eqref{E:relation-L_1} and \eqref{E:relation-L_2}, we see easily that
\begin{gather*}
\omega^{\circ}_{A^\vee,1}\cong L_1/V^{-1}(L_2^{(p)}),\quad  \tcD(\cB_{z'})^{\circ}_1)/\tilde\omega^{\circ}_{A^\vee,1}\cong M_1/L_1,\\
\tilde\omega_{A^\vee,2}^{\circ}\big/\big(\tcD(\cB_{z})^{\circ}_2+\tcD(\cB_{z'})^{\circ}_2\big)\cong L_2, \quad \tcD(A)^{\circ}_2/\tilde\omega^{\circ}_{A^\vee,2}\cong F(L_1^{(p)})/L_2.
\end{gather*}
It follows that $\varphi$ induces a bijection between $T_{Z_{\delta}(M_{\bullet}),y}$ and $ T_{Y_{i,z}\times_{\overline \Sh_{1,n-1}}Y_{j,z'},x}$. This finishes the proof of Proposition~\ref{P:intersection-cycles}.
\end{proof}

\subsection{Applications to cohomology}
Recall that we have a morphism $\JL_j$ \eqref{E:jacquet-langlands} for each $j =1, \dots, n$.
 We consider another map in the opposite direction:
\[
\nu_{j}\colon H_\et^{2(n-1)}(\overline \Sh_{1,n-1},\Ql(n-1))\xra{\pr_j^*} H_\et^{2(n-1)}(\overline Y_{j},\Ql)\xra{\sim} H_\et^0(\overline \Sh_{0,n},\Ql),
\]
where the second isomorphism is induced by the trace map $\mathrm{Tr}_{\pr'_j}: R^{2(n-1)}\pr'_{j,*}\Ql(n-1)\xra{\sim} \Ql$.
For $1\leq i,j\leq n$, we define
\[
m_{i,j}=\nu_{j}\circ \JL_{i}: H_\et^{0}(\overline \Sh_{0,n},\Ql)\xra{\JL_i}H_\et^{2(n-1)}\big(\overline \Sh_{1,n-1},\Ql(n-1)\big)\xra{\nu_{j}} H_\et^{0}(\overline \Sh_{0,n},\Ql).
\]
Putting all the morphisms $\JL_i$ and $\nu_{j}$ together, we get a sequence of  morphisms:
\begin{equation}\label{E:sequence-morphisms}
\bigoplus_{i=1}^{n} H_\et^0(\overline \Sh_{0,n},\Ql)\xra{\JL} H_\et^{2(n-1)}\big(\overline \Sh_{1,n-1},\Ql(n-1)\big)\xra{\nu=(\nu_1,\dots, \nu_n)} \bigoplus_{j=1}^n H_\et^0(\overline \Sh_{0,n},\Ql).
\end{equation}
We see that the composed morphism above is given by the  matrix $M=(m_{i,j})_{1\leq i,j\leq n}$, and we call it the \emph{intersection matrix} of cycles $Y_j$ on $\Sh_{1,n-1}$.  All these morphisms are equivariant under the natural action of the Hecke algebra $\scrH(K^p,\Ql)$. We will describe the intersection matrix in terms of the Hecke action of $\Qlb[K_{\gothp}\backslash \GL_n(E_{\gothp})/K_\gothp]$ on $H_\et^0(\overline \Sh_{0,n},\Ql)$.

  The group $H^0_\et(\overline \Sh_{0,n},\Ql)$ is the space of functions on $\Sh_{0,n}(\Fpb)$ with values in $\Ql$. For $z\in \Sh_{0,n}(\Fpb)$, let $e_{z}$ denote the characteristic function of $z$. Then the image of $z$ under $K_\gothp\gamma K_\gothp$ for $\gamma\in \GL_n(E_{\gothp})$ is
\[
[K_\gothp\gamma K_\gothp]_*(e_z)=\sum_{z'\in T_\gothp(\gamma)(z)} e_{z'},
\]
where $T_\gothp(\gamma)(z)$ means the set theoretic Hecke correspondence defined in Subsection~\ref{S:hecke-operator}.
In the sequel, we will use the same notation $T_\gothp(\gamma)$ to denote the  action of $[K_\gothp\gamma K_\gothp]$ on $H^0_\et(\overline \Sh_{0,n},\Ql)$. In particular, we have Hecke operators $T_\gothp^{(i)}$, $S_\gothp$, $R_\gothp^{(a,b)}$, \dots

\begin{proposition}\label{P:intersection-matrix}
For $1\leq i\leq j\leq n$, we have
\begin{align*}
m_{i,j}&=\sum_{\delta=0}^{\min\{i-1,n-j\}} N(n+i-j-2\delta,i-\delta) R_\gothp^{(j-i+\delta, n-\delta)} S_\gothp^{-1},\\
m_{j,i}&=\sum_{\delta=0}^{\min\{i-1,n-j\}} N(n+i-j-2\delta, i-\delta) R_\gothp^{(\delta, n+i-j-\delta)} S_\gothp^{-1},
\end{align*}
where $N(n+i-j-2\delta,i-\delta)$ are the fundamental intersection numbers defined by \eqref{D:fundamental-number}.
\end{proposition}
\begin{proof}
 We have a commutative diagram:
\begin{equation}\label{E:diag-cycles}
\xymatrix@C=40pt{
A_{n-1}(\overline Y_{i})\ar[r]^-{\pr_{i,*}}\ar[d]^{\cl}& A_{n-1}(\overline \Sh_{1,n-1})\ar[r]^-{\pr_{j}^!}\ar[d]^{\cl} &A_{0}(\overline Y_{j})\ar[d]^{\cl}\\
H_\et^0(\overline Y_{i},\Ql)\ar[r]^-{\mathrm{Gys}_{\pr_i}} &H_\et^{2(n-1)}\big(\overline \Sh_{1,n-1},\Ql(n-1)\big) \ar[r]^-{\pr_j^*} & H_\et^{2(n-1)}(\overline Y_{j},\Ql).
}
\end{equation}
Here, the vertical arrows are cycle class maps, and $\pr_j^!$ is the refined Gysin map defined in \eqref{E:refined-Gysin}.
 For $z\in \Sh_{0,n}(\Fpb)$, the image of $e_z$ under $m_{i,j}$ is given by
\begin{align*}
m_{i,j}(e_{z})&=\mathrm{Tr}_{\pr'_j}\pr^*_j\Gys_{\pr_i}\cl([Y_{i,z}])=\mathrm{Tr}_{\pr'_j}\big(\cl(\pr_j^!\pr_{i,*}[Y_{i,z}])\big)\\
&=\mathrm{Tr}_{\pr'_j}\Big(\sum_{z'\in \Sh_{0,n}(\Fpb)}\cl\Big( c_{r(z,z')}\big(\calE(Y_{j,z'},Y_{i,z})\big)\Big)\cdot \cl\big(Y_{j,z'}\times_{\overline \Sh_{1,n-1}}Y_{i,z}\big)\Big)\\
&=\sum_{z'\in \Sh_{0,n}(\Fpb)}\Big(\int_{Y_{j,z'}\times_{\Sh_{1,n-1}}Y_{i,z}}c_{r(z,z')}\big(\calE(Y_{j,z'},Y_{i,z})\big)\Big)e_{z'},
\end{align*}
where $r(z,z')$ is the rank of $\calE(Y_{j,z'},Y_{i,z})$, and we used \eqref{E:excess-intersection} in the second step.   Indeed, Proposition~\ref{P:intersection-cycles}(2) says that the schematic intersection $Y_{i,z}\times_{\overline \Sh_{1,n-1}}Y_{j,z'}$ is smooth, so  the closed immersion $Y_{i,z}\times_{\overline \Sh_{1,n-1}}Y_{j,z'}\hra Y_{j,z'}$ is a regular immersion and the assumptions for \eqref{E:excess-intersection} are thus satisfied here. 

By Proposition~\ref{P:intersection-cycles}(1),  $e_{z'}$ has a non-zero contribution to the summation above if and only if there exists an integer $\delta$ with $0\leq \delta\leq \min\{i-1,n-j\}$ such that $z'\in R_\gothp^{(j-i+\delta,n-\delta)}S_\gothp^{-1}(z)$. 
 In that case, Proposition~\ref{P:intersection-cycles}(2) implies that  the coefficient of $e_{z'}$ is nothing but the fundamental intersection number  $N(n+i-j-2\delta, i-\delta)$ defined in ~\eqref{D:fundamental-number}. 
The formula for $m_{i,j}$ now follows immediately. The formula for $m_{j,i}$ is proved in the same manner.
\end{proof}

If we express $m_{i,j}$ in terms of the elementary Hecke operators $T_\gothp^{(k)}$, we get the following.
\begin{theorem}\label{Th-Conj}
Put $d(n,k)=(2k-1)n-2k(k-1)-1$ for integers $1\leq k\leq n$. Then, for $1\leq i\leq j\leq n$, we have
\begin{align*}
m_{i,j}&= \sum_{\delta=0}^{\min\{i-1,n-j\}}(-1)^{n+1+i-j} (n+i-j-2\delta) p^{d(n+i-j-2\delta,i-\delta)} T_\gothp^{(j-i+\delta)}T_\gothp^{(n-\delta)}S_\gothp^{-1},\\
m_{j,i}&=\sum_{\delta=0}^{\min\{i-1,n-j\}}(-1)^{n+1+i-j} (n+i-j-2\delta) p^{d(n+i-j-2\delta,i-\delta)} T_\gothp^{(\delta)} T_\gothp^{(n+i-j-\delta)}S_\gothp^{-1}.
\end{align*}
\end{theorem}
\begin{proof}
We prove only the statement for $m_{i,j}$,  and that for $m_{j,i}$ is similar. By Proposition~\ref{P:multiplication-Hecke} in Appendix A,  the right hand side of the first formula above is
\begin{align*}
& \sum_{\delta=0}^{\min\{i-1,n-j\}} (-1)^{n+1-i-j} (n+i-j-2\delta) p^{d(n+i-j-2\delta,i-\delta)}\Big(\sum_{k=0}^{\delta} \tbinom{n+i-j-2\delta +2k}{k}_{p^2} R^{(j-i+\delta-k,n-\delta +k)}_\gothp S_\gothp^{-1}\Big)\\
&=\sum_{r=0}^{\min\{i-1,n-j\}}\big( \star\big) R_\gothp^{(j-i+r, n-r)}S_\gothp^{-1}.
\end{align*}
Here, we have put $r=\delta-k$, and the expression $\star$ in the parentheses is
\begin{align*}
\star&=\sum_{k=0}^{\min\{i-1-r,n-j-r\}}(-1)^{n+1+i-j}(n+i-j-2r-2k) p^{d(n+i-j-2r-2k,i-r-k)}\binom{n+i-j-2r}{k}_{p^2}\\
&=N(n+i-j-2r,i-r).
\end{align*}
Here, the last equality is Theorem~\ref{T:intersection}. 
The statement for $m_{i,j}$ now follows from Proposition~\ref{P:intersection-matrix}.
\end{proof}

\begin{example}
We  write down explicitly the intersection matrices when $n$ is small.

(1) Consider first the case $n=2$. This case is essentially the same as the Hilbert quadratic case studied in \cite{tian-xiao2}, and the intersection matrix can be written:
\[
M=\begin{pmatrix}
-2p & T_\gothp^{(1)}\\
T_\gothp^{(1)} S_\gothp^{-1} &-2p
\end{pmatrix}.
\]

(2) When $n=3$,  Theorem~\ref{Th-Conj} gives
\[
M=\begin{pmatrix}
3p^2 & -2p T_\gothp^{(1)} &T_\gothp^{(2)}\\
-2p T_\gothp^{(2)}S_\gothp^{-1} & 3p^4+T_\gothp^{(1)} T_\gothp^{(2)}S_\gothp^{-1} & -2p T_\gothp^{(1)}\\
T_\gothp^{(1)} S_\gothp^{-1} & -2p T_\gothp^{(2)}S_\gothp^{-1} & 3p^2
\end{pmatrix}.
\]

(3) The intersection matrix for $n=4$ can be written:
\[
M=\begin{pmatrix}
-4p^3 & 3p^2T_\gothp^{(1)} & -2p T_\gothp^{(2)} & T_\gothp^{(3)}\\
3p^2 T_\gothp^{(3)}S_\gothp^{-1} & -4p^7-2pT_\gothp^{(1)}T_\gothp^{(3)}S_\gothp^{-1} & 3p^4 T_\gothp^{(1)} +T_\gothp^{(2)}T_\gothp^{(3)}S_\gothp^{-1} &-2pT_\gothp^{(2)}\\
-2p T_\gothp^{(2)} S_\gothp^{-1} &3p^4T_\gothp^{(3)}S_\gothp^{-1}+T_\gothp^{(1)}T_\gothp^{(2)}S_\gothp^{-1} & -4p^7-2pT_\gothp^{(1)}T_\gothp^{(3)}S_\gothp^{-1} & 3p^2T_\gothp^{(1)}\\
T_\gothp^{(1)}S_\gothp^{-1} &-2p T_\gothp^{(2)}S_\gothp^{-1} & 3p^2T_\gothp^{(3)}S_\gothp^{-1} & -4p^3
\end{pmatrix}.
\]

\end{example}

\subsection{Proof of Theorem~\ref{T:main-theorem}(1)}
 Let $\pi\in \scrA_K$ as in the statement of  Theorem~\ref{T:main-theorem}(1). Consider the $(\pi^p)^{K^p}$-isotypic direct factor  of  the $\scrH(K^p,\Qlb)$-equivariant sequence \eqref{E:sequence-morphisms}:
  \begin{equation}
  \label{E:pi component of JL nu}
 \bigoplus_{i=1}^n H_\et^{0}(\overline \Sh_{0,n},\Ql)_{\pi^p}\xra{\JL_{\pi}} H_\et^{2(n-1)}\big(\overline \Sh_{1,n-1},\Ql(n-1)\big)_{\pi^p}\xra{\nu_{\pi}}  \bigoplus_{j=1}^n
 H_\et^{0}(\overline \Sh_{0,n},\Ql)_{\pi^p}.
\end{equation}
In particular, when $i=j=1$, $\nu_1\circ \JL_1$ is given by multiplication by $-np^{n-1}$. So the $\pi^p$-isotypic component of \eqref{E:pi component of JL nu} is nonzero.
 This implies that $\pi^p$ appears in $H^{2(n-1)}_\et(\overline \Sh_{1,n-1}, \overline \QQ_\ell(n-1))$, i.e., there exist
 admissible irreducible representations $\pi'_p$ of $G_{1,n-1}(\QQ_p)$ and  $\pi_{\infty}'$ of $G_{1,n-1}(\RR)$, which is  cohomological in degree $n-1$, such that $\pi^p \otimes \pi'_p\otimes \pi'_{\infty}$ is a cuspidal automorphic representation $\pi'\otimes \pi'_{\infty}$  of $G_{1,n-1}(\AAA_\QQ)$. 
 By Lemma~\ref{L:automorphic-prime-to-p},  $\pi'\simeq \pi$ satisfies  Hypothesis~\ref{H:automorphic assumption}(2) for $a_{\bullet}=(1,n-1)$. 
Thus, taking the $\pi^p$-isotypic component of \eqref{E:pi component of JL nu} is the same as taking its $\pi$-isotypic component.
 From now on, we use subscript $\pi$ in places of subscript $\pi^p$.

If $a_{\gothp}^{(i)}$ denotes the eigenvalues of $T_{\gothp}^{(i)}$ on $\pi_\gothp^{K_\gothp}$ for each $1\leq i\leq n$, then  $T_\gothp^{(i)}$ acts as the scalar $a_\gothp^{(i)}$ on all the terms in \eqref{E:pi component of JL nu}. Therefore, $\nu_{\pi}\circ\JL_{\pi}$ is given by the matrix $M_{\pi}$, which is obtained by replacing  $T_\gothp^{(i)}$ by $a_\gothp^{(i)}$ in each entry of $M$.
  By definition,  the $\alpha_{\pi_{\gothp},i}$  are the roots of the Hecke polynomial \eqref{E:hecke-polynomial}:
  \[
  X^n + \sum_{i=1}^n(-1)^{i} p^{i(i-1)}a_\gothp^{(i)} X^{n-i}.
  \]
  Then  Theorem~\ref{T:main-theorem}(1) follows easily from the following.

  \begin{lemma}\label{L:Hecke-determiant}

We have
  \[
  \det(M_{\pi})=\pm p^{\frac{n(n^2-1)}{3}}\frac{\prod_{i<j}(\alpha_{\pi_{\gothp},i}-\alpha_{\pi_\gothp, j})^2}{(\prod_{i=1}^n\alpha_{\pi_\gothp,i})^{n-1}}.
  \]
  Here, $\pm$ means that the formula holds up to sign.
  In particular,  $\nu_{\pi}\circ\JL_{\pi}$ is an isomorphism if the $\alpha_{\pi_\gothp,i}$ are distinct.
  \end{lemma}

  \begin{proof}
Put $\beta_i=\alpha_{\pi_\gothp,i}/p^{n-1}$ for $1\leq i\leq n$. For $i=1,\dots, n$, let $s_i$ be the  $i$-th elementary symmetric polynomial in  $\beta_1, \dots, \beta_n$.  Then we have  $a_\gothp^{(i)}=p^{i(n-i)}s_i$. It follows from Theorem~\ref{Th-Conj} that the $(i,j)$-entry of $M_{\pi}$ with $1\leq i\leq j\leq n$ is given by
\[
m_{i,j}(\pi)=s_{n}^{-1}\sum_{\delta=0}^{\min\{i-1,n-j\}}(-1)^{n+1+i-j} (n+i-j-2\delta) p^{d(n+i-j-2\delta, i-\delta)+ (j-i+\delta)(n+i-j-\delta)+\delta(n-\delta)} s_{j-i+\delta} s_{n-\delta} .
\]
A direct computation shows that the exponent index on $p$ in each term above is independent of $\delta$, and is equal to $e(i,j):=(n+1)(i+j-1)-(i^2+j^2)$. The same holds when $i>j$. In summary, we get $m_{i,j}(\pi)=s_n^{-1}p^{e(i,j)}m_{i,j}'(\pi)$ with
\[
m'_{i,j}(\pi)=\begin{cases}
\sum_{\delta=0}^{\min\{i-1,n-j\}}(-1)^{n+1+i-j} (n+i-j-2\delta)  s_{j-i+\delta}\; s_{n-\delta}, &\text{if }i\leq j,\\
\sum_{\delta=0}^{\min\{j-1,n-i\}}(-1)^{n+1+j-i} (n+j-i-2\delta)  s_{\delta}\; s_{n+j-i-\delta}, &\text{if }i>j.
\end{cases}
\]
For any $n$-permutation $\sigma$, we have
\[\sum_{i=1}^{n}e(i,\sigma(i))=\frac{n(n^2-1)}{3}.\]
Thus we get  $\det(M_{\pi})=p^{\frac{n(n^2-1)}{3}}s_n^{-n}\det(m'_{i,j}(\pi))$.  The rest of the computation is purely combinatorial, which is the case $q=-1$ of Theorem~\ref{T:determinant} in Appendix~\ref{A:determinant}.
  \end{proof}
  
\begin{remark}
We point out that the determinant of the intersection matrix computed by Theorem~\ref{T:determinant} holds with  an auxiliary variable $q$. The similar phenomenon also appeared in the case of Hilbert modular varieties \cite{tian-xiao2}, where the computation was related to the combinatorial model of periodic semi-meanders.
These motivate us to ask, out of curiosity,  whether there might be some quantum version of the construction of cycles, or even Conjecture~\ref{Conj:main}, possibly for the geometric Langlands setup.

\end{remark}

\subsection{Proof of  Theorem~\ref{T:main-theorem}(2)} Given  Theorem~\ref{T:main-theorem}(1), it suffices to prove that
\begin{equation}\label{E:inequality-finite-part}
n \dim H_\et^0(\overline \Sh_{0,n},\Ql)_{\pi}\geq \dim H_\et^{2(n-1)}\big(\overline \Sh_{1,n-1},\Ql(n-1)\big)_{\pi}^{\mathrm{fin}}.
\end{equation}
Actually, by \eqref{E:decomposition-cohomology} and \eqref{E:Galois-Sh},
we have
 \[
H_\et^0(\overline \Sh_{0,n},\Ql)_{\pi}=\pi^K\otimes R_{(0,n),\ell}(\pi), \quad H_\et^{2(n-1)}(\overline \Sh_{1,n-1},\Ql)_{\pi}=\pi^K\otimes R_{(1,n-1),\ell}(\pi).
\]
Write $\pi_p=\pi_{p,0}\otimes \pi_{\gothp}$ as a representation of $G(\Q_p)\simeq \Q_p^{\times}\times \GL_n(E_{\gothp})$.  Let $\chi_{\pi_{p,0}}: \Gal(\Fpb/\F_{p^2})\ra \Qlb^{\times}$ denote the character sending $\Frob_{p^2}$ to $\pi_{p,0}(p^2)$, and let $\rho_{\pi_{\gothp}}$ be as in \eqref{E:defn-rho-pi}.
According to \eqref{E:Galois-Sh}, up to  semi-simplification, we have
\begin{align}
\label{E:R0n}
\big[R_{(0,n),\ell}(\pi)\big]&= \#\ker^1(\Q,G_{0,n})m_{0,n}(\pi)\big[\wedge^{n}\rho_{\pi_{\gothp}}\otimes\chi_{\pi_{p,0}}^{-1}\otimes \Qlb(\tfrac{n(n-1)}{2})\big],\\
\label{E:R1n-1}
\big[R_{(1,n-1),\ell}(\pi)\big]&= \#\ker^1(\Q,G_{1,n-1})m_{0,n}(\pi)\big[ \rho_{\pi_\gothp}\otimes\wedge^{n-1}\rho_{\pi_\gothp}\otimes \chi_{\pi_{p,0}}^{-1}\otimes \Qlb(\tfrac{(n-1)(n-2)}{2})\big].
\end{align}
Note that
\[
\dim \Big(\rho_{\pi_\gothp}\otimes\wedge^{n-1}\rho_{\pi_\gothp}\otimes \chi_{\pi_{p,0}}^{-1}\otimes \Qlb\big(\tfrac{(n-1)(n-2)}{2}\big)\Big)^{\mathrm{fin}}=\sum_{\zeta} \dim \big(\rho_{\pi_\gothp}\otimes \wedge^{n-1}\rho_{\pi_\gothp}\big)^{\Frob_{p^2}=p^{n(n-1)}\zeta},
   \]
where the superscript ``$\mathrm{fin}$'' means taking the subspace on which $\Gal(\Fpb/\F_{p^2})$ acts through a finite quotient, and  $\zeta$ runs through all roots of unity.
If $\alpha_{\pi_{\gothp}, i}/\alpha_{\pi_{\gothp},j}$ is not a root of unity for any pair $i\neq j$, the right hand side above is equal to  the sum of the multiplicities of $\prod_{i=1}^n\alpha_{\pi_\gothp, i}=p^{n(n-1)}\zeta$ as eigenvalues of $(\rho_{\pi_\gothp}\otimes\wedge^{n-1}\rho_{\pi_\gothp})(\Frob_{p^2})$, which is $n$. 
Therefore, under these conditions on the $\alpha_{\pi_\gothp, i}$, we have by \eqref{E:R1n-1}
\[
\dim R_{(1,n-1),\ell}(\pi)^{\mathrm{fin}}\leq n \cdot\#\ker^1(\Q,G_{1,n-1})\cdot m_{1,n-1}(\pi),
\]
and the equality holds if $\Frob_{p^2}$ is semi-simple on $R_{(1,n-1),\ell}(\pi)$. 
On the other hand, we have from \eqref{E:R0n}
   \[\dim R_{(0,n),\ell}(\pi)=\#\ker^1(\Q,G_{0,n})\cdot m_{0,n}(\pi).\]
By a result of White \cite[Theorem E]{white}, the multiplicity $m_{a_{\bullet}}(\pi)$  above is equal to $1$ for $a_{\bullet}=(1,n-1)$  and   $a_{\bullet}=(0,n)$. 
Now the inequality \eqref{E:inequality-finite-part} follows immediately from this and the fact that $\#\ker^1(\Q,G_{1,n-1})=\#\ker^1(\Q,G_{0,n})$. This finishes the proof of Theorem~\ref{T:main-theorem}(2). \hfill $\Box$

\section{Construction of cycles in the case of $G(U(r,s) \times U(s,r))$}
\label{Sec:GU rs}
We keep the notation of Subsection~\ref{S:notation-real-quadratic}.
In this section,  we will give the construction of certain cycles on Shimura varieties for $G(U(r,s) \times U(s,r))$. We always assume that $s \geq r$.

\subsection{Description of the cycles in terms of Dieudonn\'e modules}
\label{S:description of cycles}
Let $\delta$ be a non-negative integer with $\delta \leq r$. We consider the case of Conjecture~\ref{Conj:main} when $n=r+s$, $a_1 = r$, $a_2=s$, $b_1 = r-\delta$, and $b_2 = s+\delta$.
The representation $r_{a_\bullet}$ of $\GL_n$ involved is
\[
r_{a_\bullet} = \wedge^r \mathrm{Std} \otimes \wedge^s \mathrm{Std}.
\]
The weight $\lambda$ of Conjecture~\ref{Conj:main} is
\[
\lambda=
\big( \underbrace{2, \dots, 2}_{r-\delta}, \underbrace{1, \dots, 1}_{s-r+2\delta}, \underbrace{0, \dots, 0}_{r-\delta} \big).
\]
By elementary calculation of representations of $\GL_n$, the multiplicity of $\lambda$ in $r_{a_\bullet}$ is $m_\lambda(a_\bullet) = \binom{s-r+2\delta}{\delta}$. Then
Conjecture~\ref{Conj:main} thus predicts the existence of $\binom{s-r+2\delta}{\delta}$ cycles $Y_j$ on $\Sh_{r,s}$, each of dimension \[
\tfrac{1}{2}\big(\dim \Sh_{r,s} + \dim \Sh_{r-\delta,s+\delta} \big) = \tfrac{1}{2}\big(
2rs + 2(r-\delta)(s+\delta)\big) = 2rs - (s - r)\delta - \delta^2,
\]
and each admits a rational map to $\Sh_{r - \delta, s + \delta}$.  The principal goal of this section is to construct these cycles, at least conjecturally.  We start with the description in terms of the Dieudonn\'e modules at closed points.

Consider the interval $[r-\delta, s+\delta]$; it contains $s-r+2\delta$ unit segments with integer endpoints.
We will parametrize the cycles 
on the Shimura variety by the subsets of these $s-r+2\delta$ unit segments of cardinality $\delta$. There are exactly $\binom{s-r+2\delta}\delta$ such subsets.
Let $\bfj$ be one of them.
  Then we can write the union of all the segments in $\bfj$ as
\begin{equation}
\label{E:break j into intervals}
[j_{1, 1}, j_{1, 2}] \cup [j_{2,1}, j_{2 ,2}] \cup \cdots \cup [j_{\epsilon,1}, j_{\epsilon,2}] 
\end{equation}
such that all $j_{\alpha,i}$ are integers,
\[
r-\delta \leq j_{1,1} < j_{1,2}< j_{2,1}<j_{2,2} < \cdots < j_{\epsilon,1}<j_{\epsilon,2} \leq s+\delta,
\]
and we have $\sum_{\alpha=1}^\epsilon (j_{\alpha,2} - j_{\alpha,1}) = \delta$.
For notational convenience, we put $j_{0,1} = j_{0,2} = 0$.

We define $Z_{\bfj}$ to be the subset of $\overline \FF_p$-points $z$ of $\Sh_{r,s}$ such that the reduced Dieudonn\'e modules $\tilde \calD(\calA_z)^\circ_1$ and $\tilde\calD(\calA_z)^\circ_2$ contain submodules $\tilde \calE_1$ and $\tilde \calE_2$ satisfying \eqref{E:condition for getting another abelian varieties} for $m =\epsilon$, i.e.
\[
p^\epsilon\tcD(\calA_{z})^{\circ}_{i}\subseteq \tcE_i, \quad F(\tcE_i)\subseteq \tcE_{3-i}, \quad \textrm{and} \quad  V(\tcE_i)\subseteq \tcE_{3-i},\quad \textrm{for }i=1,2,
\]
 and the following condition for $i=1,2$:
\begin{equation}
\label{E:D over E}
\tilde \calD(\calA_z)^\circ_i/\tilde \calE_i \simeq \big(W(\overline \FF_p) / p^\epsilon\big)^{\oplus j_{1,i}} \oplus \big(W(\overline \FF_p) / p^{\epsilon-1}\big)^{\oplus (j_{2,i}-j_{1,i})} \oplus \cdots \oplus \big(W(\overline \FF_p) / p\big)^{\oplus (j_{\epsilon,i} - j_{\epsilon-1,i})}.
\end{equation}
We refer to the toy model discussed in Example~\ref{E:toy model} for the motivation of this condition.
For technical reasons, we will not prove the set $Z_\bfj$ is the set of $\overline \FF_p$-points of a closed subscheme of $\Sh_{r,s}$; instead we prove that a closely related subset of $Z_\bfj$ is.  See Remark~\ref{R:Yj neq Zj}.

Applying Proposition~\ref{P:abelian-Dieud} with $m=\delta$, the submodules $\tilde \calE_1$ and $\tilde \calE_2$ give rise to a 
polarized abelian variety  $(\calA'_{z},\lambda'_z)$ over $z$ with an $\cO_{D}$-action and an $\cO_{D}$-equivariant  isogeny $\calA'_{z} \to \calA_z$.
Moreover, by \eqref{E:dimension of new differentials}, we have
\[
\dim\omega^\circ_{\calA'^\vee_z/\overline \FF_p, 1} = 
\dim \omega^\circ_{\calA^\vee_z/\overline \FF_p,1} + \sum_{\alpha =0}^{\epsilon-1} \Big((\epsilon-\alpha)(j_{\alpha+1, 1} - j_{\alpha,1}) - (\epsilon-\alpha)(j_{\alpha+1,2} - j_{\alpha,2})\Big) = r-\delta
\]
and similarly $\dim\omega^\circ_{\calA'^\vee_z/\overline \FF_p, 2} = s+\delta$.
So $\calA'_{z}$ satisfies the moduli problem for $\Sh_{r-\delta, s+\delta}$;
  this suggests a geometric relationship between $Z_{\bfj}$ and $\Sh_{r-\delta,s+\delta}$ that we make precise
in Definition~\ref{D:cycle U(r,s)}.

We make an immediate remark that when $\delta = r$, the abelian variety $\calA_z$ coming from a point $z$ of $Z_{\bfj}$ is isogenous to an abelian variety 
$\calA'_z$ that is a moduli object for the Shimura variety $\Sh_{0, n}$.  Thus both $\calA'_z$ and $\calA_z$ are supersingular.
So every $Z_\bfj$ is contained in the supersingular locus of $\Sh_{r,s}$. In fact, we shall show in Theorem~\ref{T:Zj supersingular locus} that the supersingular locus of $\Sh_{r,s}$ is exactly the union of these $Z_\bfj$.

\subsection{Towards a moduli interpretation}
\label{S:towards a moduli}
We need to reinterpret the Dieudonn{\'e}-theoretic condition defining  $Z_{\bfj}$ in a more geometric manner.
For $\alpha = 0, \dots, \epsilon$, we define submodules 
\[
\tilde \calE_{\alpha,1}: = \tilde \calD(\calA_z)^{\circ}_1 \cap \frac 1{p^{\epsilon -\alpha}}\tilde \calE_1\quad \textrm{and} \quad\tilde \calE_{\alpha,2} : =  \tilde \calD(\calA_z)^{\circ}_2 \cap \frac 1{p^{\epsilon -\alpha}} \tilde \calE_2
\]
of $\tilde \calD(\calA_z)^{\circ}_1$ and $\tilde \calD(\calA_z)^{\circ}_2$. They are easily seen to satisfy condition~\eqref{E:condition for getting another abelian varieties} with $m = \alpha$.  Thus, Proposition~\ref{P:abelian-Dieud}
generates a polarized abelian variety $(A_\alpha,\lambda_\alpha)$ with $\cO_D$-action and an $\cO_D$-equivariant isogeny $A_\alpha \rightarrow \calA_z $, where 
\begin{equation}
\label{E:signature of A epsilon}
r_\alpha: = \dim \omega^\circ_{A_\alpha^\vee/\overline \FF_p, 1} = 
r - \sum_{\alpha'=1}^\alpha \big(j_{\alpha',2} - j_{\alpha',1}\big) \quad \textrm{and} \quad s_\alpha: =\dim \omega^\circ_{A_\alpha^\vee/\overline \FF_p,2} = n - \dim \omega^\circ_{A_\alpha^\vee/\overline \FF_p, 1}
\end{equation}
by the formula \eqref{E:dimension of new differentials}.
In particular $r_0 = r$, $s_0 = s$, $r_\epsilon = r-\delta$, and $s_\epsilon = s+\delta$.

In fact, applying Proposition~\ref{P:abelian-Dieud} (with $m=1$) to the sequence of inclusions
\[
\tilde \calE_i  =
\tilde \calE_{\epsilon,i} \subset \tilde \calE_{\epsilon-1,i} \subset
\cdots \subset \tilde \calE_{0,i} = \tcD(\calA_z)^\circ_i,
\]
we
obtain a sequence of isogenies (each with $p$-torsion kernels):
\begin{equation}
\label{E:chain of isogenies}
\calA'_z = 
A_{\epsilon} \xrightarrow{\;\phi_\epsilon\;} A_{\epsilon - 1} \xrightarrow{\phi_{\epsilon-1}} \cdots \xrightarrow{\;\phi_1\;} A_0 = \calA_z.
\end{equation}
We have $\ker \phi_\alpha \subseteq A_\alpha[p]$, so that there exists a unique isogeny $\psi_\alpha: A_{\alpha-1} \rightarrow A_\alpha$ such that $\psi_\alpha \phi_\alpha = p \cdot \id_{A_\alpha}$ and $\phi_\alpha \psi_\alpha = p \cdot \id_{A_{\alpha-1}}$.

For each $\alpha$, the cokernel of the induced map on cohomology
\[
\phi_{\alpha, *,i}: H_1^\dR(A_\alpha /\overline \FF_p)^\circ_i \to H_1^\dR(A_{\alpha-1} / \overline\FF_p)^\circ_i \quad \big(\textrm{resp. }\psi_{\alpha, *,i}: H_1^\dR(A_{\alpha-1} /\overline \FF_p)^\circ_i \to H_1^\dR(A_{\alpha} / \overline\FF_p)^\circ_i  \big)
\]
is canonically isomorphic to  $\tilde \calE_{\alpha-1,i} / \tilde \calE_{\alpha,i}$ (resp. $\tilde \calE_{\alpha,i} / p\tilde \calE_{\alpha-1,i}$), which has dimension $j_{\alpha,i}$ (resp. $n-j_{\alpha,i}$) over $\overline \FF_p$ by a straightforward computation using 
\eqref{E:D over E}.

The upshot is that all these numeric information of the chain of isogenies \eqref{E:chain of isogenies} can be used to reconstruct $\tilde \calE_i$ inside $\tcD(\calA_z)^\circ_i$.
This idea will be made precise after this important example.

\begin{example}
\label{E:toy model}
We give a good toy model for the isogenies of Dieudonn\'e modules.  This is the inspiration of the construction of this section.
We start with the Dieudonn\'e module  $\tcD(A_\epsilon)^\circ_1 = \oplus_{i=1}^n W(\overline \FF_p) \bfe_j$ and $\tcD(A_\epsilon)^\circ_2 = \oplus_{j=1}^n W(\overline \FF_p) \bff_j$.
The maps $V_1: \tcD(A_\epsilon)^\circ_1 \to \tcD(A_\epsilon)^\circ_2$ and $V_2: \tcD(A_\epsilon)^\circ_2 \to \tcD(A_\epsilon)^\circ_1$, with respect to the given bases, are given by the diagonal matrices 
\[
\Diag\big(\underbrace{1,\dots, 1}_{s+\delta}, \underbrace{p, \dots, p}_{r-\delta} \big) \quad \textrm{and} \quad \Diag\big(\underbrace{1,\dots,1}_{r-\delta}, \underbrace{p,\dots,p}_{s+\delta} \big),
\]
respectively.
Using the isogenies $\phi_\alpha$ we may naturally identify $\tcD(A_\alpha)^\circ_i$ as lattices in $\tcD(A_\epsilon)^\circ_i[\frac{1}{p}]$ with induced Frobenius and Verschiebung morphisms.
For our toy model, we choose
\begin{align*}
\tcD(A_\alpha)^\circ_1 = W(\overline \FF_p)\textrm{-span of } &
\tfrac 1{p^{\epsilon-\alpha}}\bfe_1, \dots,\tfrac 1{p^{\epsilon-\alpha}}\bfe_{j_{\alpha+1,1}}, 
\tfrac1{p^{\epsilon-\alpha-1}} \bfe_{j_{\alpha+1,1}+1}, \dots,\tfrac1{p^{\epsilon-\alpha-1}} \bfe_{j_{\alpha+2,1}},\\
&
\tfrac1{p^{\epsilon-\alpha-2}} \bfe_{j_{\alpha+2,1}+1},
\dots, 
\tfrac 1p \bfe_{j_{\epsilon, 1}},  \bfe_{j_{\epsilon,1}+1},\dots,
\bfe_n; \textrm{ and}
\\
\tcD(A_\alpha)^\circ_2 = W(\overline \FF_p)\textrm{-span of } &
\tfrac 1{p^{\epsilon-\alpha}}\bff_1, \dots,\tfrac 1{p^{\epsilon-\alpha}}\bff_{j_{\alpha+1,2}}, 
\tfrac1{p^{\epsilon-\alpha-1}} \bff_{j_{\alpha+1,2}+1}, \dots,\tfrac1{p^{\epsilon-\alpha-1}} \bff_{j_{\alpha+2,2}},\\
&
\tfrac1{p^{\epsilon-\alpha-2}} \bff_{j_{\alpha+2,2}+1},
\dots, 
\tfrac 1p \bff_{j_{\epsilon, 2}},  \bff_{j_{\epsilon,2}+1},\dots,
\bff_n.
\end{align*}
In particular, the Verschiebung $V_1:\tcD(A_\alpha)^\circ_1\to \tcD(A_\alpha)^\circ_2$ with respect to the bases above is given by 
\[
\Diag \big(\underbrace{1,\dots, 1}_{j_{\alpha+1, 1}}, ***\dots *** , \underbrace{p,\dots,p}_{r-\delta} \big),
\]
where the $***$ part is $p$ if the place is in $[j_{\alpha',1}+1,j_{\alpha',2}]$ for some $\alpha' \geq \alpha$, and is $1$ otherwise.
Similarly, the Verschiebung $V_2:\tcD(A_0)^\circ_2\to \tcD(A_0)^\circ_1$ with respect to the bases above is given by 
\[
\Diag \big(\underbrace{1,\dots, 1}_{r-\delta}, \underbrace{p, \dots, p}_{j_{\alpha+1,1} - r+\delta}, ***\cdots ***, \underbrace{p,\dots,p}_{n-j_{\alpha_{\epsilon,2}}} \big),
\]
where the $***$ part is $1$ if the place is in $[j_{\alpha',1}+1,j_{\alpha',2}]$ for some $\alpha' \geq \alpha$, and is $p$ otherwise.

So the sheaf of differentials is given by
\begin{align*}
\omega_{A_\alpha^\vee/\overline \FF_p, 1}^\circ = \overline \FF_p\textrm{-span of } &
\tfrac 1{p^{\epsilon-\alpha}}\bfe_1, \dots,\tfrac 1{p^{\epsilon-\alpha}}\bfe_{r-\delta}, 
\tfrac1{p^{\epsilon-\alpha}} \bfe_{j_{\alpha,1}+1}, \dots,\tfrac1{p^{\epsilon-\alpha}} \bfe_{j_{\alpha,2}},\tfrac1{p^{\epsilon-\alpha-1}} \bfe_{j_{\alpha+1,1}+1}, \dots,\\
& 
\tfrac 1p \bfe_{j_{\epsilon-1,2}}, \bfe_{j_{\epsilon,1}+1},\dots,
\bfe_{j_{\epsilon,2}}; \textrm{ and}
\\
\omega_{A_\alpha^\vee/\overline \FF_p,2}^\circ = \overline \FF_p\textrm{-span of } &
\tfrac 1{p^{\epsilon-\alpha}}\bff_1, \dots,\tfrac 1{p^{\epsilon-\alpha}}\bff_{j_{\alpha+1,1}}, 
\tfrac1{p^{\epsilon-\alpha-1}} \bff_{j_{\alpha+1,2}+1}, \dots,\tfrac1{p^{\epsilon-\alpha-1}} \bff_{j_{\alpha+2,1}},\\
& \tfrac1{p^{\epsilon-\alpha-2}}\bff_{j_{\alpha+2,1}},\dots, \tfrac 1p \bff_{j_{\epsilon,1}}, \bff_{j_{\epsilon,2}+1},\dots,
\bff_{s+\delta-1}.
\end{align*}
\end{example}

\begin{definition}
\label{D:cycle U(r,s)}
Let $\bfj$ be as above.  Define the numbers $j_{\alpha,i}$ as in \eqref{E:break j into intervals} and the numbers $r_\alpha$, $s_\alpha$ as in \eqref{E:signature of A epsilon}.
Let $\underline Y_{\bfj}$ be the functor
taking a locally noetherian ${\FF}_{p^2}$-scheme $S$ to the set of isomorphism classes of tuples
\begin{equation}
\label{E:tuple for underline Yj}\big(A_0, \dots, A_{\epsilon}, \lambda_0, \dots, \lambda_{\epsilon}, \eta_0, \dots, \eta_{\epsilon}, \phi_1, \dots, \phi_{\epsilon}, \psi_1, \dots, \psi_\epsilon \big)
\end{equation}
such that:
\begin{enumerate}
\item For each $\alpha$, $(A_\alpha,\lambda_\alpha,\eta_\alpha)$ is an $S$-point of $\Sh_{r_\alpha,s_\alpha}$.
\item For each $\alpha$, $\phi_\alpha$ is an $\calO_D$-isogeny $A_\alpha \rightarrow A_{\alpha-1}$, with kernel contained in $A_\alpha[p]$, which is compatible with the polarizations
in the sense that $p \lambda_\alpha = \phi_\alpha^{\vee} \circ \lambda_{\alpha-1} \circ \phi_\alpha$ and with the tame level structures in the sense that $\phi_\alpha \circ \eta_\alpha = \eta_{\alpha-1}$.
\item $\psi_\alpha$ is the isogeny $A_{\alpha-1} \to A_\alpha$ such that $\phi_\alpha \psi_\alpha =p \cdot \id_{A_\alpha}$ and $\psi_\alpha \phi_\alpha = p \cdot \id_{A_{\alpha-1}}$.
\item For each $\alpha$ and $i = 1,2$, the cokernel of the induced map $\phi^\dR_{\alpha,*,i}: H_1^{\dR}(A_{\alpha}/S)_i^{\circ} \rightarrow H_1^{\dR}(A_{\alpha-1}/S)_i^{\circ}$
is a locally free $\cO_S$-module  of  rank $j_{\alpha,i}$.
\item
For each $\alpha$ and $i = 1,2$, the cokernel of the induced map $\psi^\dR_{\alpha,*,i}: H_1^{\dR}(A_{\alpha-1}/S)_i^{\circ} \rightarrow H_1^{\dR}(A_{\alpha}/S)_i^{\circ}$
is a locally free $\cO_S$-module of rank $n-j_{\alpha,i}$.\footnote{This is in fact a corollary of (2) and (4).}
\item
For each $\alpha$, $\Ker(\phi^\dR_{\alpha, *,2})$ is contained in $\omega_{A_{\alpha}^\vee/S,2}^\circ$.
\item
For each $\alpha$,
the $(r_{\alpha-1}-r_\alpha+ r_{\epsilon}+1)$st Fitting ideal of the cokernel 
of $\phi_{\alpha, *, 1}^\dR: \omega^\circ_{A_{\alpha}^\vee/S, 1} \to \omega^\circ_{A_{\alpha-1}^\vee/S, 1}$ is zero, or equivalently, 
Zariski locally on $S$, if we represent the map $\phi_{\alpha, *, 1}^\dR: \omega^\circ_{A_{\alpha}^\vee/S, 1} \to \omega^\circ_{A_{\alpha-1}^\vee/S, 1}$
by an $r_{\alpha-1} \times r_\alpha$-matrix (after choosing local bases), 
then all $(r_{\alpha}-r_\epsilon +1) \times (r_{\alpha}-r_\epsilon +1)$-minors vanish.
\item
For each $\alpha$, the $(r_\alpha-r_\epsilon+1)$st Fitting ideal of
the cokernel of $\psi^\dR_{\alpha,*,1}: \omega^\circ_{A_{\alpha-1}^\vee/S,1} \to \omega^\circ_{A_\alpha^\vee/S,1}$ is zero.
\end{enumerate}
Note that conditions (6)(7)(8) are all closed conditions. So the moduli problem $\underline Y_{\bfj}$ is represented by a proper scheme $Y_\bfj$ of finite type over $\FF_{p^2}$.
The moduli space $Y_\bfj$ admits natural maps to $\Sh_{r,s}$ and $\Sh_{r-\delta, s+\delta}$ by sending the tuple \eqref{E:tuple for underline Yj} to $(A_0, \lambda_0, \eta_0)$ and $(A_\epsilon, \lambda_\epsilon, \eta_\epsilon)$, respectively.
\[
\xymatrix{
& Y_\bfj \ar[dl]_{\pr_\bfj} \ar[dr]^{\pr'_{\bfj}}
\\
\Sh_{r,s} && \Sh_{r-\delta, s+\delta}.
}
\]

We also point out that conditions (2) and (3) together imply that, for each $\alpha$ and $i=1,2$, we have $\mathrm{Im} (\psi^\dR_{\alpha,*,i}) = \Ker(\phi^\dR_{\alpha,*,i})$ and $\mathrm{Im}( \phi^\dR_{\alpha,*,i} )= \Ker(\psi^\dR_{\alpha,*,i})$.
We shall freely use this property later.
\end{definition}

\begin{remark}
\label{R:Yj neq Zj}
Conditions (6)(7)(8) in Definition~\ref{D:cycle U(r,s)} are satisfied by the toy model in Example~\ref{E:toy model}. 
They did not appear in moduli problem in Subsection~\ref{S:cycles} because they  trivially hold in that case. The purpose of keeping these conditions in the moduli problem and carefully formulating them is so that the moduli space may hope to have the correct irreducible components. 
We think the picture is the following: $Z_\bfj$ is probably or at least heuristically the set of $\overline \FF_p$-points of a closed subscheme of $\Sh_{r,s}$.
But this scheme has many irreducible components, which may have overlaps with other $Z_{\bfj'}$.
Conditions (6)(7)(8) will help select one irreducible component that is ``special" for $\bfj$.
When taking the union of all images of $Y_\bfj$'s, we should still get the union of $Z_\bfj$'s.
This is verified in the case of supersingular locus (i.e. $r=\delta$) in Theorem~\ref{T:Zj supersingular locus}.
\end{remark}

\begin{notation}
\label{N:Yj circ}
Let $Y_\bfj$ as above. It will be convenient to introduce some dummy notation:
\begin{itemize}
\item
$\phi_0$ is the identity map on $A_0$;
\item
$\psi_\epsilon$ is the identity map on $A_\epsilon$.
\end{itemize}
We use $Y_\bfj^\circ$ to denote the open subscheme of $Y_\bfj$ representing the functor that takes a locally noetherian $\FF_{p^2}$-scheme $S$ to the subset of isomorphism classes of tuples 
$$\big(A_0, \dots, A_{\epsilon}, \lambda_0, \dots, \lambda_{\epsilon}, \eta_0, \dots, \eta_{\epsilon}, \phi_1, \dots, \phi_{\epsilon}, \psi_1, \dots, \psi_\epsilon \big)$$
of $Y_\bfj(S)$ such that
\begin{itemize}
\item[(i)]
for each $\alpha = 1, \dots, \epsilon$, the sum $\phi_{\alpha, *,2}(\omega^\circ_{A^\vee_\alpha/S, 2}) + \Ker(\phi^\dR_{\alpha -1, *, 2})$ is an $\calO_S$-subbundle of $H_1^\dR(A_{\alpha-1}/S)^\circ_2$ of rank
\[
\rank \omega_{A_\alpha^\vee/S, 2} - \rank \Ker(\phi_{\alpha,*,2}^\dR) + \rank \Ker(\phi^\dR_{\alpha -1, *, 2}) = s_\alpha-j_{\alpha,2}+j_{\alpha-1,2},
\]
\item[(ii)]
for each $\alpha =1, \dots, \epsilon$, $\Ker(\phi_{\alpha, *,1}^\dR) + \Ker(\psi_{\alpha+1, *, 1}^\dR)$ is an $\calO_S$-subbundle of rank 
\[
\rank \Ker(\phi_{\alpha, *,1}^\dR) + \rank \Ker(\psi_{\alpha+1, *, 1}^\dR) = 
j_{\alpha, 1} + (n-j_{\alpha+1,1}),
\]
\item[(iii)]
for each $\alpha$, the cokernel of $\phi_{\alpha, *, 1}^\dR: \omega^\circ_{A_{\alpha}^\vee/S, 1} \to \omega^\circ_{A_{\alpha-1}^\vee/S, 1}$
is a locally free $\cO_S$-module of rank $r_{\alpha-1} - (r_\alpha-r_\epsilon)$,
\item[(iv)]
for each $\alpha$, the cokernel of $\psi^\dR_{\alpha,*,1}: \omega^\circ_{A_{\alpha-1}^\vee/S,1} \to \omega^\circ_{A_\alpha^\vee/S,1}$ is a locally free $\cO_S$-module of rank $r_{\alpha} - r_\epsilon$.

\end{itemize}
We note that the ranks in Conditions (i) and (ii) are maximal possible and the ranks in Conditions (iii) and (iv) are minimal possible, under the conditions in Definition~\ref{D:cycle U(r,s)}.  So $Y_\bfj^\circ$ is an open subscheme of $Y_\bfj$.

We point out an additional benefit of having Conditions (ii)--(iv).
By (iii), $\omega^\circ_{A_\alpha^\vee/S, 1} \cap \Ker(\phi_{\alpha, *, 1}^\dR)$ is an $\calO_S$-subbundle of $\omega^\circ_{A_\alpha^\vee/S, 1}
$ of rank $r_\epsilon$, for $\alpha =1, \dots, \epsilon$; by (iv), $\omega^\circ_{A_{\alpha}^\vee/S, 1} \cap \Ker(\psi_{\alpha+1, *, 1}^\dR)$ is an $\calO_S$-subbundle of $\omega^\circ_{A_\alpha^\vee/S, 1}
$ of rank $r_\alpha -r_\epsilon$, for $\alpha = 0, \dots, \epsilon-1$.
Combining these two rank estimates and Condition (ii) which implies that $\Ker(\phi^\dR_{\alpha, *,1})$ and $\Ker(\psi^\dR_{\alpha+1, *, 1})$ are disjoint subbundles, we arrive at a \emph{direct sum} decomposition
\begin{equation}
\label{E:phi psi decomposition}
\omega^\circ_{A^\vee_{\alpha}/S, 1} = \big(\omega^\circ_{A^\vee_{\alpha}/S, 1} \cap \Ker(\phi_{\alpha, *, 1}^\dR) \big) \oplus \big(\omega^\circ_{A^\vee_{\alpha}/S, 1} \cap \Ker(\psi_{\alpha+1, *, 1}^\dR) \big),
\end{equation}
for $\alpha = 1, \dots, \epsilon-1$; and we know that $\omega^\circ_{A^\vee_0/S, 1} \cap \Ker(\psi^\dR_{1, *, 1})$ has rank $r_0-r_\epsilon = \delta$  and $\omega^\circ_{A^\vee_\epsilon/S, 1} \subseteq \Ker(\phi^\dR_{\epsilon, *, 1})$.

We shall show below in Theorem~\ref{T:smoothness Yj circ} that $Y_\bfj^\circ$ is smooth.
Unfortunately, we do not know how to prove the non-emptiness of $Y_\bfj^\circ$, nor do we know if some $Y_\bfj$ is completely contained in some other $Y_\bfj$;
but the fact that the Dieudonne modules in Example~\ref{E:toy model} satisfy conditions (i) - (iv) above is good evidence for this nonemptiness.
Of course, if one can compute the intersection matrix in the sense of Theorem~\ref{Th-Conj} and calculate the determinant, one can then probably show that these $Y_\bfj$ are essentially different.  But the difficulties of this computation lie in understanding the singularities at $Y_\bfj \backslash Y_\bfj^\circ$, which seems to be very combinatorially involved.

\end{notation}

\begin{theorem}
\label{T:smoothness Yj circ}
Each $Y_\bfj^\circ$ is smooth of dimension $rs + (r-\delta)(s+\delta)$ (if not empty).
\end{theorem}
\begin{proof}

Let $\hat R$ be a noetherian $\FF_{p^2}$-algebra and $\hat I \subset \hat R$ an ideal such that $\hat I^2=0$. Put $R = \hat R/\hat I$.
Say we want to lift an $R$-point  $$\big(A_0, \dots, A_{\epsilon}, \lambda_0, \dots, \lambda_{\epsilon}, \eta_0, \dots, \eta_{\epsilon}, \phi_1, \dots, \phi_{\epsilon}, \psi_1, \dots, \psi_\epsilon \big)$$ of $Y_\bfj^\circ$ an $\hat R$-point and we try to compute the corresponding tangent space.
By Serre--Tate and Grothendieck--Messing deformation theory we recalled in Theorem~\ref{T:STGM}, it is enough to lift, for $i =1,2$ and each $\alpha =0, \dots, \epsilon$, the differentials $\omega^\circ_{A_\alpha^\vee/R, i} \subseteq H_1^\dR(A_\alpha/R)^\circ_i$ to a subbundle $\hat \omega_{\alpha,i} \subseteq H_1^\cris(A_\alpha/\hat R)^\circ_i$ such that
\begin{itemize}
\item[(a)]
 $\phi_{\alpha, *,i}^\cris(\hat  \omega_{\alpha,i}) \subseteq \hat \omega_{\alpha-1,i}$ and $\psi_{\alpha, *,i}^\cris(\hat \omega_{\alpha-1,i}) \subseteq \hat \omega_{\alpha,i}$ (so that both $\phi_\alpha$ and $\psi_\alpha$ are lifted, which would automatically imply $\Ker(\phi_\alpha) \in A_\alpha[p]$),
\item[(b)]
$\hat \omega_{\alpha,2} \supseteq \Ker(\phi_{\alpha, *,2}^\cris)$, and
\item[(c)]
$\hat \omega_{\alpha-1,1}  / \phi_{\alpha, *,1}^\cris(\hat  \omega_{\alpha,1})$ is a flat $\hat R$-module of rank $r_{\alpha-1} - (r_\alpha-r_\epsilon)$, and $\hat \omega_{\alpha,1}  / \psi_{\alpha, *,1}^\cris(\hat  \omega_{\alpha-1,1})$ is a flat $\hat R$-module of rank $r_\alpha-r_\epsilon$.
\end{itemize}
We shall see that Condition (i) of Notation~\ref{N:Yj circ} is automatic.  Also, Condition (ii) already holds: since $ H_1^\cris(A_\alpha/\hat R)^\circ_1 \big/ \big( \Ker(\phi_{\alpha, *, 1}^\cris)+ \Ker(\psi_{\alpha+1, *, 1}^\cris) \big)$ is locally generated by $j_{\alpha+1, 1} -j_{\alpha, 1}$ elements after modulo $\hat I$, it is so prior to modulo $\hat I$ by Nakayama's lemma.
Note that rank of $\Ker(\phi_{\alpha, *, 1}^\cris)$ and $\Ker(\psi_{\alpha+1, *, 1}^\cris)$ and the number of the generators of the quotient above add up to exactly $n$; it follows that $\Ker(\phi_{\alpha, *, 1}^\cris)+\Ker(\psi_{\alpha+1, *, 1}^\cris)$ is a direct sum and the sum is a subbundle of $H_1^\cris(A_\alpha/\hat R)^\circ_1$.

We separate the discussion of lifts at $q_1$ and $q_2$, and show that the tangent space $T_{Y_\bfj^\circ}$ is isomorphic to $T_1 \oplus T_2$ for the contributions $T_1$ and $T_2$ from the two places.
We first look at $q_2$, as it is easier.
Note that the condition (b) $\hat \omega_{\alpha, 2} \supseteq \Ker(\phi_{\alpha, *,2}^\cris) = \mathrm{Im}(\psi_{\alpha,*,2}^\cris)$ automatically implies that $\psi_{\alpha, *,2}^\cris(\hat \omega_{\alpha-1,2}) \subseteq \hat \omega_{\alpha,2}$; so we can proceed as follows:
\begin{itemize}
\item[Step 0:]
First lift $\omega^\circ_{A^\vee_\epsilon/R, 2}$ to a subbundle  $\hat \omega_{\epsilon,2}$ of $H_1^\cris(A_\epsilon / \hat R)^\circ_2$ so that it contains $ \Ker(\phi_{\epsilon, *,2}^\cris)$,
\item[Step 1:]
then lift $\omega^\circ_{A^\vee_{\epsilon-1}/R, 2}$ to a subbundle $\hat \omega_{\epsilon-1,2}$ of $H_1^\cris(A_{\epsilon-1} / \hat R)^\circ_2$ so that it contains $\phi^\cris_{\epsilon, *,2}(\hat \omega_{\epsilon,2}) +  \Ker(\phi_{\epsilon-1, *,2}^\cris)$,
\item[$\cdots \cdots$] $\cdots \cdots$
\item[Step $\alpha$:]
then lift $\omega^\circ_{A^\vee_{\epsilon-\alpha}/R, 2}$ to a subbundle $\hat \omega_{\epsilon-\alpha,2}$ of $H_1^\cris(A_{\epsilon-\alpha} / \hat R)^\circ_2$ so that it contains $\phi^\cris_{\epsilon-\alpha+1, *,2}(\hat \omega_{\epsilon-\alpha+1,2}) +  \Ker(\phi_{\epsilon-\alpha, *,2}^\cris)$,
\item[$\cdots \cdots$] $\cdots \cdots$
\item[Step $\epsilon$:]
finally lift $\omega^\circ_{A^\vee_0/R, 2}$ to a subbundle $\hat \omega_{0,2}$ of $H_1^\cris(A_{0} / \hat R)^\circ_2$ so that it contains $\phi^\cris_{1, *,2}(\hat \omega_{1,2})$.
\end{itemize}

At Step 0, the choices form a torsor for the group
\[
\Hom_R\big( \omega^\circ_{A_{\epsilon}^\vee/R,2} \big/\Ker(\phi_{\epsilon, *,2}^\dR), \Lie_{A_{\epsilon}/R,2}^\circ \big) \otimes_R \hat I;
\]
the Hom space is a locally free $R$-module of rank $(s_\epsilon - j_{\epsilon, 2}) r_\epsilon$.

At Step $\alpha = 1, \dots, \epsilon$, 
we observe that condition (i) of the moduli problem $\underline Y_\bfj^\circ$ implies that $\phi_{\epsilon-\alpha+1, *,2}( \omega_{A^\vee_{\epsilon-\alpha+1}/R,2}^\circ) +  \Ker(\phi_{\epsilon-\alpha, *,2}^\dR)$ is an $R$-subbundle of $H_1^\dR(A_{\epsilon-\alpha} / R)_2^\circ$
 of rank
 \begin{equation}
 \label{E:rank of phi+Ker}
s_{\epsilon-\alpha+1} - j_{\epsilon-\alpha+1, 2} + j_{\epsilon-\alpha,2} =s_{\epsilon-\alpha} + (j_{\epsilon-\alpha+1, 1} - j_{\epsilon-\alpha,2})
\end{equation}
if $\alpha =1, \dots, \epsilon-1$, and of rank $s_1 - j_{1,2}$ if $\alpha = \epsilon$.
So $\phi^\cris_{\epsilon-\alpha+1, *,2}(\hat \omega_{\epsilon-\alpha+1,2}) +  \Ker(\phi_{\epsilon-\alpha, *,2}^\cris)$ is an $\hat R$-subbundle of $H_1^\cris(A_{\epsilon-\alpha} / \hat R)_2^\circ$ of the same rank.
The choices of the lifts $\hat \omega_{\epsilon -\alpha, 2}$ form a torsor for the group
\[
\Hom_R\Big( \omega^\circ_{A_{\epsilon-\alpha}^\vee/R,2} \big/\big(\phi_{\epsilon-\alpha+1, *,2}(\omega^\circ_{A_{\epsilon-\alpha+1}^\vee/R, 2}) +\Ker(\phi_{\epsilon-\alpha, *,2}^\dR)\big), \Lie_{A_{\epsilon-\alpha}/R,2}^\circ \Big) \otimes_R \hat I.
\]
By \eqref{E:rank of phi+Ker}, this Hom space is a locally free $R$-module of rank $(j_{\epsilon-\alpha+1, 1} - j_{\epsilon -\alpha, 2}) r_{\epsilon-\alpha}$ if $\alpha =1, \dots, \epsilon-1$ and of rank $(s_0 - (s_1-j_{1,2}))r_0$ if $\alpha =\epsilon$.
This implies that the contribution $T_2$ to the tangent space $T_{Y_\bfj^\circ}$ at $q_2$ admits a filtration such that the subquotients are 
\[
\calH om\Big( \omega^\circ_{\calA_{\epsilon-\alpha}^\vee,2} \big/\big(\phi_{\epsilon-\alpha+1, *,2}(\omega^\circ_{\calA_{\epsilon-\alpha+1}^\vee, 2}) +\Ker(\phi_{\epsilon-\alpha, *,2}^\dR)\big), \Lie_{\calA_{\epsilon-\alpha},2}^\circ \Big)
\]
where the $\calA_{\epsilon-\alpha}$'s are the universal abelian varieties and $\phi_{\epsilon+1, *,2}(\omega^\circ_{\calA_{\epsilon+1}^\vee, 2})$ is interpreted as zero.
In particular, $T_2$ is a locally free sheaf on $Y_\bfj^\circ$ of rank 
\begin{align}
\nonumber
&(s_\epsilon - j_{\epsilon, 2})
r_\epsilon + \big(s_0 - (s_1-j_{1,2})\big)r_0 + 
\sum_{\alpha = 1}^{\epsilon-1}
(j_{\epsilon-\alpha+1, 1} - j_{\epsilon-\alpha, 2}) r_{\epsilon -\alpha}
\\
= &\label{E:dimension of T2}
(s_\epsilon - j_{\epsilon, 2})
r_\epsilon + j_{1,1}r_0 + 
\sum_{\alpha = 1}^{\epsilon-1}
(j_{\alpha+1, 1} - j_{\alpha, 2}) r_{\alpha}.
\end{align}

We now look at the place $q_1$.
By condition (ii), $\phi^\cris_{\alpha, *, 1}$ when restricted to $\Ker(\psi^\cris_{\alpha+1, *, 1})$ is a saturated injection of $\hat R$-bundles; and $\psi^\cris_{\alpha, *, 1}$ when restricted to $\Ker(\phi^\cris_{\alpha-1, *, 1})$ is also a saturated injection of $\hat R$-bundles.
We first recall from the discussion in Notation~\ref{N:Yj circ} especially \eqref{E:phi psi decomposition} that, when $\alpha = 1, \dots, \epsilon-1$, $ \omega^\circ_{A_\alpha^\vee/R, 1}$ is the direct sum of 
\[
\omega^{\circ, \Ker\phi}_{A_\alpha^\vee/R, 1}: = \omega^\circ_{A_\alpha^\vee/R, 1} \cap \Ker(\phi^\dR_{\alpha, *, 1}) \quad \textrm{and} \quad \omega^{\circ, \Ker\psi}_{A_\alpha^\vee/R, 1}: = \omega^\circ_{A_\alpha^\vee/R, 1} \cap \Ker(\psi^\dR_{\alpha+1, *, 1}),
\]
which are locally free $R$-modules of rank $r_\epsilon$ and $r_\alpha - r_\epsilon$, respectively.
Similarly, we put 
\[
\omega^{\circ, \Ker\phi}_{A^\vee_\epsilon/R, 1} := \omega^\circ_{A^\vee_\epsilon/R, 1}, \quad \omega^{\circ, \Ker\psi}_{A^\vee_\epsilon/R, 1}:=0, \quad \textrm{and}\quad
\omega^{\circ, \Ker\psi}_{A^\vee_0/R, 1} = \omega^\circ_{A^\vee_0/R, 1} \cap \Ker(\psi^\dR_{1,*,1});
\]
they have ranks $r_\epsilon$, $0$, and $r_0-r_\epsilon$, respectively.  We shall avoid talking about $\omega^{\circ, \Ker\phi}_{A^\vee_0/R, 1}$ (as it does not make sense) but only psychologically understands it as the process that enlarges $\omega^{\circ, \Ker\phi}_{A^\vee_0/R, 1}$ to $\omega^{\circ}_{A^\vee_0/R, 1}$.

For $\alpha =1, \dots, \epsilon$, the lift $\hat \omega_{\alpha, 1}$ takes the form of $\hat \omega^{\Ker \phi}_{\alpha, 1} \oplus \hat \omega^{\Ker\psi}_{\alpha, 1}$, where the two direct summands are $\hat R$-subbundles of $\Ker(\phi^\cris_{\alpha, *,1})$ and of $\Ker(\psi^\cris_{\alpha+1, *,1})$, lifting $\omega^{\circ, \Ker\phi}_{A_\alpha^\vee/R, 1} $ and $ \omega^{\circ, \Ker\psi}_{A_\alpha^\vee/R, 1}$, respectively.
Whereas, the lift $\hat \omega_{0,1}$ contains the lift $\hat \omega_{0,1}^{\Ker \psi}$ of $\omega^{\circ, \Ker\psi}_{A^\vee_0/R, 1}$ as an $\hat R$-subbundle of $\Ker(\psi^\cris_{1,*,1})$.
Now the compatibility conditions $\phi^\cris_{\alpha, *,1}(\hat \omega_{\alpha, 1}) \subseteq \hat \omega_{\alpha-1,1}$ and  $\psi^\cris_{\alpha, *,1}(\hat \omega_{\alpha-1, 1}) \subseteq \hat \omega_{\alpha,1}$ 
together with the condition (c)
are equivalent to
\[
\phi^\cris_{\alpha, *,1}(\hat \omega_{\alpha, 1}^{\Ker \psi}) \subseteq \hat \omega_{\alpha-1,1}^{\Ker \psi} \quad \textrm{and} \quad 
\psi^\cris_{\alpha, *,1}(\hat \omega_{\alpha-1, 1}^{\Ker\phi}) \subseteq \hat \omega_{\alpha,1}^{\Ker\phi}.
\]
(The condition (c) on ranks of the quotients are also automatic.)
In particular, the tangent space $T_1$ has three contributions, coming from the lifts $\hat \omega_{\alpha, 1}^{\Ker\phi}$ (for $\alpha =1, \dots, \epsilon$), from the lifts $\hat \omega_{\alpha, 1}^{\Ker\psi}$ (for $\alpha =0, \dots, \epsilon$), and from lifting $\omega^\circ_{A^\vee_0/R, 1}$ to an $\hat R$-subbundle $\hat \omega_{0,1}$ of $H_1^\cris(A_0/\hat R)^\circ_1$ containing $\hat \omega_{0,1}^{\Ker \psi}$. We shall use $T_1^{\Ker \phi}$, $T_1^{\Ker \psi}$, and $T_1^{\Ker \phi,0}$ to denote these three parts of the tangent space; and they will sit in an exact sequence
\begin{equation}
\label{E:T1 exact sequence}
0 \to T_1^{\Ker \phi,0} \longrightarrow T_1 \longrightarrow T_1^{\Ker \phi} \oplus T_1^{\Ker \psi} \to 0.
\end{equation}

We first determine the lifts $\hat \omega_{\alpha, 1}^{\Ker \phi}$ for $\alpha =1, \dots, \epsilon$.
For $\hat \omega_{1, 1}^{\Ker \phi}$, it lifts $\omega^{\circ, \Ker \phi}_{A^\vee_1/R, 1}$ as an $\hat R$-subbundle of $H_1^\cris(A_1/\hat R)^\circ_1$ of rank $r_\epsilon$ (with no further constraint).
Then due to the rank constraint (and the injectivity of $\psi^\cris_{\alpha, *, 1}$ when restricted to $\Ker(\phi^\cris_{\alpha+1, *, 1})$), the lift $\hat \omega_{\alpha, 1}^{\Ker\phi}$ for each $\alpha = 2, \dots, \epsilon$  is then forced to be equal to the image
\[
\psi^\cris_{\alpha, *, 1}\circ \dots \circ \psi^\cris_{1, *, 1}(\hat \omega_{1, 1}^{\Ker\phi}).
\] 
So it suffices to consider the choices of the lift $\hat \omega_{1, 1}^{\Ker \phi}$, which form a torsor for the group
\[
\Hom_R \Big(\omega^{\circ, \Ker\phi}_{A^\vee_1/R, 1}, \Ker(\phi^\dR_{1, *, 1}) \big/  \omega^{\circ, \Ker\phi}_{A^\vee_1/R, 1} \Big) \otimes_R \hat I.
\]
This Hom space is  a locally free $R$-module of rank 
\begin{equation}
\label{E:rank T1Ker phi}
r_\epsilon (j_{1,1} - r_\epsilon).
\end{equation}
It follows that the tangent space $T_1^{\Ker\phi}$ is simply just
\[
\calH om\Big( \omega^{\circ, \Ker\phi}_{\calA^\vee_0,1}, \Ker(\phi^\dR_{1,*, 1}) \big/ \omega^{\circ, \Ker\phi}_{\calA^\vee_0,1} \Big).
\]

We now determine the lifts $\hat \omega_{\alpha,1}^{\Ker \psi}$ for $\alpha = 0, \dots, \epsilon$ following the steps below:
\begin{itemize}
\item[Step 0:]
We start with putting $\hat \omega_{\epsilon,1}^{\Ker \psi} =0$ because $\omega^{\circ, \Ker\psi}_{A^\vee_\epsilon/R, 1}$ is,
\item[$\cdots \cdots$] $\cdots \cdots$
\item[Step $\alpha$:]
lift $\omega^{\circ, \Ker\psi}_{A_{\epsilon-\alpha}^\vee/R, 1}$ to a subbundle  $\hat \omega_{\epsilon-\alpha,1}^{\Ker\psi}$ of $\Ker(\psi^\cris_{\epsilon-\alpha+1, *,1})$ so that it contains $\phi^\cris_{\epsilon-\alpha+1}(\hat \omega_{\epsilon-\alpha+1,1}^{\Ker \psi})$,
\item[$\cdots \cdots$] $\cdots \cdots$
\item[Step $\epsilon$:]
finally lift $\omega^{\circ, \Ker\psi}_{A^\vee_0/R, 1}$ to a subbundle $\hat \omega^{\Ker\psi}_{0, 1}$ of $\Ker(\psi_{1, *,1}^\cris)$ so that it contains $\phi^\cris_{1, *,1}(\hat \omega_{1,1}^{\Ker \psi})$.
\end{itemize}

At Step $\alpha =1, \dots, \epsilon$, 
the choices of the lifts $\hat \omega^{\Ker \psi}_{\epsilon -\alpha, 1}$ form a torsor for the group
\[
\Hom_R\Big( \omega^{\circ, \Ker\psi}_{A_{\epsilon-\alpha}^\vee/R,1}\big /\phi_{\epsilon-\alpha+1, *,1}(\omega^{\circ, \Ker \psi}_{A_{\epsilon-\alpha+1}^\vee/R, 1}), \Ker(\psi^\dR_{\epsilon-\alpha+1, *, 1}) \big/ \omega^{\circ, \Ker\psi}_{A_{\epsilon-\alpha}^\vee/R,1} \Big) \otimes_R \hat I.
\]
This Hom space is a locally free $R$-module of rank 
\[
\big((r_{\epsilon-\alpha} - r_\epsilon) - (r_{\epsilon-\alpha+1} - r_\epsilon)\big)\big((n-j_{\epsilon-\alpha+1,1}) - (r_{\epsilon-\alpha} - r_\epsilon) \big).
\]
This implies that the tangent space $T_1^{\Ker \psi}$ admits a filtration such that the subquotients are 
\[
\calH om\Big( \omega^{\circ, \Ker\psi}_{\calA_{\epsilon-\alpha}^\vee,1}\big /\phi_{\epsilon+1-\alpha, *,1}(\omega^{\circ, \Ker \psi}_{\calA_{\epsilon+1-\alpha}^\vee, 1}), \Ker(\psi^\dR_{\epsilon+1-\alpha, *, 1}) \big/ \omega^{\circ, \Ker\psi}_{\calA_{\epsilon-\alpha}^\vee,1} \Big)
\]
In particular, $T_1^{\Ker\psi}$ is a locally free sheaf on $Y_\bfj^\circ$ of rank 
\begin{align}
\nonumber
&\sum_{\alpha =1}^\epsilon \big((r_{\epsilon-\alpha} - r_\epsilon) - (r_{\epsilon-\alpha+1} - r_\epsilon)\big)\big((n-j_{\epsilon-\alpha+1,1}) - (r_{\epsilon-\alpha} - r_\epsilon) \big)
\\
\label{E:rank of T1ker psi}
=&\sum_{\alpha =0}^{\epsilon-1}(r_{\alpha} - r_{\alpha+1})(s_\alpha-j_{\alpha+1,1}+ r_\epsilon).
\end{align}

Finally, we discuss the $\hat R$-module $\hat \omega_{0, 1}$ that lifts $\omega^\circ_{A^\vee_0/R, 1}$ and contains $\hat \omega_{0,1}^{\Ker \psi}$ we obtained earlier.
The lift is subject to one condition: $\hat \omega_{0,1} \subseteq (\psi^\cris_{1,*,1})^{-1}(\hat \omega_{1,1}^{\Ker \phi})$.
So the choices of the lift form a torsor for the group
\[
\Hom_R \Big(\omega^\circ_{A^\vee_0/R, 1} \big/ \omega_{A^\vee_0/R,1}^{\circ, \Ker \psi} ,(\psi^{\dR}_{1,*, 1})^{-1}  (\omega_{A^\vee_1/R, 1}^{\circ, \Ker \phi}) \big/\omega^\circ_{A^\vee_0/R, 1} \Big) \otimes_R \hat I.
\]
This implies that
\[
T_1^{\Ker\phi,0} = \calH om\Big(\omega^\circ_{\calA^\vee_0, 1} \big/ \omega_{\calA^\vee_0,1}^{\circ, \Ker \psi}, (\psi^{\dR}_{1,*, 1})^{-1}  (\omega_{\calA^\vee_1, 1}^{\circ, \Ker \phi}) \big/\omega^\circ_{\calA^\vee_0, 1} \Big),
\]
which is locally free of rank
\begin{equation}\label{E:rank of T1Kerphi0}
\big(r_0 - (r_0 -r_\epsilon)\big) \big( (r_\epsilon + n-j_{1,1}) - r_0 \big)
=r_\epsilon (s_0 + r_\epsilon -j_{1,1}).
\end{equation}

To sum up, 
the tangent space $T_{Y_\bfj^\circ}$, as the direct sum $ T_1 \oplus T_2$ with $T_1$ sitting in the exact sequence \eqref{E:T1 exact sequence}, is a locally free sheaf
of rank given by $\eqref{E:rank of T1Kerphi0}+\eqref{E:rank T1Ker phi}+\eqref{E:rank of T1ker psi}+ \eqref{E:dimension of T2} $, that is
\begin{align*}
& r_\epsilon (s_0 + r_\epsilon -j_{1,1})+ r_\epsilon (j_{1,1}-r_\epsilon) + \sum_{\alpha =0}^{\epsilon-1}(r_{\alpha} - r_{\alpha+1})(s_\alpha-j_{\alpha+1,1}+ r_\epsilon)
\\
& \quad \quad  + (s_\epsilon - j_{\epsilon, 2})
r_\epsilon + j_{1,1}r_0 + 
\sum_{\alpha = 1}^{\epsilon-1}
(j_{\alpha+1, 1} - j_{\alpha, 2}) r_{\alpha}
\\
=\;&
r_\epsilon s_0 + \sum_{\alpha =0}^{\epsilon-1}r_{\alpha}(s_\alpha-j_{\alpha+1,1}+ r_\epsilon) -
 \sum_{\alpha =1}^{\epsilon}r_{\alpha}(s_{\alpha-1}-j_{\alpha,1}+ r_\epsilon)
\\
& \quad \quad  + (s_\epsilon - j_{\epsilon, 2})
r_\epsilon + j_{1,1}r_0 + 
\sum_{\alpha = 1}^{\epsilon-1}
(j_{\alpha+1, 1} - j_{\alpha, 2}) r_{\alpha}
\\
=\;& r_\epsilon s_0 + r_0(s_0 -j_{1,1} +r_\epsilon) + r_\epsilon(s_{\epsilon-1} -j_{\epsilon,1}+r_\epsilon) + (s_\epsilon - j_{\epsilon, 2})r_\epsilon + j_{1,1}r_0
\\
&\quad \quad \sum_{\alpha =1}^{\epsilon-1} r_\alpha \Big((s_\alpha-j_{\alpha+1,1}+ r_\epsilon) - (s_{\alpha-1}-j_{\alpha,1}+ r_\epsilon) +(j_{\alpha+1, 1} - j_{\alpha, 2}) \Big).
\end{align*}
One easily checks that the first line adds up to $r_\epsilon s_\epsilon + r_0s_0$, and the second line cancels to zero.
This concludes the proof of this Theorem.
\end{proof}

In the special  case of $\delta = r$, each abelian variety $A_\alpha$ appearing in the moduli problem of $Y_{\bfj}$ is isogenous to $A_\epsilon$, which is a certain abelian variety parametrized by the discrete Shimura variety $\Sh_{0, n}$ and is hence supersingular (by Remark~\ref{R:Sh0n supersingular}).
So in particular, the image $\pr_{\bfj}(Y_\bfj)$ in this case is contained in the supersingular locus of $\Sh_{r,s}$.  In fact, the converse is also true.
\begin{theorem}
\label{T:Zj supersingular locus}
Assume $\delta =r$.  The supersingular locus of $\Sh_{r,s}$ is the union of all $\pr_{\bfj}(Y_{\bfj})$.
\end{theorem}
\begin{proof}
We say a finite torsion $W(\overline \FF_p)$-module has \emph{divisible sequence} $(a_1, a_2, \dots, a_\epsilon)$ with nonnegative integers $a_1 \leq \cdots \leq a_\epsilon$ if it is isomorphic to
\[
\big( W(\overline \FF_p)/p^\epsilon \big)^{\oplus{a_1}} \oplus \big( W(\overline \FF_p)/p^{\epsilon-1} \big)^{\oplus(a_2 - a_1)} \oplus  \cdots \oplus \big( W(\overline \FF_p) / p \big)^{\oplus(a_\epsilon - a_{\epsilon-1})}.
\]
The following is an elementary linear algebra fact, whose proof we leave as an exercise.

{\bf Claim:} If $M_1$ and $M_2$ are two torsion $W(\overline \FF_p)$-modules with divisible sequences $(a_{1,i}, \dots, a_{\epsilon,i})$ for $i=1,2$ respectively, and if $M_1 \subseteq M_2$, then we have  $a_{\alpha, 1} \leq a_{\alpha,2}$ for all $\alpha = 1, \dots, \epsilon$.

The proof of the Theorem is similar to the proof of Proposition~\ref{P:normal-bundle}(3), which is a special case of this theorem. It suffices to look at the closed points of $\Sh_{r,s}$.
Let $z = (\calA_z, \lambda, \eta)\in \Sh_{r,s}(\overline \FF_p)$ be a supersingular point. Consider
\[
\LL_\QQ = \big( \tilde \calD(\calA_z)^\circ_1[1/p] \big)^{F^2=p} = \big\{ a \in \tilde \calD(\calA_z)^\circ_1[1/p] \; \big|\; F^2(a) = pa \big\}.
\]
Since $x$ is supersingular, $\LL_\QQ$ is a $\QQ_{p^2}$-vector space of dimension $n$, and $\tilde \calD(\calA_z)^\circ_1[1/p]$ may be identified with the extension of scalars of $\LL_\QQ$ from $\QQ_{p^2}$ to $W(\overline \FF_p)[1/p]  $.
We put
\[
\tilde \calE^\circ_1 = \big( \LL_\QQ \cap \tilde \calD(\calA_z)^\circ_1 \big) \otimes_{\ZZ_{p^2}} W(\overline \FF_p) \quad \textrm{and} \quad \tilde \calE_2^\circ = F(\tilde \calE_1^\circ) = V(\tilde \calE_1^\circ) \subseteq \tilde \calD(\calA_z)^\circ_2.
\]
Then we have 
\begin{equation}
\label{E:D over E supsersingular}
\tilde \calD(\calA_z)^\circ_i/\tilde \calE_i \simeq \big(W(\overline \FF_p) / p^\epsilon\big)^{\oplus j_{1,i}} \oplus \big(W(\overline \FF_p) / p^{\epsilon-1}\big)^{\oplus (j_{2,i}-j_{1,i})} \oplus \cdots \oplus \big(W(\overline \FF_p) / p\big)^{\oplus (j_{\epsilon,i} - j_{\epsilon-1,i})},
\end{equation}
for non-decreasing sequences $0 \leq j_{1,i}\leq j_{2,i} \leq \cdots \leq j_{\epsilon, i} \leq n$ with $i=1,2$; in other words, $\tilde \calD(\calA_z)^\circ_i / \tilde \calE_i$ has divisible sequence $(j_{1,i}, \dots, j_{\epsilon,i})$.
Without loss of generality, we assume that $j_{1,1}$ and $j_{1,2}$ are not both zero.
The essential part of the proof consists of checking the sequence of inequalities
\begin{equation}
\label{E:j inequality}
0 \leq j_{1,1} < j_{1,2}< j_{2,1}<j_{2,2} < \cdots < j_{\epsilon,1}<j_{\epsilon,2} \leq n.
\end{equation}

We first prove \eqref{E:j inequality} with all strict inequalities replaced by non-strict ones.
Indeed, the obvious inclusion $F(\tilde \calD(\calA_z)^\circ_i) \subseteq \tilde \calD(\calA_z)^\circ_{3-i}$ implies that
\[
F\big(\tilde \calD(\calA_z)^\circ_1 / \tilde \calE_1 \big) = 
F(\tilde \calD(\calA_z)^\circ_1) / \tilde \calE_2 \subseteq \tilde \calD(\calA_z)^\circ_2 / \tilde \calE_2, \quad \textrm{and}
\]
\[
F\big(\tilde \calD(\calA_z)^\circ_2 / \tilde \calE_2 \big) = 
F(\tilde \calD(\calA_z)^\circ_2) / p\tilde \calE_1 \subseteq \tilde \calD(\calA_z)^\circ_1 / p \tilde \calE_1.
\]
By \eqref{E:D over E supsersingular}, the first inclusion embeds a torsion $W(\overline \FF_p)$-module with divisible sequence $(j_{1,1}, \dots, j_{\epsilon,1})$ into  a torsion $W(\overline \FF_p)$-module with divisible sequence $(j_{1,2}, \dots, j_{\epsilon,2})$. The Claim above implies that $j_{\alpha, 1} \leq j_{\alpha,2}$ for all $\alpha =1, \dots, \epsilon$.
Similarly, by \eqref{E:D over E supsersingular}, the second inclusion embeds a torsion $W(\overline \FF_p)$-module with divisible sequence $(j_{1,2}, \dots, j_{\epsilon,2})$ into  a torsion $W(\overline \FF_p)$-module with divisible sequence $(j_{1,1}, \dots, j_{\epsilon,1}, n)$. The Claim above implies that $j_{\alpha, 2} \leq j_{\alpha+1,1}$ for all $\alpha =1, \dots, \epsilon-1$, and $j_{\epsilon,2} \leq n$.

We now use the construction of $\LL_\QQ$ to show the strict inequalities in \eqref{E:j inequality}.
Suppose first that $j_{\alpha,1} = j_{\alpha,2}$ for some $\alpha =1, \dots, \epsilon$.
Then it follows that the maps
\begin{equation}
\label{E:FV isomorphisms}
F,\ V: \Big( p^{\epsilon-\alpha} \tilde \calD(\calA_z)^\circ_1 \cap \frac 1{p} \tilde \calE_1^\circ \Big) +\tilde \calE_1^\circ \longrightarrow \Big( p^{\epsilon-\alpha}\tilde \calD(\calA_z)^\circ_2 \cap \frac 1{p} \tilde \calE_2^\circ \Big) + \tilde \calE_2^\circ
\end{equation}
are both isomorphisms (due to an easy length computation as $\tilde \calE_2^\circ =F(\tilde \calE_1^\circ) = V(\tilde \calE_1^\circ)$).
By the definition of $\LL_\QQ$ and $\tilde \calE_1^\circ$, we must have
\[
\Big(
\Big( p^{\epsilon-\alpha} \tilde \calD(\calA_z)^\circ_1 \cap \frac 1{p} \tilde \calE_1^\circ \Big) +\tilde \calE_1^\circ \Big)^{F = V} \subseteq \LL_\QQ \cap \tilde \calD(\calA_z)^\circ_1 \subseteq \tilde \calE_1^\circ.
\]
But this is absurd because the isomorphisms \eqref{E:FV isomorphisms} implies by Hilbert 90 that the left hand side above generates the source of \eqref{E:FV isomorphisms}, which is clearly not contained in $\tilde \calE_1^\circ$.

Similarly, suppose that $j_{\alpha,2} = j_{\alpha+1,1}$ for some $\alpha=1, \dots, \epsilon-1$.
Then the following morphisms are isomorphisms
\begin{equation}
\label{E:FV isomorphisms2}
F,\ V: \Big( p^{\epsilon-\alpha} \tilde \calD(\calA_z)^\circ_2 \cap \frac 1{p} \tilde \calE_2^\circ \Big) +\tilde \calE_2^\circ \longrightarrow \Big( p^{\epsilon-\alpha}\tilde \calD(\calA_z)^\circ_1 \cap \tilde \calE_1^\circ \Big) + p\tilde \calE_1^\circ,
\end{equation}
because $p\tilde \calE_1^\circ =F(\tilde \calE_2^\circ) = V(\tilde \calE_2^\circ)$ and for length reasons.
By the definition of $\LL_\QQ$ and $\tilde \calE_1^\circ$, we must have
\[
\Big(
\Big( p^{\epsilon-\alpha} \tilde \calD(\calA_z)^\circ_1 \cap  \tilde \calE_1^\circ \Big) +p\tilde \calE_1^\circ \Big)^{F^{-1} = V^{-1}} \subseteq \LL_\QQ \cap p\tilde \calD(\calA_z)^\circ_1 \subseteq p\tilde \calE_1^\circ.
\]
(Note that $\epsilon -\alpha \geq 1$ now.)
But this is absurd because the isomorphisms \eqref{E:FV isomorphisms2} implies by Hilbert 90 that the left hand side above generates the target of \eqref{E:FV isomorphisms2}, which is clearly not contained in $p\tilde \calE_1^\circ$.

Summing up, we have proved the strict inequalities \eqref{E:j inequality}. So the $j_{\alpha, i}$'s define a $\bfj$ as in the beginning of Subsection~\ref{S:description of cycles}.  We now construct a point of $Y_\bfj$ which maps to the point $z \in \Sh_{r,s}$.  Put 
\begin{equation}
\label{E:defintiion of calE alpha}
\tilde \calE_{\alpha,1}: = \tilde \calD(\calA_z)^{\circ}_1 \cap \frac 1{p^{\epsilon -\alpha}}\tilde \calE_1\quad \textrm{and} \quad\tilde \calE_{\alpha,2} : =  \tilde \calD(\calA_z)^{\circ}_2 \cap \frac 1{p^{\epsilon -\alpha}} \tilde \calE_2.
\end{equation}
Using the exact construction in Subsection~\ref{S:towards a moduli}, we get the sequence of isogenies of abelian varieties
\[
\xymatrix{
A_\epsilon \ar@<2pt>[r]^{\phi_\epsilon} & A_{\epsilon -1} \ar@<2pt>[l]^{\psi_\epsilon} \ar@<2pt>[r] ^{\phi_{\epsilon-1}} & \ar@<2pt>[l]^{\psi_{\epsilon-1}} \cdots \ar@<2pt>[r]^-{\phi_1} & \ar@<2pt>[l]^-{\psi_1} A_0 = \calA_z,
}
\]
such that $A_\alpha$ together with the induced polarization $\lambda_\alpha$ and the tame level structure $\eta_\alpha$ gives an $\overline \FF_p$-point of $\Sh_{r_\alpha, s_\alpha}$, and $\tilde \calD(A_\alpha)^\circ_i = \tilde \calE_{\alpha,i}$ for all $\alpha$ and $i=1,2$. 

Conditions (2)--(5) of Definition~\ref{D:cycle U(r,s)} easily follow from the description of the quotients $\tilde \calD(\calA_z)^\circ_i/\tilde \calE_i$ in \eqref{E:D over E supsersingular}.
Condition (6) of Definition~\ref{D:cycle U(r,s)} is equivalent to
\[
p \tilde \calD(A_{\alpha-1})^\circ_2 \subseteq
V(\tilde \calD(A_{\alpha})^\circ_1).
\]
By the construction of these Dieudonn\'e modules in \eqref{E:defintiion of calE alpha}, this is equivalent to
\[
p \Big( \tilde \calD(\calA_z)^\circ_2 \cap \frac{1}{p^{\epsilon-\alpha+1}} \tilde \calE_2 \Big) \subseteq V\Big(\tilde \calD(\calA_z)^\circ_1 \cap \frac{1}{p^{\epsilon-\alpha}} \tilde \calE_1 \Big).
\]
But this follows from $p\tilde \calD(\calA_z)^\circ_2 \subseteq V\tilde \calD(\calA_z)^\circ_1$ and $ \tilde \calE_2 = V \tilde \calE_1$.
Condition (7) of Definition~\ref{D:cycle U(r,s)} is equivalent to $\omega^\circ_{A_\alpha^\vee/\overline \FF_p, 1} \cap \Ker(\phi^\dR_{\alpha, *, 1})$ has dimension $r_\epsilon$, which is zero in our case.
Translating it into the language of Dieudonn\'e modules, this is equivalent to
\[
V\tilde \calD(A_{\alpha})^\circ_2 \cap p \tilde \calD(A_{\alpha-1})^\circ_1
= p \tilde \calD(A_{\alpha})^\circ_1.
\]
By the construction of these Dieudonn\'e module in \eqref{E:defintiion of calE alpha}, this is equivalent to
\[
\Big(V \tilde \calD(\calA_z)^\circ_2 \cap \frac 1{p^{\epsilon-\alpha}} V\tilde \calE_2\Big) \cap \Big(p \tilde \calD(\calA_z)^\circ_1 \cap \frac 1{p^{\epsilon-\alpha}} \tilde \calE_1 \Big) = p \tilde \calD(\calA_z)^\circ_1 \cap \frac 1{p^{\epsilon-\alpha-1}} \tilde \calE_1,
\]
which follows from observing that $V\tilde \calD(\calA_z)^\circ_2 \supseteq p \tilde \calD(\calA_z)^\circ_1
$ and $V\tilde \calE_2 = p\tilde \calE_1$.
Condition (8) of Definition~\ref{D:cycle U(r,s)} is equivalent to $\omega^\circ_{A_{\alpha-1}^\vee/\overline \FF_p, 1} \subseteq \Ker(\psi^\dR_{\alpha, *, 1})$ (note that $r_\epsilon =0$ in our case).
Translating it into the language of Dieudonn\'e modules and using \eqref{E:defintiion of calE alpha}, this is equivalent to
\[
V\tilde \calD(A_{\alpha-1})^\circ_2 \subseteq \tilde \calD(A_\alpha)^\circ_1,\textrm{ or equivalently, }
V \tilde \calD(\calA_z)^\circ_2 \cap \frac 1{p^{\epsilon-\alpha+1}} V\tilde \calE_2 \subseteq  \tilde \calD(\calA_z)^\circ_1 \cap \frac 1{p^{\epsilon-\alpha}} \tilde \calE_1,
\]
which follows from observing that $V\tilde \calD(\calA_z)^\circ_2 \subseteq  \tilde \calD(\calA_z)^\circ_1
$ and $V\tilde \calE_2 = p\tilde \calE_1$.
This concludes the proof of the Theorem.
\end{proof}

\begin{conjecture}
\label{Conj:quadratic-case}
The varieties $Y_{\bfj}$ together with the natural morphisms  to $\Sh_{r-\delta, s+\delta}$ and $\Sh_{r,s}$ satisfy the condition (3) of Conjecture~\ref{Conj:main}.
Moreover, the union of the images of $Y_{\bfj}$ in $\Sh_{r,s}$ is the closure of the locus where the Newton polygon of the universal abelian variety has slopes $0$ and  $1$ each with multiplicity $2(r-\delta)n$, and slope $\frac 12$ with multiplicity $2(n-2r+2\delta)n$.
\end{conjecture}
This Conjecture in the case of $r=\delta = 1$ was proved in Theorem~\ref{T:main-theorem}.

\bigskip
\setcounter{section}{0}
\setcounter{subsection}{0}
\renewcommand{\thesection}{\Alph{section}}





\if
false

\yichao{Old proof:}

\begin{proof}[Proof of Lemma~\ref{L:change-signature}]
The proof is similar to \cite[Lemma I.7.1]{harris-taylor}. We will look for an $\alpha\in D^{\times}$ such that $\beta_{b_{\bullet}}=\alpha\beta_{a_{\bullet}}$ satisfies the conditions of the Lemma.
  Let $G_{a_{\bullet}}^{1,\mathrm{ad}}$ denote the adjoint group of $G_{a_{\bullet}}^1$. Note that $G_{a_{\bullet}}^{1,\mathrm{ad}}$ is the Weil restriction to $\Q$ of an algebraic group $PG^1_{a_{\bullet}}$ over $F$.
  The condition that $\beta_{b_{\bullet}}\in (D^{\times})^{*=-1}$ is equivalent to the condition that $\beta_{a_{\bullet}}\alpha^*\beta_{a_{\bullet}}^{-1}=\alpha$.
   Thus $\alpha$ defines a class in $H^1(E/F, PG^1_{a_{\bullet}})$. Conversely, every class in $H^1(E/F,PG^1_{a_{\bullet}})$ arises in this way (See \emph{loc. cit.} for details).

  Clozel shows in \cite[Lemma 2.1]{clozel} that if $n$ is odd then the natural restriction map
  \[
  H^1(F,PG^1_{a_{\bullet}})\longrightarrow \bigoplus_{x} H^1(F_{x}, PG^1_{a_{\bullet}})
  \]
  is surjective, where $x$ runs through all places of $F$.  If $n$ is even he shows that   there is an exact sequence
  \[
  H^1(F,PG^1_{a_{\bullet}})\longrightarrow \bigoplus_{x} H^1(F_{x},PG^1_{a_{\bullet}})\xra{\phi=\sum_x \phi_x} \Z/2\Z\ra 0.
  \]
 Here, $H^1(F_x, PG^1_{a_{\bullet}})$ is  the pointed set  classifying the inner forms of $G^1_{a_{\bullet},F_x}$ over $F_x$. 
If $x=\tau_i$ is an infinite place of $F$, 
the map $\phi_{\tau_i}: H^1(F_{\tau_i}, PG^1_{a_{\bullet}})\ra \Z/2\Z$  sends the class of $U(r_i, n-r_i)$ to $r_i-a_i\mod 2$.

Since $\sum_i a_i \equiv \sum_i b_i \pmod 2$ if $n$ is even,   there exists a class $\eta \in H^1(F,PG^1_{a_\bullet})$ whose image in $H^1(F_x, PG^1_{a_{\bullet}})$
\begin{itemize}
\item is the distinguished class given by $G^1_{a_{\bullet},F_x}$ itself  if $x$ a finite place,
\item and represents the class $U(b_i, n-b_i)$ if $x=\tau_i$ is an infinite place.
\end{itemize}
Note that this $\eta$ maps to zero  in $H^1(E,PG^1_{a_{\bullet}})$, because the natural map  
$$H^1(E,PG^1_{a_{\bullet}})\ra \bigoplus_{x} H^1(E_x, PG^1_{a_{\bullet}})$$ is injective. Thus $\eta$ is the image of some class $[\alpha]\in H^1(E/F,PG^1_{a_{\bullet}})$ for some $\alpha\in D^{\times}$ with $\beta_{a_{\bullet}}\alpha^*\beta_{a_{\bullet}}^{-1}=\alpha$. This $\alpha$ satisfies the  desired property by construction.
\end{proof}
\fi

\section{An explicit formula in the local spherical Hecke algebra for $\GL_n$}
\label{Sec:appendix A}

In this appendix,  let $F$ be a local field with ring of integers $\cO$,  $\varpi\in \cO$ be a uniformizer, $\FF=\cO/\varpi\cO$ and $q=\#\FF$. Fix an integer $n\geq 1$. We consider the spherical Hecke algebra $\scrH_K=\Z[K\backslash \GL_n(F)/K]$ with $K=\GL_n(\cO)$.
Here, the product of two double cosets $u=KxK$ and $v=KyK$ in $\scrH_K$ is defined as \begin{equation}
\label{E:double-coset}
u\cdot v=\sum_{w} m(u,v;w)\, w, \footnote{
We may also view elements of $\scrH_K$ as $\Z$-valued locally constant and compactly supported functions on $\GL_{n}(F)$ which are bi-invariant under $K$, and define the  product of  $f,g\in \scrH_K$ as  $(f*g)(x)=\int_{\GL_n(F)}f(y)g(y^{-1}x) dy$, where $dy$ means the unique bi-invariant Haar measure on $\GL_n(F)$ with $\int_K dy=1$. For the equivalence between these two definitions, see \cite[p.4]{gross}.}
\end{equation}
where the sum runs through all the double cosets $w=KzK$ contained in $KxKyK$, and the coefficient $m(u,v;w)\in \Z$ is determined as follows:
If $KxK=\coprod_{i\in I} x_iK$ and $KyK=\coprod_{j\in J} y_jK$, then
\begin{equation}\label{E:coefficient-Hecke}
m(u,v;w)=\#\big\{(i,j)\in I\times J \;|\;  x_iy_jK= zK\text{ for a fixed element $z$ in  $w$}\big\}.
\end{equation}
 By the theory of elementary divisors, all double cosets $KxK$ are of the form
\[
T(a_1,\dots, a_n):=K \, \diag(\varpi^{a_1},\dots, \varpi^{a_n})\, K, \quad \text{for $a_i\in \Z$ with }a_1\geq a_2\geq \cdots\geq a_n.
\]
They form a $\Z$-basis of $\scrH_K$. We put
 \begin{align*}
 T^{(r)}&=T(\underbrace{1,\dots,1}_r,\underbrace{0,\dots, 0}_{n-r})\quad \text{for }0\leq r\leq n,\\
 R^{(r,s)}&=T(\underbrace{2,\dots, 2}_{r},\underbrace{1,\dots, 1}_{s-r}, \underbrace{0,\dots, 0}_{n-s}) \quad\text{for }0\leq r\leq s\leq n.
 \end{align*}
In particular, $R^{(0,s)}=T^{(s)}$ and $T^{(0)}=[K]$.

Because of the lack of reference, we include a proof of the following.
\begin{proposition}\label{P:multiplication-Hecke}
For $1\leq r\leq n$, let
\begin{equation}
\label{E:gaussian-binom}
\binom{n}{r}_{q}=\frac{(q^n-1)(q^{n-1}-1)\cdots (q^{n-r+1}-1)}{(q-1)(q^2-1)\cdots (q^r-1)}
\end{equation}
 be the Gaussian binomial coefficients, and put $\binom{n}{0}_{q}=1$. Then for $0\leq r\leq s\leq n$, we have
 \begin{align*}
 T^{(r)}T^{(s)}= \sum_{i=0}^{\min\{r, n-s\}} \binom{s-r+2i}{i}_q R^{(r-i, s+i)}.
 \end{align*}
\end{proposition}
 \begin{proof}
 We fix a set of representatives $\tilde \FF\subseteq \cO$ of $\FF=\cO/\varpi\cO$ which contains $0$. Then we have
$
 T^{(r)}= \coprod_{x\in \calS(n,r)} xK,
 $
 where $\calS(n,r)$ is the set of  $n\times n$ matrices $x=(x_{i,j})_{1\leq i,j\leq n}$ such that
 \begin{itemize}
 \item $r$ of the diagonal entries are equal to $\varpi$ and the reminding $n-r$ ones are qual to $1$;
 \item  if $i\neq j$, then  $x_{i,j}=0$ unless $i>j$, $x_{i,i}=1$ and $x_{j,j}=\varpi$, in which case
 $x_{i,j}$ can take any values in $\tilde \FF$.
 \end{itemize}
 For instance,  the set $\calS(3,2)$ consists of matrices:
 \[
 \begin{pmatrix}
 1&0&0\\
 x_{2,1}&\varpi&0\\
 x_{3,1}&0&\varpi
 \end{pmatrix},\quad
 \begin{pmatrix}
 \varpi&0&0\\
 0&1&0\\
 0&x_{3,2}&\varpi
 \end{pmatrix},\quad
 \begin{pmatrix}\varpi &0&0\\
 0&\varpi&0\\
 0&0&1\end{pmatrix}
  \]
 with $x_{2,1},x_{3,1},x_{3,2}\in \tilde\FF$.
 We have a similar decomposition  $T^{(s)}=\coprod_{y\in \calS(n,s)} yK$.
  We write  $T^{(r)}T^{(s)}$ as a linear combination of $T(a_{1},\dots, a_n)$ with $a_i\in \ZZ$ and $a_{1}\geq \cdots\geq a_n$.
By looking at the diagonal entries of $xy$, we see easily that only $R^{(r-i,s+i)}$ with $0\leq i\leq \min\{r, n-s\}$ have non-zero coefficients, namely, we have
\[
T^{(r)}T^{(s)} = \sum_{i=0}^{\min\{r, n-s\}} C^{(r,s)}(n,i) R^{(r-i, s+i)} \quad \textrm{ for some }C^{(r,s)}(n,i) \in \ZZ.
\]
Here by \eqref{E:double-coset}, $C^{(r,s)}(n,i)$ is the number of pairs $(x,y)\in \calS(n,r)\times \calS(n,s)$ such that
 \[
xy K=\diag(\underbrace{\varpi^{2},\dots, \varpi^2}_{r-i},\underbrace{ \varpi,\dots,\varpi}_{s-r+2i},  \underbrace{1,\dots, 1}_{n-s-i})K.
\]
In this case, $x$ and $y$ must be of the form
\[
x=\begin{pmatrix} \varpi I_{r-i} &0&0\\
0&A &0\\
0&0&I_{n-s-i}
\end{pmatrix},
\quad y=\begin{pmatrix}
 \varpi I_{r-i}& 0&0\\
0&B&0\\
0&0&I_{n-s-i}
\end{pmatrix},
\]
where  $I_k$ denotes the $k\times k$ identity matrix, and  $A\in \calS(s-r+2i, i)$, $B\in  \calS(s-r+2i, s-r+i)$ satisfying $ AB\cdot \GL_{s-r+2i}(\cO)= \varpi I_{s-r+2i}\GL_{s-r+2i}(\cO)$.
By \eqref{E:double-coset}, we see that
$C^{(r,s)}(n,i)=C^{(i,s-r+i)}(s-r+2i, i)$.
 Therefore, one is reduced to proving the following Lemma, which is a special case of our proposition.
 \end{proof}

 \begin{lemma}
 Under the notation and hypothesis of Proposition above, assume moreover that $n=r+s$. Then the coefficient of $R^{(0,n)}$ in the product  $T^{(r)}T^{(s)}$ is $\binom{n}{r}_q$.
 \end{lemma}
 \begin{proof}
 We make an induction on $n\geq1$. The case  $n=1$ is trivial. We assume thus $n>1$, and that  the statement is true when $n$ is replaced by $n-1$. The case of $r=0$ being trivial, we may assume that $r\geq1$.
 We say a pair $(x,y)\in \calS(n,r)\times \calS(n,n-r)$ is admissible if  $xy K=\varpi I_{n}\, K$. We have to show that the number of admissible pairs is   equal to $\binom{n}{r}_q$.
 Let $(x,y)$ be an admissible  pair.
 Denote by $I$  (resp. by $J$) the set integers $1\leq i\leq n$ such that $x_{i,i}=\varpi$ (resp. $y_{i,i}=\varpi$).
Note that $(x,y)$ being admissible implies that  $J=\{1,\dots,n\}\backslash I$.

 Assume first that $x_{1,1}=1$. Then $x$ and $y$ must be of the form
 $
x= \begin{pmatrix} 1 &0\\
 *& A
 \end{pmatrix}\quad \text{and}\quad
 y=\begin{pmatrix}\varpi & 0\\ 0&B\end{pmatrix}
 $
 where $(A,B)\in \calS(n-1, r)\times \calS(n-1,n-1-r)$ admissible.
  Note that $xy K= \varpi I_n\, K$ always hold. We have $x_{i,1}=0$ for $i\notin I$,  and $x_{i,1}$ can take any values in $\FF$ for $i\in I$.
   Therefore, the number of admissible pairs $(x,y)$ with $x_{1,1}=1$ is equal to $q^{\#I}=q^r$ times that of admissible $(A,B)$'s. The latter is equal to $\binom{n-1}{r}_{q}$ by the induction hypothesis.

   Consider now the case $x_{1,1}=\varpi$. One can write
     $
   x=\begin{pmatrix}\varpi & 0\\ 0& A\end{pmatrix}$,
    and $y=\begin{pmatrix}1 &0\\
   *&B\end{pmatrix}$
  with $(A,B)\in \calS(n-1, r-1)\times \calS(n-1, n-r)$ admissible. Put $z=xy$.
  Then an easy computation shows that $z_{j,1}=y_{j,1}$ if $j\in J$, and $z_{j,1}=0$ if $j\notin J$.
  Hence, $xy K= \varpi I_n\,K$ forces that $y_{j,1}=0$ for all $j>1$.
  Therefore,  the number of admissible $(x,y)$ in this case is equal to  that of admissible $(A,B)$'s, which is $\binom{n-1}{r-1}_q$ by induction hypothesis. The Lemma now follows immediately from the equality
 \[
 \binom{n}{r}_q=q^r\binom{n-1}{r}_q+\binom{n-1}{r-1}_q.\qedhere
 \]

 \end{proof}
 
\section{A determinant formula}\label{A:determinant}
In this appendix, we prove the following.
\begin{theorem}\label{T:determinant}
Let $\alpha_{1},\dots, \alpha_{n}$ be $n$ indeterminates. For $i=1,\dots, n$, let $s_{i}$ denote the $i$-th elementary symmetric polynomial in $\alpha$'s, and $s_0=1$ by convention. Let $q$ be another indeterminate. We put $q_r=q^{r-1}+q^{r-3}+\cdots+q^{1-r}$. Consider the matrix $M_n(q)=(m_{i,j})$ given as follows:
\[
m_{i,j}=\begin{cases}
\displaystyle\sum_{\delta=0}^{\min\{i-1,n-j\}}q_{n+i-j-2\delta}\;s_{j-i+\delta}\;s_{n-\delta}&\text{if }i\leq j;\\
\displaystyle\sum_{\delta=0}^{\min\{j-1,n-i\}}q_{n+j-i-2\delta}\;s_{\delta} \; s_{n+j-i+\delta} &\text{if }i>j.
\end{cases}
\]
Then we have
\[
\det(M_n(q))=\alpha_1\cdots \alpha_n\prod_{i\neq j}(q\alpha_i-\frac{1}{q}\alpha_j).
\]
\end{theorem}


\begin{proof}
Let $N_n(q)$ be the resultant matrix of the polynomials $f(x)=\prod_{i=1}^n(x+q^{-1}\alpha_i)$ and $g(x)=\prod_{i=1}^n(x+q\alpha_i)$, that is, $N_{n}(q)$ is the $2n\times 2n$ matrix given by
\[
N_n(q)=
\begin{pmatrix}
s_0 &q^{-1}s_1 & q^{-2}s_2 &\cdots & q^{1-n}s_{n-1} &q^{-n}s_n &0 & \cdots &0\\
0& s_0 &q^{-1} s_1 &\cdots & q^{2-n}s_{n-2} & q^{1-n}s_{n-1} &q^{-n}s_n &\cdots &0\\
\vdots &\vdots & \vdots &\ddots & \vdots &\vdots& \vdots &\ddots& \vdots\\
0 &0 &0&\cdots &s_0 & q^{-1}s_1 &q^{-2}s_2&\cdots &q^{-n}s_n\\
s_0 &qs_1& q^2s_2 &\cdots & q^{n-1}s_{n-1} &q^ns_n & 0 &\cdots & 0\\
0 &s_0 &qs_1 &\cdots &q^{n-2}s_{n-2} & q^{n-1}s_{n-1} &q^ns_n &\cdots &0\\
\vdots &\vdots &\vdots  &\ddots & \vdots &\vdots& \vdots &\ddots& \vdots\\
0 &0 &0&\cdots &s_0 & qs_1 &q^2s_2&\cdots &q^n s_n\\
\end{pmatrix}.
\]
It is well known that $\det(N_n(q))=\prod_{i,j}(-q^{-1}\alpha_i+q\alpha_j)$. Thus it suffices to show that $\det(N_n(q))=(q-q^{-1})^n\det(M_n(q))$.

We first make the following row operations on $N_n(q)$: subtract row $i$ from row $n+i$ for all $i=1,\dots, n$. We obtain a matrix whose first column are all $0$ expect the first entry being $1$; moreover, one can take out a factor $(q-q^{-1})$ from row $n+1, \dots, 2n$. Let $N_n'(q)$ be the right lower $(2n-1)\times (2n-1)$ submatrix of the remaining matrix. Then we have
\[
N_{n}'(q)=
\begin{pmatrix}
s_0 &q^{-1}s_1 & q^{-2}s_2 &\cdots & q^{1-n}s_{n-1} &q^{-n}s_n &0 & \cdots &0\\
0& s_0 &q^{-1} s_1 &\cdots & q^{2-n}s_{n-2} & q^{1-n}s_{n-1} &q^{-n}s_n &\cdots &0\\
\vdots &\vdots & \vdots &\ddots & \vdots &\vdots& \vdots &\ddots& \vdots\\
0 &0&0& \cdots &s_0 &q^{-1}s_1 &q^{-2}s_2&\cdots &q^{-n}s_n\\
q_1s_1 &q_2s_2 &q_3s_3 &\cdots & q_{n-1}s_{n-1} &q_ns_n &0 &\cdots &0\\
0 &q_1s_1 &q_2s_2 & \cdots & q_{n-2}s_{n-2} &q_{n-1}s_{n-1} &q_ns_n &\cdots &0\\
\vdots &\vdots & \vdots &\ddots & \vdots &\vdots& \vdots &\ddots& \vdots\\
0 &0 &0 &\cdots & 0 &q_1s_1 &q_2s_2 &\cdots & q_ns_n
\end{pmatrix}
\]
with $\det(N_n(q))=(q-q^{-1})^n\det(N_n'(q))$. Thus we are reduced to  proving that $\det(N_n'(q))=\det(M_n(q))$. Consider the $(2n-1)\times (2n-1)$ matrix $R=\begin{pmatrix}I_{n-1} &0\\ C& D\end{pmatrix}$ with the lower $n\times (2n-1)$ submatrix given by
 \[
 \begin{pmatrix}C&D\end{pmatrix}=\begin{pmatrix}
 -q_1s_1 &-q_2s_2 &\cdots &-q_{n-1}s_{n-1} &1 &q^{-1}s_{1} &q^{-2}s_{2} &\cdots &q^{2-n}s_{n-2}&q^{1-n}s_{n-1}\\
 0 &-q_1s_1 &\cdots& -q_{n-2}s_{n-2} &0  &1 & q^{-1}s_1&\cdots & q^{3-n}s_{n-3}&q^{2-n}s_{n-2}\\
 \vdots &\vdots &\ddots &\vdots &\vdots & \vdots &\vdots &\ddots  &\vdots &\vdots\\
 0 & 0&\cdots &-q_1s_1 &0  & 0& 0&\cdots &1&q^{-1}s_1\\
 0 &0 &\cdots &0 &0 &0&0 &\cdots &0 &1
 \end{pmatrix}.
 \]
By a careful computation, one verifies without difficulty that $RN_n'(q)=\begin{pmatrix}U& *\\
 0&M_{n}(q)\end{pmatrix}$, where $U$ is an $(n-1)\times (n-1)$-upper triangular matrix with all diagonal entries equal to $1$.
Note that $\det(R)=\det(D)=\det(U)=1$, it follows immediately  that $\det(N_n'(q))=\det(M_n(q))$.
\end{proof}

\end{document}